\documentclass[11pt]{article}
\usepackage{amssymb,amsmath,amsthm,color,array,fullpage}
\usepackage[notref,notcite,color,final]{showkeys}
\usepackage{ifpdf}
\ifpdf \usepackage[pdftex,pdfstartview=FitH,pdfpagemode=none,colorlinks,linkcolor=blue]{hyperref} \else  \usepackage[hypertex]{hyperref} \fi

\newtheorem{theorem}{Theorem}[section]
\newtheorem{lemma}[theorem]{Lemma}
\newtheorem{corollary}[theorem]{Corollary}

\newtheorem{definition}[theorem]{Definition}
\newtheorem{condition}[theorem]{Condition}
\newtheorem{notation}[theorem]{Notation}

\newtheorem{remark}[theorem]{Remark}
\newtheorem{proposition}[theorem]{Proposition}

\newcounter{subconst}[subsection]

\newcommand{\demosubconst}{\iota_1,\iota_2,\dots}
\newcommand{\declaresubconst}[1]{{\refstepcounter{subconst}\label{#1}}}

\newcounter{const}

\newcommand{\democonst}{\kappa_1,\kappa_2,\dots}
\newcommand{\declareconst}[1]{{\refstepcounter{const}\label{#1}}}

\newcounter{CONST}

\newcommand{\demoCONST}{c_1,c_2,\dots}
\newcommand{\declareCONST}[1]{{\refstepcounter{CONST}\label{#1}}}

\numberwithin{equation}{section}

\begin{document}

\title{Quantitative Density under Higher Rank Abelian Algebraic Toral Actions}
\author{Zhiren Wang\footnote{Department of Mathematics, Princeton University, Princeton, NJ 08544, USA; {\tt zhirenw@math.princeton.edu}}}

\date{}

\maketitle{\thispagestyle{empty}\begin{abstract}We generalize Bourgain-Lindenstrauss-Michel-Venkatesh's recent one-dimensional quantitative density result to abelian algebraic actions on higher dimensional tori. Up to finite index, the group actions that we study are conjugate to the action of $U_K$, the group of units of some non-CM number field $K$, on a compact quotient of $K\otimes_{\mathbb Q}\mathbb R$. In such a setting, we investigate how fast the orbit of a generic point can become dense in the torus. This effectivizes a special case of a theorem of Berend; and is deduced from a parallel measure-theoretical statement which effectivizes a special case of a result by Katok-Spatzier. In addition, we specify two numerical invariants of the group action that determine the quantitative behavior, which have number-theoretical significance.\end{abstract}\newpage}

{\thispagestyle{empty}\small\tableofcontents\newpage}

\section{Introduction}
\setcounter{page}{1}

\subsection{Background}

\hskip\parindent The rigidity of higher rank abelian algebraic actions has since long been studied in various forms. The first result of this type was achieved by Furstenberg's disjointness theory :
\begin{theorem}{\rm (Furstenberg \cite{F67}, '67)} Any minimal closed subset of $\mathbb T=\mathbb R/\mathbb Z$ simultaneously invariant under $\times 2$ and $\times 3$ is either $\mathbb T$ itself or a finite set of rational points.\end{theorem}

The measure-theorectical analogue of this theorem is the famous Furstenberg's conjecture, which asks whether any ergodic $\times 2, \times 3$-invariant measure on $\mathbb T$ is either uniform or finitely supported on rational points. The conjecture remains open.

There are several ways to extend Furstenberg's theorem. One of them is achieved by Berend, who proved an analogue on higher dimensional tori.

\begin{theorem}{\rm(Berend \cite{B83}, '83)} Let $\Sigma<M_d(\mathbb Z)=\mathrm{End}(\mathbb T^d)$ be an abelian subsemigroup such that: (i) for any common eigenspace $V$ of $\Sigma$, there is an element $g$ whose eigenvalue corresponding to $V$ has absolute value strictly greater than 1; (ii) $\mathrm{rank}(\Sigma)\geq 2$; (iii) $\Sigma$ contains a totally irreducible element. Then any minimal $\Sigma$-invariant closed set on $\mathbb T^d$ is either the full torus or finite.\end{theorem}

Here the action of a single toral endomorphism $A\in M_d(\mathbb Z)$ is {\it irreducible} if there is no non-trivial $A$-invariant subtorus on $\mathbb T^d=\mathbb R^d/\mathbb Z^d$; it is {\it totally irreducible} if $A^n$ is irreducible for every positive integer $n$.

Another way of extension was to go to the measure-theoretical category. Under the assumption of positive entropy, Furstenberg's conjecture has been proved by Rudolph  and Johnson following work of Lyons\cite{L88}.

\begin{theorem}{\rm(Rudolph\cite{R90},'90-Johnson\cite{J92},'92)} Suppose a Borel probability measure $\mu$ on $\mathbb T$ is invariant and ergodic under both $\times 2$ and $\times 3$. If the measure-theoretical entropy $h_\mu(\times 2)$ of the $\times 2$ action with respect to $\mu$ is positive, then $\mu$ is the Lebesgue measure. \end{theorem}

It is possible to pursue these two directions of extension at the same time. Namely, under certain conditions, for an abelian action on $\mathbb T^d$ with some kind of hyperbolicity, higher rank and irreducibility,  any ergodic invariant measure with positive entropy is expected to be uniform. This was first proved by Katok and Spatzier\cite{KS96} for abelian subgroups of $SL_d(\mathbb Z)$ under a special assumption called total non-symplecticity (TNS); for a detailed treatment of their result, see Kalinin-Katok\cite[\S 3]{KK01}. Later on, such a measure rigidity statement was proved  by Einsiedler-Lindenstrauss\cite{EL03} for a more general family of groups of toral automorphisms.

Further, there is a third way to extend these results: all the theorems mentioned above have quantitative versions. Recently \cite{BLMV08}, Bourgain, Lindenstrauss, Michel and Venkatesh effectivized Furstenberg's theorem.

\begin{theorem} {\rm(Bourgain-Lindenstrauss-Michel-Venkastesh \cite{BLMV08},'08)} For any pair of multiplicatively independent positive integers $a,b$:

(i). If $x\in\mathbb T$ is diophantine generic: $$|x-\frac pq|\geq q^{-k},\ \forall p,q\in\mathbb Z,\ q\geq 2,$$ then $\{a^mb^nx|0<m,n\leq N\}$ is $(\log\log Q)^{-c}$-dense in $\mathbb T$ for all $Q\geq Q_0$, where $c=c(a,b)$ and $Q_0=N_0(k,a,b)$ are constants.

(ii). If $x=\frac pQ$ where $Q$ is coprime with $ab$, then $\{a^mb^nx|0<m,n<3\log Q\}$ is $C(\log\log\log Q)^{-c}$-dense in $\mathbb T$ where $C=C(a,b)$ and $c$ is the same as in (i). \end{theorem}

Roughly speaking, this says the orbit of a point $x$ is of certain quantitative density unless $x$ is very close to a rational number with small denominator.

\subsection{Statement of main results}

\hskip\parindent In this paper we generalize Bourgain-Lindenstrauss-Michel-Venkatesh's theorems to the higher dimensional case. We investigate a special case of the situation studied by Berend \cite{B83}, namely the action on $\mathbb T^d$ by a group of toral automorphisms $G<SL_d(\mathbb Z)$ that satisfies: \begin{condition}\label{condG} $G$ is an abelian subgroup of $SL_d(\mathbb Z)$ such that:

(1). $\mathrm{rank}(G)\geq 2$, where $\mathrm{rank}$ refers to the torsion free rank of finitely generated abelian groups.;

(2). $G$ contains an totally irreducible toral automorphism;

(3). $G$ is maximal in rank: there is no intermediate abelian subgroup $G_1$ in $SL_d(\mathbb Z)$ containing $G$ such that $\mathrm{rank}(G)<\mathrm{rank}(G_1)$ .\end{condition}

We now equips $G$ with a norm.

\begin{definition}\label{mahler} (i). For a matrix $A\in M_d(\mathbb R)$, we define the {\bf logarithmic Mahler measure} of $A$ to be $$m(A):=\frac1{2\pi}\int_0^{2\pi}\log|\det(A-e^{i\theta}\mathrm{id})|\mathrm d\theta.$$ An alternative definition is \begin{equation}\label{mahlerdef}m(A)=\sum_{i=1}^d\log_+|\zeta_A^i|,\end{equation} where $\zeta_A^1,\cdots,\zeta_A^d$ are the eigenvalues of $A$ and $\log_+x=\max(0,\log x)$.

(ii). For a subgroup $G<SL_d(\mathbb Z)$ and $L>0$, let $B_G^\mathrm{Mah}(L)$ be the ball of radius $L$ with respect to logarithmic Mahler measure: $$B_G^\mathrm{Mah}(L):=\{g\in G|m(g)\leq L\}.$$\end{definition}

For more information on Mahler measures, we refer to \cite{EW99}.

Here is our main result in the topological category:

\declareCONST{efftopoepsilonCONST}  \declareCONST{efftopoalphaCONST}  \declareCONST{efftopometricballCONST} \declareCONST{efftopodensityCONST} 
\begin{theorem}\label{efftopo} If an abelian subgroup $G<SL_d(\mathbb Z)$ satisfies Condition \ref{condG} then there are effective constants $\ref{efftopoepsilonCONST} $, $\ref{efftopoalphaCONST} $, $\ref{efftopometricballCONST} $, and $\ref{efftopodensityCONST} $ depending only on $G$ such that if a finite set $E\subset\mathbb T^d$ is $\epsilon$-separated and $|E|\geq\epsilon^{-\alpha d}$ for some $\alpha,\epsilon>0$ with $\epsilon\leq\ref{efftopoepsilonCONST} $, $\alpha\geq \ref{efftopoalphaCONST} \frac{\log\log\log\frac1\epsilon}{\log\log\frac1\epsilon}$ then $B_G^\mathrm{Mah}(\ref{efftopometricballCONST} \log\frac1\epsilon).E:=\{g.x| g\in B_G^\mathrm{Mah}(\ref{efftopometricballCONST} \log\frac1\epsilon), x\in E\}$ is $(\log\frac1\epsilon)^{-\ref{efftopodensityCONST} \alpha}$-dense.\end{theorem}

The proof of Theorem \ref{efftopo} is based on the analogous effective measure-theoretical Theorem \ref{effmeas} below which studies the behaviour of a measure $\mu$ under the $G$-action. We assume $\mu$ has positive entropy up to a certain scale.

\begin{condition}\label{condmu} $\mu$ is a Borel measure on $\mathbb T^d$ such that:

For some given pair $\alpha>0$, $\epsilon>0$, the entropy $H_\mu(\mathcal P)=\sum_{P\in\mathcal P}-\mu(P)\log\mu(P)$ is at least $\alpha d\log\frac1\epsilon$ for all measurable partitions $\mathcal P$ of $\mathbb T^d$ such that $\mathrm{diam}\mathcal P\leq\epsilon$, where $\mathrm{diam}\mathcal P$ is the maximal diameter of atoms from $\mathcal P$. \end{condition}

\declareCONST{effmeasdeltaCONST} \declareCONST{effmeascoeffCONST} 
\begin{theorem}\label{effmeas} Suppose an abelian subgroup $G<SL_d(\mathbb Z)$ meets Condition \ref{condG}. There are effective constants $\ref{efftopoepsilonCONST} $, $\ref{efftopometricballCONST} $, $\ref{effmeasdeltaCONST} $, and $\ref{effmeascoeffCONST} $ depending only on $G$ such that if $\epsilon\leq \ref{efftopoepsilonCONST} $, $\delta\in[\ref{effmeasdeltaCONST} (\log\log\frac1\epsilon)^{-1},\frac\alpha2]$ and a Borel measure $\mu$ on $\mathbb T^d$ satisfies the entropy condition \ref{condmu} then there exists a measure $\mu'$ which has total mass $|\mu'|\geq\alpha-\delta$ and is dominated by some element $\mu''$ from the convex hull of $B_G^\mathrm{Mah}(\ref{efftopometricballCONST} \log\frac1\epsilon).\mu$ in the space of probability measures on $\mathbb T^d$ (here the group $G$ acts on $\mu$ by pushforward: $g.\mu=g_*\mu$), so that $\forall f\in\mathcal C^\infty(\mathbb T^d)$, $$\big|\mu'(f)-|\mu'|\int_{\mathbb T^d}f(x)\mathrm dx\big|\leq \ref{effmeascoeffCONST} (\log\frac1\epsilon)^{-\frac12\ref{effmeasdeltaCONST} ^{-1}\delta}\|f\|_{\dot H^{\frac{d+1}2}}.$$ \end{theorem}

Theorem \ref{effmeas} actually effectivizes a special case of Katok-Spatzier's result \cite{KS96}.

Another interesting corollary to Theorem \ref{efftopo} will be Theorem \ref{dioQ}, regarding how fast a generic single $G$-orbit can fill up $\mathbb T^d$, which is a quantitative form of the fact that any infinite $G$-orbit is dense (proved by Berend \cite{B83}) and the generalization of a similar theorem in \cite{BLMV08}.

\declareCONST{dioQsizeCONST} \declareCONST{dioQdensityCONST}
\begin{theorem}\label{dioQ} Suppose $G$ satisfies Condition \ref{condG} then there are effective constants $\ref{dioQsizeCONST}$, $\ref{dioQdensityCONST}$ depending only on $G$ such that:

For all $Q\in\mathbb N$, $Q\geq \ref{dioQsizeCONST}$, if a point $x\in\mathbb T^d=\mathbb R^d/\mathbb Z^d$ satisfies one of the following conditions:

(i) $x$ is diophantine generic: $\exists k>1$ such that $|x-\frac{v}{q}|\geq q^{-k}$ for any coprime pair $(v, q)$ where $q\in\mathbb N$ and $v\in\mathbb Z^d$; OR

(ii) $x=\frac vQ$ where $v\in\mathbb Z^d$ is coprime with $Q$, in which case we denote $k=1$,

then the set $B_G^{\mathrm{Mah}}\big((k+2)\log Q\big).x$ is $(\log\log\log Q)^{-\ref{dioQdensityCONST}}$-dense.\end{theorem}

\begin{remark}\label{numbertheorydependence} In fact, in all the theorems stated above, we are able to know on which features of $G$ the constants really depend and how the dependence looks like. Actually, all constants are determined by the dimension $d$ and two algebraic invariants $\mathcal M_\psi$ and $\mathcal F_{\phi(G)}$ of $G$. For the explicit expressions, see Propositions \ref{effmeas'}, \ref{efftopo'}, \ref{largetopo'}, \ref{dioQ'}.

The exact definitions of $\mathcal M_\psi$ and $\mathcal F_{\phi(G)}$ are going to appear later in \S\ref{field}. However, essentially $\mathcal M_\psi$ is a measurement of how ``twisted'' the eigenbasis of $G$ is (notice $G$ is commutative so all its elements share a common eigenbasis); and $\mathcal F_{\phi(G)}$ is the bound on a set of generators of $G$ (up to finite-index) in terms of logarithmic Mahler measure. 

We make an effort to track the dependence on $\mathcal M_\psi$ and $\mathcal F_{\phi(G)}$ in this paper, the reason is that this dependence can be interesting from a number-theoretical point of view (cf. Appendix \S\ref{minmaxapp} for example). \end{remark}

\begin{remark}\label{wordnorm} It should be remarked that if $\|\cdot\|_\mathrm{WL}$ is the word length metric with respect to some fixed generating set $S\subset G$, then $m(g)\lesssim_{G,S}\|g\|_\mathrm{WL}+1,\forall g\in G$. Therefore in all the theorems stated above, the ball $B_G^\mathrm{Mah}(\ref{efftopometricballCONST} \log\frac1\epsilon)$ with respect to Mahler measure can be replaced by some ball $B_G^\mathrm{WL}(\ref{efftopometricballCONST} '\log\frac1\epsilon)$ defined in terms of word length metric, where $\ref{efftopometricballCONST} '=\ref{efftopometricballCONST} '(G,S)$ is effective.\end{remark}

A major restricion we adopted in addition to Berend's conditions is that $G$ is supposed to be maximal in rank, which guarantees that one can expand an arbitrary eigenspace while contracting everything else. The study of the case without this assumption is also possible, though more careful arguments are required and only weaker estimates will be obtained, and is hopefully going to appear in another paper \cite{LW0001}.

\subsection{Organization of paper and notations}

\hskip\parindent While our proof follows the structure of that of Bourgain et. al. which uses Fourier analysis, we also borrow several arguments from \cite{B83}. A little bit of number theory is needed as our group $G$ turns out to be closely linked to the full unit group of some number field. There are some new difficulties to be treated in the higher dimensional case. The main point is that instead of the case on a one dimensional circle, we are going to decompose a multidimensional torus into eigenspaces and work on one of them. We will show some kind of ``dense'' distribution along a line in that eigenspace and then make use of the irreduciblity of the action (so any eigenspace does not form a close subtorus). One of the difficulties here is that when the eigenspace is complex (i.e. 2-dimensional), a line in it may not be equidistributed in $\mathbb T^d$, this is going to be dealt with in \S \ref{escapeline}

The organization of paper is as follows:

\S\ref{field} discusses the number-theoretical implication of Condition \ref{condG} and proves the $G$-action comes from the group of units of a certain number field and its eigenspaces correspond to the real and complex embeddings of that field. Further, in \S\ref{actionUK} we discuss a few algebraic properties of this action, in particular the irrationality of the eigenspaces is specified quantitatively.

In \S\ref{vecinVi} we give an effective description of the fact that if a measure on $\mathbb T^d$ has positive entropy then its projection to at least one of the eigenspaces $V_i$ has positive entropy. Since the restriction of the group action on $V_i$ is just a multiplication, we are able to discuss its behaviors in \S\ref{GVi}. Roughly speaking, we are going to show that using the group action, it is possible to stretch a short vector, to generate an approximation of a line segment (or an arithmetic progression) from a given vector, as well as to relocate an arbitrary line to a general position if the eigenspace is a complex one.

Proof of Theorem \ref{effmeas} is given in \S\ref{measresults}. The main tool used there is Fourier analysis. And the underlying geometric idea is that the positive entropy in $V_i$ guarantees that the difference vector between a random pair of nearby points taken with respect to $\mu$ is not too short, and by applying the results obtained in \S\ref{GVi} to this vector we can create a sufficiently long line segment placed in a general position, which is equidistributed on the torus.

In \S\ref{toporesults} Theorem \ref{efftopo} is deduced from \ref{effmeas} by taking a test function, we then prove its corollaries. The appendix \S\ref{minmaxapp} gives an application of results in \S\ref{toporesults} to a number theoretical problem, following a strategy observed by Cerri \cite{C07}.

\begin{notation} Throughout this paper,

\begin{itemize}
\item $\log$ stands for the logarithm of base $2$ and the natural logarithm is denoted by $\ln$.
\item For a linear map $f$, $\|f\|:=\sup_{|x|=1}|f(x)|$ is its norm as an operator.
\item If $A$ and $B$ are subsets of some additive group, let $A-B$ (resp. $A+B$) denote the set $\{a-b\text{ (resp. } a+b \text{) }|a\in A, b\in B\}$.
\item The symbol $\mathrm{rank}$ refers to the torsion-free rank of a finitely generated abelian group.
\item We denote constants by $\demoCONST$ and less important ones by $\democonst$. The constants only locally used by a proof are written as $\demosubconst$.  All constants in this paper are going to be positive and effective: i.e. an explicit value for it can be deduced from the information already known if one really want to. When a constant first appears, we may write it as a function to signify all the variables that it depends on, for example if we write $c_3(N,d)$, it means that the constant $c_3$ depends only on $N$ and $d$, and so forth.
\item We write $A\lesssim B$ (or $B\gtrsim A$) for the estimate $\exists c>0, A\leq cB$. Moreover, we always include all factors that the implied constant $c$ depends on within the $\lesssim$ symbol as subscripts. For example, if $c$ depends on and only on $d$ then we will always write $A\lesssim_d B$; the inequality $A\lesssim B$ without any subscript means the implied constant is absolute.
\end{itemize}
\end{notation}
\noindent{\bf Acknowledgments:} This paper is part of my Ph.D. thesis work. I am grateful to my advisor Prof. Elon Lindenstrauss for introducing me to the subject and guiding me through the research. I also would like to thank Prof. Jean Bourgain for comments and encouragement.

\section{Abelian algebraic actions on the torus}\label{field}

\hskip\parindent In this section, we are going to impose an alternative condition on $G$ to replace Condition \ref{condG}.

\subsection{Preliminaries on number fields}\label{prelimfield}

\hskip\parindent Consider now a degree $d$ number field $K$, $d\geq 3$. $K\otimes_{\mathbb Q}{\mathbb R}\cong\mathbb R^d$. Any element of $K$ acts on this space naturally by multiplication: $\times_s.(t\otimes x)=st\otimes x, \forall s,t\in K, x\in\mathbb R$. Recall that if $K$ has $r_1$ real embeddings $\sigma_1,\cdots,\sigma_{r_1}$ and $r_2$ pairs of conjugate imaginary embeddings $\sigma_{r_1+j},\sigma_{r_1+r_2+j},j=1,\cdots,r_2$, then there is an isomorphism $\sigma:K\otimes_{\mathbb Q}\mathbb R\mapsto\mathbb R^{r_1}\times\mathbb C^{r_2}\cong\mathbb R^d$ where \begin{equation}\label{KotimesR}\begin{split}\sigma(t\otimes x) =&x\cdot\big(\sigma_1(t),\cdots,\sigma_{r_1}(t),\mathrm{Re}\sigma_{r_1+1}(t),\cdots,\mathrm{Re}\sigma_{r_1+r_2}(t),\\&\hskip1cm \mathrm{Im}\sigma_{r_1+1}(t),\cdots,\mathrm{Im}\sigma_{r_1+r_2}(t)\big)\end{split}\end{equation} for all $x\in\mathbb R$, $t\in K$. With this identification, the action of $K$ on $\mathbb R^d\cong\mathbb R^{r_1}\times\mathbb C^{r_2}$ is easy to describe: $\times_s$ acts on $K\otimes_{\mathbb Q}{\mathbb R}$ via the linear mapping $\label{canonemb}(\times_{\sigma_1(s)},\cdots,\times_{\sigma_{r_1}(s)},\times_{\sigma_{r_1+1}(s)},\cdots,\times_{\sigma_{r_1+r_2}(s)}),$ where the first $r_1$ multiplications are on the $r_1$ real subspaces and the last $r_2$ ones are on the complex ones. Here we view the $j$-th copy of $\mathbb C$ as a two-dimensional real subspace. In addition, let it be spanned by the $(r_1+j)$-th (real part) and the $(r_1+r_2+j)$-th (imaginary part) coordinates: a number $z$ in this copy of $\mathbb C$ is identified with $(\mathrm{Re}z)e_{r_1+j}+(\mathrm{Im}z)e_{r_1+r_2+j}$, then the multiplicative action of $s$ on this embedded copy of $\mathbb C$ is given by the matrix $\left(\begin{array}{cc}\mathrm{Re}\sigma_{r_1+j}(s)&-\mathrm{Im}\sigma_{r_1+j}(s)\\\mathrm{Im}\sigma_{r_1+j}(s)&\mathrm{Re}\sigma_{r_1+j}(s)\end{array}\right)$.

Let $\mathcal O_K$ be the ring of integers and $U_K=\mathcal O_K^*$ be the group of units of $K$. Dirichlet's Unit Theorem claims that, modulo the torsion part $\mathcal T_k$ of $U_k$, which is a finite set of roots of unity, the group morphism \begin{equation}\label{Logemb}\mathcal L:t\rightarrow\big(\log|\sigma_1(t)|,\log|\sigma_2(t)|,\cdots,\log|\sigma_{r_1+r_2}(t)|\big)\end{equation} embeds $U_K$ as a lattice in the $(r_1+r_2-1)$-dimensional subspace \begin{equation}\label{Logembspace}W=\{(w_1,w_2,\cdots,w_{r_1+r_2})|\sum_{j=1}^{r_1}w_j+2\sum_{j=1}^{r_2}w_{r_1+j}=0\}\subset\mathbb R^{r_1+r_2}.\end{equation}

Let $r=r_1+r_2-1=\mathrm{rank}(U_K)$ and $d_i=\left\{\begin{array}{ll}1,\ &1\leq i\leq r_1;\\2,& r_1+1\leq i\leq r_1+r_2.\end{array}\right.$ then $d=\sum_1^{r_1+r_2}d_i$.

\begin{remark}\label{rankchoice} For a given $d$, there are only finitely many possible values for $r$ as $\frac d2-1=\frac{r_1}2+r_2-1\leq r=r_1+r_2-1\leq r_1+2r_2-1=d-1$.\end{remark}

For an element in $K$, its logarithmic Mahler measure measures the size of its image under $\mathcal L$.

\begin{definition}\label{logh}For a non-zero element $t$ from a number field $K$ of degree $d$, the {\bf logarithmic Mahler measure} of $t$ is $$h^\mathrm{Mah}(t)=\sum_{\nu}d_\nu\log_+|t|_{\nu},$$ where $\log_+x=\max(\log x,0)$, $\nu$ runs over all finite or infinite places of $K$ and $d_\nu=[K_\nu:\mathbb Q_{\tilde \nu}], \nu|\tilde\nu$ is the local degree of $\nu$.

$h(t)=\dfrac{h^\mathrm{Mah}(t)}d$ is the {\bf absolute logarithm height} of $t$, it is determined by the algebraic number $t$ and does not depend on the field $K$ in which $t$ is regarded as an element.\end{definition}

In particular, if $t\in U_K$ then all its non-Archimedean absolute values are equal to $1$, so the logarithmic Mahler measure involves only Archimedean embeddings:

\begin{equation}\label{unitheight}h^\mathrm{Mah}(t)=\sum_{i=1}^{r_1+r_2}d_i\log_+|\sigma_i(t)|=\frac12\sum_{i=1}^{r_1+r_2}d_i\big|\log|\sigma_i(t)|\big|=\frac12\sum_{i=1}^d|\log\sigma_i(t)|\end{equation} because $\sum_{i=1}^{r_1+r_2}d_i\log|\sigma_i(t)|=\log|\prod_{i=1}^d\sigma_i(t)|=0$.

It is easy to see \begin{equation}\label{heightrules}h^\mathrm{Mah}(tt')\leq h^\mathrm{Mah}(t)+h^\mathrm{Mah}(t');\ h^\mathrm{Mah}(t^n)=|n|h^\mathrm{Mah}(t),\forall n\in\mathbb Z.\end{equation} 

\begin{definition}\label{funduni} Let $U$ be a finite index subgroup in $U_K$. Define the {\bf size of (virtually) fundamental units} to be $$\mathcal F_U:=\inf\Big\{\mathcal F|\{h^\mathrm{Mah}(t)\leq\mathcal F, t\in U\}\text{ generates a finite-index subgroup of }U\Big\}.$$\end{definition}

Notice as $U$ is a discrete set, $\mathcal F_U$ can be achieved, i.e. there are $r$ elements $t_1,\cdots,t_r$ that virtually generate the torsion-free part of $U$, such that $\max_kh^\mathrm{Mah}(t_k)=\mathcal F_U$. 

\declareconst{voutierconst}
\begin{remark}\label{voutier} It is known that for all algebraic numbers $t$ of degree $d$, $h^\mathrm{Mah}(t)$ has an effective positive lower bound $\ref{voutierconst}(d)$ depending only on $d$ as long as $t$ is not a root of unity (see Voutier \cite{V96}). By definition this is also a lower bound for $\mathcal F_U$.\end{remark}

We are interested in lattices embedded by $\sigma$ into $K\otimes_\mathbb Q\mathbb R$, the covolumes and shapes of such lattices will be important for us. 

\begin{definition} \label{isouni}(1). The {\bf scale} of a lattice $\Lambda$ in a $d$-dimensional vector space $V$ is $\mathcal S_\Lambda=\sqrt[d]{\mathrm{covol}(\Lambda)}$.

(2). The {\bf uniformity} of a linear isomorphism $\Psi:\mathbb R^d\mapsto\mathbb R^d$ is \begin{align*}\mathcal M_\Psi:=&\max(\|\Psi\|\mathcal S_{\Psi(\mathbb Z^d)}^{-1},\|\Psi^{-1}\|\mathcal S_{\Psi(\mathbb Z^d)})\\=&\max(\|\Psi\|\cdot|\det\Psi|^{-\frac1d},\|\Psi^{-1}\|\cdot|\det\Psi|^{\frac1d}).\end{align*}\end{definition}

\begin{definition} A number field $K$ is a {\bf CM-number field} if it has a proper subfield $F$ such that $\mathrm{rank}(U_K)=\mathrm{rank}(U_F)$.\end{definition}

Actually, it follows from Dirichlet's Unit Theorem that a number field is CM if and only if $K$ is a totally complex quadratic extension of a totally real field $F$ (see Parry \cite{P75}).

We show a property of non-CM fields that is going to be relevant later.

\begin{lemma}\label{maxreal} If $K$ is not CM and has a complex embedding $\sigma_{r_1+j}, 1\leq j\leq r_2$, let $F$ be the proper subfield $\sigma_{r_1+j}^{-1}(\mathbb R)$. If $F\neq\mathbb Q$ then $\exists k,l\in\{1,\cdots,r_1+r_2\}\backslash\{r_1+j\}$ such that $k\neq l$ but $\forall t\in F$, $|\sigma_k(t)|=|\sigma_l(t)|$.\end{lemma}

\begin{proof} Let $d'=[K:F]\geq 2$, then each embedding of $F$ extends to exactly $d'$ different embeddings of $K$. In other words, the set of embeddings $\{\sigma_1, \cdots, \sigma_d\}$ of $K$ can be divided into $d'$-tuples, each corresponds to one embedding of $F$; the number of these tuples is $[F:\mathbb Q]=\frac d{d'}$.

If $d'=2$, each tuple is a pair. As $\sigma_{r_1+j}$ and its complex conugate $\sigma_{r_1+r_2+j}$ has the same restriction on the maximal real subfield $F=\sigma_{r_1+j}^{-1}(\mathbb R)$, they form one of the pairs. So any other pair does not contain $\sigma_{r_1+j}$ or $\sigma_{r_1+r_2+j}$. Suppose every other pair consists of two conjugate complex embeddings $\sigma$ and $\bar\sigma$ then as they have the same restriction on $F$, $F$ lies in their common real part. So all embeddings of $F$ is real and extends to two conjugate complex embeddings of $K$, which contradicts the assumption that $K$ is not CM. Hence there is a pair made of two embeddings which are not complex conjugates to each other.

If $d'\geq 3$, then one of the tuple contains  $\sigma_{r_1+j}$ and $\sigma_{r_1+r_2+j}$. As $F\neq\mathbb Q$, there is at least one other $d'$-tuple. As $d'\geq 3$, there are two embeddings in that tuple that are not complex conjugate to each other.

So in any case there are two different embeddings $\sigma$ and $\sigma'$ of $K$, which are not $\sigma_{r_1+j}$ or $\sigma_{r_1+r_2+j}$ and are not complex conjugate to each other, such that $\sigma|_F=\sigma'|_F$. If one or both of them is not in $\sigma_1,\cdots,\sigma_{r_1+r_2}$ then we replace it by its complex conjugate, which does not change the absolute value. This completes the proof.\end{proof}

\begin{condition}\label{condG'}$G$ is an abelian subgroup of $SL_d(\mathbb Z)$ such that there are:

(1). a non-CM number field $K$ of degree $d$ whose group of units $U_K$ has rank at least $2$;

(2). an isomorphim $\phi$ from $G$ to a finite-index subgroup of $U_K$;

(3). a cocompact lattice $\Gamma$ in $K\otimes_{\mathbb Q}{\mathbb R}$ which sits in $K<K\otimes_\mathbb Q\mathbb R$ and is invariant under the natural action of $\phi(G)$,

such that by identifying $G$ with $\phi(G)$, the action of $G$ on $\mathbb T^d$ is conjugate to the natural $G$-action on $X=(K\otimes_{\mathbb Q}{\mathbb R})/\Gamma$, i.e., there exists a linear isomorphism $\psi:\mathbb R^d\mapsto K\otimes_\mathbb Q\mathbb R\cong\mathbb R^d$ such that $\psi(\mathbb Z^d)=\Gamma$ and $\psi^{-1}\circ\times_{\phi(g)}\circ\psi=g,\forall g\in G$.\end{condition}

\begin{remark} $\mathrm{rank}(U_K)=r_1+r_2-1\geq 2$ implies $d=r_1+2r_2\geq 3$.\end{remark}

A consequence to Condition \ref{condG'} is that the logarithmic Mahler measures defined respectly by Definition \ref{mahler} on $G$ and by Definition \ref{logh} on $\phi(G)$ are identified via $\phi$.

\begin{lemma}\label{mahlerheight} If $G$ satisfies Condition \ref{condG'}, then $\forall g\in G$, the logarithmic Mahler measure $m(g)=h^\mathrm{Mah}(\phi(g))$.\end{lemma}
\begin{proof} The eigenvalues of $g$ are exactly those of $\times_{\phi(g)}$, namely $\sigma_1(\phi(g))$, $\cdots$, $\sigma_d(\phi(g))$. By (\ref{mahlerdef}), $m(g)=\sum_i\log_+\sigma_i(\phi(g))=h^\mathrm{Mah}(\phi(g))$.\end{proof}

\begin{remark} If a generating set $S$ of $G$ is fixed, by Dirichlet's Unit Theorem it is not difficult to see that the word metric $\|g\|^\mathrm{WL}$ of an arbitrary element $g$ with respect to $S$ has some upper bound $\|g\|^\mathrm{WL}\lesssim_{G,S}h^\mathrm{Mah}(\phi(g))+1$, from which Remark \ref{wordnorm} follows.\end{remark}

It turns out Condition \ref{condG'} is a good substitute for Condition \ref{condG}.

\begin{theorem}\label{condGfield}Condition \ref{condG} implies Condition \ref{condG'}.\end{theorem}
\begin{proof}The theorem decomposes into Propositions \ref{KphiGamma}, \ref{finiteidx} and \ref{subfield}, which are proved below.\end{proof}

In fact, it can be shown that the two conditions are equivalent.

\subsection{Construction of the number field}

\hskip\parindent In the followsing three subsections we are going to prove Theorem \ref{condGfield}. The following statement is a special case of a known theorem (see for example \cite[Prop. 2.1]{EL03}, , for a proof of the general statement see \cite{S95}).
\begin{proposition}\label{KphiGamma} Suppose an abelian subgroup $G$ of $SL_d(\mathbb Z)$ has an element whose action is irreducible, then there exist

(i). a number field $K$ of degree $d$;

(ii). an isomorphim $\phi$ from $G$ to a subgroup of the group of units $U_K$;

(iii). a rank $d$ lattice $\Gamma$ in $K<K\otimes_\mathbb Q\mathbb R$ which is invariant under the natural action of $\phi(G)$,

such that there exists a linear isomorphism $\psi:\mathbb R^d\mapsto K\otimes_\mathbb Q\mathbb R\cong\mathbb R^d$ such that $\psi(\mathbb Z^d)=\Gamma$ and $\psi^{-1}\circ\times_{\phi(g)}\circ\psi=g,\forall g\in G$;\end{proposition}
\begin{proof}[Proof of Proposition] Assume $g\in G$ acts irreducibly, then the characteristic polynomial of $g$ is irreducible over $\mathbb Q$; otherwise the rational canonical form of $g$ over $\mathbb Q$ has more than one block, giving a non-trivial $g$-invariant rational subspace of $\mathbb R^d$, which projects to a non-trivial $g$-invariant subtorus in $\mathbb T^d$. Therefore $g$ has $d$ distinct eigenvalues $\zeta_g^1,\cdots,\zeta_g^d$ which are algebraic conjugates to each other, where $\zeta_g^1,\cdots,\zeta_g^{r_1}$ are real and the rest are $r_2$ imaginary pairs, $r_1+2r_2=d$. $\zeta_g^{r_1+j}=\overline{\zeta_g^{r_1+r_2+j}},\forall j=1,\cdots,r_2$.

Construct number field $K_i=\mathbb Q(\zeta_g^i)$, so $\mathrm{deg}K_i=d$.  Denote $K=K_1$, then $K$ has $d$ embeddings: $\sigma_i(\zeta_g^1)=\zeta_g^i, \sigma_i(K)=K_i, \forall i=1,\cdots,d$. The first $r_1$ embeddings are real;  $\sigma_{r_1+j}$ and $\sigma_{r_1+r_2+j}$ are complex conjugates for $1\leq j\leq r_2$.

An eigenvector $v_1$ of $\zeta_g^i$ lies in $\mathrm{Ker}(g-\zeta_g^1\mathrm{id})$. As all entries of $g-\zeta_g^1\mathrm{id}$ belongs to the field $K$, we can fix an eigenvector $v_1\in K^d$.

Let $v_i=\sigma_i(v_1)\in{K_i}^d$, then as $g\in SL_d(\mathbb Z)$ is fixed by $\sigma_i$, $gv_i=\sigma_i(g)\sigma_i(v_1)=\sigma_i(\zeta_g^1v_1)=\sigma_i(\zeta_g^1)\sigma_i(v_1)=\zeta_g^iv_i$. So $v_i$ is an eigenvector corresponding to the eigenvalue $\zeta_g^i$ for all $1\leq i\leq d$ and it follows that $\mathbb C^n=\oplus_{i=1}^d\mathbb Cv_i$. Notice $v_i\in \sigma_i(K^d)=K_i^d\subset\mathbb C^d$, in particular it is a real vector for $1\leq i\leq r_1$.

As $g$ is irreducible all eigenspaces are one dimensional over $\mathbb C$, thus by commutativity the basis $\{v_i\}_{i=1}^d$ diagonalizes not only $g$ but any element $h\in G$ as well: $hv_i=\zeta_h^iv_i,\forall h\in G, \forall i$.

We claim $\zeta_h^i\in K_i, \forall h\in G$. This is because an arbitrary element $t$ in the field $\mathbb Q(\zeta_g^i,\zeta_h^i)$ can be written as $p(\zeta_g^i,\zeta_h^i)$ where $p$ is a rational polynomial. As $gh=hg$, $p(g,h)v_i=p(\zeta_g^i,\zeta_h^i)v_i=tv_i$, so $t$ is an eigenvalue of the rational $d\times d$ matrix $p(g,h)$, thus an algebraic number of degree at most $d$. By choosing $t$ to be a generating
element of $\mathbb Q(\zeta_g^i,\zeta_h^i)=K_i(\zeta_h^i)$, $\mathrm{deg}K_i(\zeta_h^i)\leq d=\mathrm{deg}K_i$. Thus $\zeta_h^i\in K_i$.

Furthermore, $\sigma_i(\zeta_h^1)=\zeta_h^i$ for any $h\in G$. Actually, as $\zeta_h^1\in K$, $\zeta_h^1=f(\zeta_g^1)$ for some polynomial $f\in\mathbb Q[x]$. $f(g)v_1=f(\zeta_g^1)=\zeta_h^1v_1=hv_1$. If $f(g) \neq h$ then $0<\mathrm{rank}(f(g)-h)<d$, $\mathrm{Ker}(f(g)-h)$ is a non-trivial rational subspace of $\mathbb R^d$. Since $g$ commutes with $f(g)-h$, this rational subspace is $g$-invariant, which is prohibited to happen as $g$ acts irreducibly. So $h=f(g)$, in particular $\sigma_i(\zeta_h^1)=\sigma_i(f(\zeta_g^1))=f(\sigma_i(\zeta_g^1))=f(\zeta_g^i)=\zeta_h^i$.

Let $\phi(h)=\zeta_h^1$, which is an algebraic integer as $h$ is an integer matrix. Moreover $\zeta_h^1\in U_K$ because $h$ is invertible. Hence $\phi$ is a group morphism whose image lies in $U_K$. $\phi$ is injective because if $\phi(h)=1$ for some $h\neq\mathrm{Id}$ then $\mathrm{Ker}(h-\mathrm{Id})$ is a non-trivial rational subspace invariant under $g$ since $g$ commutes with $h$,  which again contradicts the irreducibility of $g$. Therefore $\phi:G\mapsto\phi(G)\subset U_K$ is injective.

For $1\leq i\leq r_1$ let $w_i=v_i\in\mathbb R^d$; for $1\leq j\leq r_2$, denote $w_{r_1+j}=2\mathrm{Re}v_{r_1+j}=v_{r_1+j}+v_{r_1+r_2+j}$ and $w_{r_1+r_2+j}=-2\mathrm{Im}v_{r_1+j}={\mathrm i}(v_{r_1+j}-v_{r_1+r_2+j})$. Then the  real vectors $w_1,w_2,\cdots,w_d$ are linearly independent and thus form a basis of $\mathbb R^d\subset\mathbb C^d$.

Let $\psi$ be the isomorphism  from $\mathbb R^d\subset\mathbb C^d$ to $K\oplus_{\mathbb Q}\mathbb R\cong\mathbb R^d$ which maps $w_i$ to $e_i$ for all $1\leq i\leq d$, where $e_i$ is the $i$-th coordinate vector in $K\otimes_{\mathbb Q}\mathbb R$ which is represesented as in (\ref{KotimesR}).

For any $h\in G$ acting on $\mathbb R^d$, it rescales $w_i$ by $\zeta_h^i=\sigma_i(\phi(h))$ when $1\leq i\leq r_1$ and acts on the two-dimensional subspace generated by $w_{r_1+j}$, $w_{r_1+r_2+j}$ as the matrix $\left(\begin{array}{ll}\mathrm{Re}\sigma_{r_1+j}(\phi(h))&-\mathrm{Im}\sigma_{r_1+j}(\phi(h))\\\mathrm{Im}\sigma_{r_1+j}(\phi(h))&\mathrm{Re}\sigma_{r_1+j}(\phi(h))\end{array}\right)$. Compare this action with that of $\times_{\phi(h)}$ on $K\otimes_{\mathbb Q}\mathbb R$ which was described earlier, we see they are conjugate to each other via $\psi$.

Therefore the only fact remaining to show is that the lattice $\Gamma=\psi(\mathbb Z^d)$ is in $K\subset K\otimes_{\mathbb Q}\mathbb R$, which is equivalent to $\psi(\mathbb Q^d)=K$ as $K$ is a $d$-dimensional $\mathbb Q$-vector space in $K\otimes_{\mathbb Q}\mathbb R$. It suffices to prove  $\psi^{-1}(t)\in\mathbb Q^d$, $\forall t\in K$. Actually by description (\ref{KotimesR}), $t$ is identified with $$\sum_{i=1}^{r_1}\sigma_i(t)e_i+\sum_{j=1}^{r_2}(\mathrm{Re}\sigma_{r_1+j}(t)\cdot e_{r_1+j}+\mathrm{Im}\sigma_{r_1+j}(t)\cdot e_{r_1+r_2+j}).$$ Using $\psi^{-1}(e_i)=w_i$, we obtain \begin{align*}
\psi^{-1}(t)=&\sum_{i=1}^{r_1}\sigma_i(t)w_i+\sum_{j=1}^{r_2}(\mathrm{Re}\sigma_{r_1+j}(t)\cdot w_{r_1+j}+\mathrm{Im}\sigma_{r_1+j}(t)\cdot w_{r_1+r_2+j})\\
=&\sum_{i=1}^{r_1}\sigma_i(t)v_i+\sum_{j=1}^{r_2}\big(\mathrm{Re}\sigma_{r_1+j}(t) (v_{r_1+j}+v_{r_1+r_2+j})\\
&\hskip4cm+\mathrm{Im}\sigma_{r_1+j}(t)\cdot\mathrm{i}(v_{r_1+j}-v_{r_1+r_2+j})\big)\\
=&\sum_{i=1}^{r_1}\sigma_i(t)v_i+\sum_{j=1}^{r_2}(\sigma_{r_1+j}(t)v_{r_1+j}+\overline{\sigma_{r_1+j}}(t)v_{r_1+r_2+j}).\end{align*}
Since $\overline{\sigma_{r_1+j}}$ is just $\sigma_{r_1+r_2+j}$, we see  
$$\psi^{-1}(t)=\sum_{i=1}^d\sigma_i(t)v_i=\sum_{i=1}^d\sigma_i(t)\sigma_i(v_1)=\sum_{i=1}^d\sigma_i(tv_1).$$ 
As $t\in K$ and $v_1\in K^d$, each coordinate in the vector $\sum_{i=1}^d\sigma_i(tv_1)$ is the trace of the corresponding coordinate in $tv_1\in K^d$, which is rational. This completes the proof.\end{proof}

\subsection{Maximal rank and the full group of units}

\begin{proposition}\label{finiteidx} In Proposition \ref{KphiGamma}, if there doesn't exist any abelian subgroup $G_1$ of $SL_d(\mathbb Z)$ containing $G$ such that $\mathrm{rank}(G)<\mathrm{rank}(G_1)$, then $\phi(G)$ is a subgroup of finite index inside $U_K$.\end{proposition}

\begin{proof} Recall that in the proof of Proposition \ref{KphiGamma}, all elements in $G$ are diagonalized as a complex matrix with respect to basis $v_1,\cdots,v_d$. For any $\theta\in K$, construct a matrix $\gamma(\theta)$ which can be diagonalized in the same basis, such that $\gamma(\theta)v_i=\sigma_i(\theta)v_i, \forall i=1,\cdots, d$, where $\sigma_i:K\mapsto K_i$ is the $i$-th embedding of $K$. It is clear that $\gamma$ is an injective ring homomorphism between $K$ and $M_d(\mathbb C)$. Remark as well that $\forall h\in G$, $\gamma(\phi(h))=h$.

As $K$ is generated by $\phi(g)=\zeta_g^1$ where $g\in G$ is an irreducible element, $\forall\theta\in K$, $\theta$ can be expressed as a rational polynomial in $\phi(g)$. Hence $\gamma(\theta)$ can be written as the same rational polynomial in terms of $g=\gamma(\phi(g))$. So $\gamma(\theta)\in M_d(\mathbb Q)$ because $g\in G<SL_d(\mathbb Z)$. Let $\theta_1,\theta_2,\cdots,\theta_d$ be a $\mathbb Z$-basis of $\mathcal O_K$, the ring of integers, then there must be an integer $D$ such that $\gamma(\theta_k)\in M_d(\frac1D\mathbb Z), \forall k=1,\cdots, d$. Therefore for all $\theta\in\mathcal O_K$, $\gamma(\theta)$, which is a $\mathbb Z$-span of $\gamma(\theta_k)$'s, lies in $M_d(\frac1D\mathbb Z)$ as well.

Moreover $\mathrm{det}(\gamma(\theta))=N(\theta)=1$ when $\theta\in U_K$, thus $\gamma$ is an isomorphism between $U_K$ and a subgroup in $SL_d(\frac1D\mathbb Z)$.

For all $\theta\in U_K$, consider the sequence of lattices $\{\gamma(\theta)^n.\mathbb Z^d|n=1,2,\cdots\}$. For any $n$, $\gamma(\theta)^n.\mathbb Z^d$ has determinant 1 and is a sublattice of $(\frac1D\mathbb Z)^d$ (because $\mathbb Z^d<(\frac1D\mathbb Z)^d$ and $(\gamma(\theta))^n=\gamma(\theta^n)\in M_d(\frac1D\mathbb Z)$), in other words it is a sublattice of index $D^d$ in $(\frac1D\mathbb Z)^d$. But there are only a finite number of sublattices of given index in a given lattice, hence there are $n_1\neq n_2$ such that $\gamma(\theta)^{n_1}.\mathbb Z^d=\gamma(\theta)^{n_2}.\mathbb Z^d$. Which implies $\gamma(\theta)^{n}.\mathbb Z^d=\mathbb Z^d$ for $n=n_1-n_2$, or equivalently, $\gamma(\theta)^{n}\in SL_d(\mathbb Z)$.

So we have actually proved, for any element $\gamma(\theta)$ in the finitely generated abelian group $\gamma(U_K)\cong U_K$, there is an exponent $n$ such that $(\gamma(\theta))^n\in SL_d(\mathbb Z)\cap\gamma(U_K)$. It follows that $\mathrm{rank}(\gamma(U_K)\cap SL_d(\mathbb Z))=\mathrm{rank}(\gamma(U_K))=\mathrm{rank}(U_K)$.

By Proposition \ref{KphiGamma}, $G=\gamma(\phi(G))\subset\gamma(U_K)\cap SL_d(\mathbb Z)$. Moreover, by assumption $\mathrm{rank}(G)=\mathrm{rank}\big(\gamma(U_K)\cap SL_d(\mathbb Z)\big)$, therefore $\mathrm{rank}(G)=\mathrm{rank}(U_K)$. Since $G\cong\phi(G)<U_K$, this yields the proposition.\end{proof}

\subsection{Total irreduciblity and non-CM number fields}

\hskip\parindent Now we deal with the assumption that $G<SL_d(\mathbb Z)$ contains at least one totally irreducible element.

\begin{lemma} If $g\in G$ is an irreducible toral automorphism, then $\phi(g)$ doesn't belong to any non-trivial proper subfield $F$ of $K$.\end{lemma}
\begin{proof} Any subfield $F$, viewed as a subset in $K\otimes_{\mathbb Q}\mathbb R$, is invariant under the multiplicative action of $U_F$. Assume $\phi(g)\in F$, then as the $g$-action on $\mathbb R^d$ is identified with that of $\phi(g)$ on $K\otimes_{\mathbb Q}\mathbb R$,  $g$ stablizes $\psi^{-1}(F)$. But in the proof of Proposition \ref{KphiGamma} we showed $\psi^{-1}(K)\subset\mathbb Q^d$,  hence $\psi^{-1}(F)$ is a rational subspace in $\mathbb Q^d$ of dimension between $1$ and $d-1$ because $F$ is a non-trivial proper $\mathbb Q$-vector subspace in $K$.  This contradicts the irreduciblity of $g$.\end{proof}

\begin{proposition}\label{subfield} In Proposition \ref{KphiGamma}, if $G$ contains a totally irreducible toral automorphism then the number field $K$ is not a CM-field.\end{proposition}

\begin{proof} Assume there is a proper subfield $F$ in $K$ with $\mathrm{rank}(U_F)=\mathrm{rank}(U_K)$ for contradiction. Let $g$ denote the totally irreducible element. Because the finite generated abelian groups $U_K$ and $U_F$ have the same rank, $\exists l\in\mathbb N$ such that $\phi(g^l)=\phi(g)^l\in U_K$ lies in $U_F$. The lemma above implies $g^l$ is not irreducible, thus contradicts the total irreducibility.\end{proof}

This proves Theorem \ref{condGfield}.

\section{Action of the group of units}\label{actionUK}

\hskip\parindent In the rest of this paper, we will always assume $G$, $K$, $\Gamma$, $X$, $\phi$ and $\psi$ are as in Condition \ref{condG'}. In addition, two constants will play important roles in ours study: the uniformity $\mathcal M_\psi$ of $\psi$ defined by Definition \ref{isouni} and the size of fundamental units $\mathcal F_{\phi(G)}$ in $\phi(G)<U_K$ defined by Definition \ref{funduni}.

\begin{notation}\label{Xnotation} $\forall 1\leq i\leq d$, $e_i$ denotes the $i$-th standard coordinate vector in $\mathbb R^d$, and by abuse of notation, those in the quotient $X=\mathbb R^d/\Gamma$ as well.

For $1\leq i\leq r_1$ let $V_i=\mathbb Re_i$ be the $i$-th real subspace in the product $\mathbb R^{r_1}\times\mathbb C^{r_2}\cong\mathbb R^d$. And let $V_{r_1+j}, 1\leq j\leq r_2$ be the $j$-th complex subspace (notice it is spanned by $e_{r_1+j}$ and $e_{r_1+r_2+j}$). Then $\mathbb R^d=\oplus_{i=1}^{r_1+r_2}V_i$.

For $g\in G$ and $1\leq i\leq d$, let $\zeta_g^i=\sigma_i(\phi(g))$, the $i$-th embedding of $\phi(g)\in U_K$. It is also the $i$-th eigenvalue of $g$, which is consistent with  earlier notations.

As before, write $r=r_1+r_2-1=\mathrm{rank}(U_K)=\mathrm{rank}(G)$ and $d_i=\mathrm{dim}V_i, 1\leq i\leq r$.\end{notation}

Recall $\forall g\in G$, if $1\leq i\leq r_1$ then $\times_{\phi(g)}.e_i=\sigma_i(g) e_i$; if $1\leq j\leq r_2$ then $$\left\{\begin{aligned}&\times_{\phi(g)}.e_{r_1+j}&&=(\mathrm{Re}\sigma_i(g))e_{r_1+j}+(\mathrm{Im}\sigma_i(g))e_{r_1+r_2+j};\\ &\times_{\phi(g)}.e_{r_1+r_2+j}&&=(-\mathrm{Im}\sigma_i(g))e_{r_1+j}+(\mathrm{Re}\sigma_i(g))e_{r_1+r_2+j}.\end{aligned}\right.$$

\subsection{Bounded totally irreducible units}

\hskip\parindent We proved earlier $K$ is not a CM-field from the total irreducibility. Conversely, a non-CM-field produces total irreducibility as well. Moreover, we can even construct a totally irreducible element with an explicit upper bound on its height. 

\declareconst{planecoverballconst}
\begin{lemma} Given $r\geq 2$, for any $N\in\mathbb N$, one needs at least $\ref{planecoverballconst}N$ proper linear subspaces in $\mathbb R^r$ to cover the discrete set $I_N^r=\{-N,-N+1,\cdots,N-1,N\}^r\subset\mathbb Z^r$ where $\ref{planecoverballconst}=\ref{planecoverballconst}(r)$ is effective.\end{lemma}

\begin{proof} $I_N^r$ is contained in the $r$-dimensional closed ball with radius $\sqrt rN$. The intersection of any proper linear subspace with this ball is a ball of radius $\sqrt rN$ in dimension $r-1$ or less. The number of lattice points in such a ball is $O_r(N^{r-1})$. Hence the number of proper subspaces neccessary to cover $I_N^r$ is at least $\frac{(2N+1)^r}{O_r(N^{r-1})}\gtrsim_r N$. Thus there is a $\ref{planecoverballconst}$ for which the lemma holds.\end{proof}

\declareconst{totirrheightconst}
\begin{proposition}\label{totirrheight} If Condition \ref{condG'} holds then there is a totally irreducible element $g\in G$ such that $h^\mathrm{Mah}(\phi(g))\leq\ref{totirrheightconst}\mathcal F_{\phi(G)}$ for some effective constant $\ref{totirrheightconst}(d)$.\end{proposition}

\begin{proof}By definition of $\mathcal F_{\phi(G)}$, there are $g_1,\cdots,g_r\in G$ such that the $\phi(g_k)$'s generate a finite-index subgroup of $\phi(G)$ and $h^\mathrm{Mah}(\phi(g_k))\leq\mathcal  F_{\phi(G)}$. Let $I_N^G=\{\prod_{i=1}^rg_j^{N_j}|-N\leq N_j\leq N\}$ for some $N\geq 1$. Then $I_N^W:=\{\mathcal L(\phi(g))|g\in I_N^G\}=\{\sum_{i=1}^rN_j\mathcal L(\phi(g_j))|-N\leq N_j\leq N\} $ is a discrete subset in $W$. It is obvious that the $\mathcal L(\phi(g_j))$'s span $W$. So $I_N^W$ is linearly isomorphic to $I_N^r$ and requires at least $\ref{planecoverballconst}N$ proper subspaces of $W$ to cover. 

\declaresubconst{numberofextconst}
It follows from Galois theory that there is an explicit constant $\ref{numberofextconst}(d)$ such that the degree $d$ number field $K$ can have at most $\ref{numberofextconst}(d)$ proper subfields. For each subfield $F$, by the assumption that $K$ is not a CM-field, $U_F$ is an abelian subgroup in $U_K$ of rank $r-1$ or less. Hence the $\mathbb R$-span of $\mathcal L(U_F)$ is a proper linear subspace in $W$. Take $N=\lfloor\frac{\ref{numberofextconst}}{\ref{planecoverballconst}}\rfloor+1$, then $\ref{planecoverballconst}N>\ref{numberofextconst}$ and there exists $g\in I_N^G$ such that $\mathcal L(\phi(g))\notin \mathbb R\mathcal L(U_F)$ for all proper subfields $F$. So $\forall n\in\mathbb N$, $\forall F$, $\mathcal L(\phi(g^n))=n\mathcal L(\phi(g))\notin\mathcal L(U_F)$ and thus $\phi(g^n)\notin F$, namely $\phi(g^n)$ is an irreducible element in $K$. Thus all eigenvalues of $g^n$, which are embeddings of $\phi(g^n)$, are of degree $d$. As a matrix in $SL_d(\mathbb Z)$ is irreducible if and only if its eigenvalues are conjugate to each other and have degree $d$, $g^n$ is irreducible for all $n\in\mathbb N$ and $g$ is totally irreducible.

$h^\mathrm{Mah}(\phi(g))\leq \sum_{j=1}^r N_jh^\mathrm{Mah}(\phi(g_j))\leq rN\mathcal F_{\phi(G)}$. As $N=\lfloor\frac{\ref{numberofextconst}}{\ref{planecoverballconst}}\rfloor+1$ explicitly depends on $d$ and $r$, we may take $\ref{totirrheightconst}=rN$, which depends only on $d$ and $r$. But by Remark \ref{rankchoice} $\ref{totirrheightconst}$ can be made independent of $r$ by taking maximum over all possible values of $r$ when $d$ is fixed.
\end{proof}

\subsection{Characters on $X$ and irrationality of eigenspaces}\label{irreig}

\hskip\parindent $X=\mathbb R^d/\Gamma$ can identified with $\mathbb T^d=\mathbb R^d/\mathbb Z^d$ by $\psi$. Notice the determinant of $\psi$ is exactly $\mathrm{covol}(\Gamma)=\mathrm{vol}(X)=\mathcal S_\Gamma^d$.

A character $\xi:X\mapsto\mathbb R/\mathbb Z$ is associated to the character $\mathbf q\in\mathbb Z^d=(\mathbb T^d)^*$ of $\mathbb T^d$ such that $\xi=\psi_*\mathbf q$, i.e. $\forall x\in X$, $\xi(x)=\langle\mathbf q,\psi^{-1}(x)\rangle$ where $\langle\cdot,\cdot\rangle$ is the usual inner product on $\mathbb R^n$. This association is bijiective.

\begin{definition}\label{absxi} Given $\psi$ with $\psi(\mathbb Z^d)=\Gamma$, for a character $\xi=\psi_*\mathbf q\in X^*$, denote $|\xi|=|\mathbf q|$ where $\mathbf q\in\mathbb Z^d$.\end{definition}

For all $1\leq i\leq r_1+r_2$, we define a vector $\tilde e_i\in V_i\otimes_{\mathbb R}\mathbb C\subset\mathbb C^n$ by: \begin{equation}\left\{\begin{aligned}&\tilde e_i&&=e_i,&& 1\leq i\leq r_1;\\&\tilde e_{r_1+j}&&=\frac{\sqrt2}2(e_{r_1+j}-{\mathrm i}e_{r_1+j}),\ &&1\leq j\leq r_2.\end{aligned}\right.\end{equation}  It is not hard to see $|\tilde e_i|=1$ and $\times_{\phi(g)}.\tilde e_i=\zeta_g^i\tilde e_i, \forall g,i$.

Recall when $V\cong\mathbb R$ or $\mathbb C$ is regarded as a real vector space and $f:V\mapsto\mathbb R$ be a real linear form on it, if $V\cong\mathbb R$ then $f(v)=f_0v$ for some $f_0\in\mathbb R$ and $\|f\|=|f_0|$; and if $V\cong\mathbb C$ then $f(v)=f_1\mathrm{Re} v+f_2\mathrm{Im}v, f_1,f_2\in\mathbb R$ and $\|f\|=\sqrt{f_1^2+f_2^2}$.

The following result characterizes quantitatively the irrationality of the $V_i$'s inside $X$.

\begin{proposition}\label{quantirr} If Condition \ref{condG'} holds, then for all $V_i$ and any non-trivial character $\xi\in X^*$, $$\big\|\xi\big|_{V_i}\big\|\geq d^{-1}2^{-\frac{d(d-1)}2\ref{totirrheightconst}\mathcal F_{\phi(G)}}\mathcal M_\psi ^{-(d-1)}|\xi|^{-(d-1)}\mathcal S_\Gamma^{-1},$$ where $\xi|_{V_i}$ denotes the restriction of $\xi$ to $V_i\subset\mathbb R^d$ and $\ref{totirrheightconst}=\ref{totirrheightconst}(d)$ is as in Proposition \ref{totirrheight}.\end{proposition}

\begin{proof} Suppose $\xi=\psi_*\mathbf q$ for some $\mathbf q\in\mathbb Z^d$. We view $\xi\in(\mathbb R^d)^*\subset(\mathbb C^d)^*\cong\mathbb C^d$ as a complex linear functional as well. Then $\|\xi\|$ coincide with the usual $l^2$-norm on $(\mathbb C^d)^*$.

By definition, if $1\leq i\leq r_1$ then $\big\|\xi\big|_{V_i}\big\|=|\xi(e_i)|=|\xi(\tilde e_i)|$ and for $1\leq j\leq r_2$ we have \begin{align*}\big\|\xi\big|_{V_i}\big\|=&\sqrt{\xi^2(e_{r_1+j})+\xi^2(e_{r_1+r_2+j})}\\
=&\sqrt{\xi^2(\sqrt 2\mathrm{Re}\tilde e_{r_1+j})+\xi^2(\sqrt 2\mathrm{Im}\tilde e_{r_1+j})}\\
=&\sqrt 2|\xi(\tilde e_{r_1+j})|.\end{align*}
So it suffices to show $|\xi(\tilde e_i)|\geq d^{-1}2^{-\frac{d(d-1)}2\ref{totirrheightconst}\mathcal F_{\phi(G)}}\mathcal M_\psi ^{-(d-1)}|\xi|^{-(d-1)}\mathcal S_\Gamma^{-1}, \forall i$.

Take the irreducible element $g\in G$ in Proposition \ref{totirrheight}. Denote $a=|\xi(\tilde e_i)|$, then $|\zeta_g^i|^na=|\xi( (\zeta_g^i)^n\tilde e_i)|=|\xi(\times_{\phi(g)}^n.\tilde e_i)|=|\big((\times_{\phi(g)}^\mathrm T)^n\xi\big)(\tilde e_i)|$ for all $n$. This means if we write $(\times_{\phi(g)}^\mathrm T)^n\xi=\xi_n^\bot+\xi_n$ where $\xi_n^\bot$ is the projection to the complex hyperplane $\tilde e_i^\bot$ and $\xi_n$ is orthogonal to it, then $\|\xi_n\|\leq\frac{|\zeta_g^i|^na}{|\tilde e_i|}=|\zeta_g^i|^na$.

Consider the determinant $\mathrm{det}\big((\times_{\phi(g)}^\mathrm T)^n\xi\big)_{n=0}^{d-1}$. It can be written as \begin{equation}\label{determdecomp}\mathrm{det}\big(\xi_n^\bot\big)_{n=0}^{d-1}+\sum_{m=0}^{d-1}\mathrm{det}\big(\xi_0^\bot,\cdots,\xi_{m-1}^\bot,\xi_m,(\times_{\phi(g)}^\mathrm T)^{m+1}\xi,\cdots,(\times_{\phi(g)}^\mathrm T)^{d-1}\xi\big).\end{equation}

$\mathrm{det}(\xi_n^\bot)_{n=0}^{d-1}=0$ as the $\xi_n^\bot$'s are all in the same hyperplane $\tilde e_i^\bot$ and for every index $m$, the absolute value of the corresponding term in the sum is at most
\begin{align*}&\|\xi_0^\bot\|\cdots\|\xi_{m-1}^\bot\|\cdot\|\xi_m\|\cdot\|(\times_{\phi(g)}^\mathrm T)^{m+1}\xi\|\cdots\|(\times_{\phi(g)}^\mathrm T)^{d-1}\xi\|\\\leq&\|\xi\|\cdots\|(\times_{\phi(g)}^\mathrm T)^{m-1}\xi\|\cdot\|\xi_m\|\cdot\|(\times_{\phi(g)}^\mathrm T)^{m+1}\xi\|\cdots\|(\times_{\phi(g)}^\mathrm T)^{d-1}\xi\|\\\leq&\|\times_{\phi(g)}^\mathrm T\|^{\sum_{0\leq j\leq d-1, j\neq m}j}\|\xi\|^{d-1}\cdot|\zeta_g^i|^ma.\\\leq&\|\times_{\phi(g)}^\mathrm T\|^{\sum_{j=0}^{d-1}j}\|\xi\|^{d-1}a.\end{align*} Here we used the fact that $\|\times_{\phi(g)}^\mathrm T\|=\max_{1\leq j\leq d}\|\sigma_j(g)\|=\max_{1\leq j\leq d}\|\zeta_g^j\|$, which follows directly from the construction of $\times_{\phi(g)}$. Moreover, $$\|\zeta_g^i\|\leq\|\times_{\phi(g)}^\mathrm T\|\leq 2^{\max_{1\leq j\leq d}\log\|\sigma_j(\phi(g))\|}\leq 2^{h^\mathrm{Mah}(\phi(g))}\leq 2^{\ref{totirrheightconst}\mathcal F_\phi(G)}.$$

From (\ref{determdecomp}) we obtain an upper bound \begin{equation}\label{determbd}|\mathrm{det}\big((\times_{\phi(g)}^\mathrm T)^n\xi\big)_{n=0}^{d-1}|\leq d2^{\frac{d(d-1)}2\ref{totirrheightconst}\mathcal F_\phi(G)}\|\xi\|^{d-1}a.\end{equation}

However $(\times_{\phi(g)}^\mathrm T)^n\xi=\psi_*\big((g^\mathrm T)^n\mathbf q\big)$, so $$\mathrm{det}\big((\times_{\phi(g)}^\mathrm T)^n\xi\big))_{n=0}^{d-1}=\det(\psi^{-1})\mathrm{det}\big((g^\mathrm T)^n\mathbf q\big)_{n=0}^{d-1}=\mathcal S_\Gamma^{-d}\mathrm{det}\big((g^\mathrm T)^n\mathbf q\big)_{n=0}^{d-1}.$$ But $\mathrm{det}\big((g^\mathrm T)^n\mathbf q\big)_{n=0}^{d-1}\neq 0$. This is because if the determinant vanishes then the $\mathbb Q$-span of the vectors $\{(g^\mathrm T)^n\mathbf q\}_{0\leq n\leq d-1}$ is a non-trivial proper rational subspace $L\subset(\mathbb R^d)^*$. But the characteristic polynomial of $g^T$ is rational and has degree $d$ by irreducibility, all powers of $g^\mathrm T$ are in the $\mathbb Q$-span of $\{(g^\mathrm T)^n\}_{0\leq n\leq d-1}$. Whence for all $n\in\mathbb Z$, $(g^\mathrm T)^n\mathbf q\in L$ and as a result, $L$ is invariant under $g^\mathrm T$. So $L^\bot$, which is a non-trivial proper subspace, is stablized by $g$. This contradicts the irreducibility of $g$. Therefore, the fact that $(g^\mathrm T)^n\mathbf q$ is an integer vector for all $n$ gives a lower bound $\big|\mathrm{det}\big((g^\mathrm T)^n\mathbf q\big)_{n=0..d-1}\big|\geq 1$. So $\big|\mathrm{det}\big((\times_{\phi(g)}^\mathrm T)^n\xi\big)_{n=0}^{d-1}\big|\geq\mathcal S_\Gamma^{-d}$.

On the other hand, $\|\xi\|=\|\psi_*q\|\leq\|\psi^{-1}\|\cdot\|q\|\leq\mathcal M_\psi \mathcal S_{\Gamma}^{-1}|q|=\mathcal M_\psi \mathcal S_{\Gamma}^{-1}|\xi|$.

Thus it follows from (\ref{determbd}) that\begin{align*}a\geq&d^{-1}2^{-\frac{d(d-1)}2\ref{totirrheightconst}\mathcal F_\phi(G)}\|\xi\|^{-(d-1)}\big|\mathrm{det}\big((\times_{\phi(g)}^\mathrm T)^n\xi\big)_{n=0}^{d-1}\big|\\
\geq&d^{-1}2^{-\frac{d(d-1)}2\ref{totirrheightconst}\mathcal F_\phi(G)}(\mathcal M_\psi \mathcal S_{\Gamma}^{-1}|\xi|)^{-(d-1)}\mathcal S_\Gamma^{-d}\\
=&d^{-1}2^{-\frac{d(d-1)}2\ref{totirrheightconst}\mathcal F_\phi(G)}\mathcal M_\psi ^{-(d-1)}|\xi|^{-(d-1)}\mathcal S_{\Gamma}^{-1}.\end{align*} Recall $a=|\xi(\tilde e_i)|$, the lemma follows.\end{proof}

\section{Positive entropy and regularity}\label{vecinVi}

\hskip\parindent So far we have only discussed algebraic properties of the group action, now we bring ergodic theory into the scene.
 
\begin{notation}By abuse of notation let $\mathrm m$ be the Lebesgue measure on $\mathbb R^d$ as well as that on $\mathbb T^d$. Denote by $\mathrm m_X$ the pushforward of $\mathrm m$ from $\mathbb R^d$ to $\mathbb R^d\cong\mathbb R^{r_1}\times\mathbb C^{r_2}$ by $\psi$. Abusing notation again, we call the projection of $\mathrm m_X$ to $X$ by $\mathrm m_X$ as well. Notice as the Lebesgue measure on $\mathbb T^d$ is $G$-invariant, $\mathrm m_X$ is $\times_{\phi(G)}$-invariant. And  on $X$, $\mathrm m_X$ is the pushforward of the Lebesgue measure $\mathrm m$ on $\mathbb T^d$ by $\psi$, thus the unique translation invariant probability measure.

Let $\pi$ denote the projection from $\mathbb R^d$ to $X=\mathbb R^d/\Gamma$.\end{notation}

As $|\det(\psi)|=\mathrm{Vol}(\psi(\mathbb Z^d))=\mathrm{Vol}(\Gamma)$, $\mathrm m_X=\frac{\mathrm{Vol}}{\mathrm{Vol}(\Gamma)}=\mathcal S_\Gamma^{-d}\cdot\mathrm{Vol}$.

From now on let $\mu$ be a Borel probability measure on $\mathbb T^d$ satisfying Condition \ref{condmu}, where $\epsilon$ will be supposed later to be very small depending on $\alpha$. We will always write \begin{equation}\eta=\mathcal M_\psi ^{-1}\epsilon.\end{equation}

We define the pushforward measure $\tau=\psi_*\mu$  on $X$. Then it has similar properties.

\begin{lemma}\label{condtau} Assuming Condition \ref{condmu}, then the inequality $H_\tau(\mathcal P)\geq\alpha d\log\frac1\eta-d\log\mathcal M_\psi $ holds for all measurable partitions $\mathcal P$ of $X$ such that $\mathrm{diam}\mathcal P\leq\eta\mathcal S_\Gamma$. \end{lemma}

\begin{proof} If the diameter of a partition $\mathcal P$ on $X$ is at most $\eta\mathcal S_\Gamma$, then that of its pullback $\psi^*\mathcal P$ on $\mathbb T^d$ is bounded by $\|\psi^{-1}\|\eta\mathcal S_\Gamma\leq\mathcal M_\psi \eta=\epsilon$. By Condition \ref{condmu}, $H_{\tau}(\mathcal P)=H_{\mu}(\psi^*\mathcal P)\geq\alpha d\log\frac1\epsilon=\alpha d(\log\frac1\eta-\log\mathcal M_\psi )\geq\alpha\log\frac1\eta-d\log\mathcal M_\psi $\end{proof}

\subsection{Some non-conventional entropies}

\hskip\parindent We demonstrated in \S \ref{irreig} the irrationality of subspace $V_i$ in $X$, a fact that makes it hard to construct measurable partitions on $X$ that are consistent with both the projection $\pi$ and the $G$-action simultaneously. For this reason we need some more technical arrangements.

\begin{definition} Let $\mu$ be a Borel probability measure on $X$. If a Borel set $B\subset\mathbb R^d$ is projected injectively by $\pi$ into $X$, the {\bf homogeneous entropy} of $B$ with respect to $\mu$ is $$\overline H_\mu(B)=\frac1{\mathrm m_X(B)}\int_{x\in X}-\mu(x+B)\log\mu(x+B)\mathrm{dm}_X(x).$$

When $\mathcal Q$ is a finite measurable partition of $B$, the {\bf homogeneous conditional entropy} of $\mathcal Q$ with respect to $\mu$ is defined to be $$\overline H_\mu(B,\mathcal Q)=\frac1{\mathrm m_X(B)}\int_{x\in X}\sum_{Q\in\mathcal Q}-\mu(x+Q)\log\frac{\mu(x+Q)}{\mu(x+B)}\mathrm{dm}_X(x).$$\end{definition}

By change of variable, the following fact is immediate.
\begin{lemma}\label{transinv} The homogeneous entropy and the homogeneous conditional entropy are translation invariant. Namely, $\forall B, \mathcal Q$, $\forall y\in\mathbb R^d$, let $y+B$ be the translation of $B$ by the vector $y$, and $y+\mathcal Q=\{y+Q|Q\in\mathcal Q\}$ be the partition defined by translating $\mathcal Q$ by $y$. Then $$\overline H_\mu(y+B)=\overline H_\mu(B), \overline H_\mu(y+B,y+\mathcal Q)=\overline H_\mu(B,\mathcal Q).$$\end{lemma}

\begin{lemma}\label{tauprob} Denote $\mathrm{dm}_{\mu,B}(x)=\frac1{\mathrm m_X(B)}\mu(x+B)\mathrm{dm}_X(x)$. Then $\mathrm  m_{\mu,B}$ is a probability measure.\end{lemma}
\begin{proof}The total mass is

\begin{align*}\int_{x\in X}\mathrm{dm}_{\mu,B}(x)=&\frac1{\mathrm m_X(B)}\int_{x\in X}\mu(x+B)\mathrm{dm}_X(x)\\
=&\frac1{\mathrm m_X(B)}\int_{x\in X}\int_{y\in X}\mathbf 1_{y-x\in B}\mathrm d\mu(y)\mathrm{dm}_X(x)\\
=&\frac1{\mathrm m_X(B)}\int_{y\in X}\int_{x\in X}\mathbf 1_{y-x\in B}\mathrm{dm}_X(x)\mathrm d\mu(y)\\
=&\frac1{\mathrm m_X(B)}\int_{y\in X}\mathrm m_X(y-B)\mathrm d\mu(y)\\
=&\frac1{\mathrm m_X(B)}\int_{y\in X}\mathrm m_X(B)\mathrm d\mu(y)=1.\end{align*}\end{proof}

\begin{remark}\label{entexp} $\overline H_\mu(B,\mathcal Q)$ can be written as \begin{align*}&\frac1{\mathrm m_X(B)}\int_{x\in X}\mu(x+B)\sum_{Q\in\mathcal Q}(-\frac{\mu(x+Q)}{\mu(x+B)}\log\frac{\mu(x+Q)}{\mu(x+B)})\mathrm{dm}_X(x)\\=&\mathbb E_{\mathrm m_{\mu,B}(x)}H_{\frac{\mu|_{x+B}}{\mu(x+B)}}(x+\mathcal Q).\end{align*} Where $H_{\frac{\mu|_{x+B}}{\mu(x+B)}}(x+\mathcal Q)$ stands for the usual measure theoretical entropy of the partition $x+\mathcal Q$ with respect to the renormalized probability measure $\frac{\mu|_{x+B}}{\mu(x+B)}$.\end{remark}

In particular, we deduce
\begin{corollary}\label{sizeQ} (i). For all $B$,$\mathcal Q$ and $\mu$, $\overline H_\mu(B,\mathcal Q)\leq\log|\mathcal Q|$.

(ii). If $\mathcal Q$ and $\mathcal Q'$ are two different partitions of $B$ then $\overline H_\mu(B,\mathcal Q\vee\mathcal Q')\leq \overline H_\mu(B,\mathcal Q)+\overline H_\mu(B,\mathcal Q')$. .\end{corollary}

\begin{definition} Let $\mu$ be as above. The entropy of a finite measuable cover $\mathcal C$ of $X$ is $H_\mu(\mathcal C)=\sum_{C\in\mathcal C}-\mu(C)\log\mu(C)$.\end{definition}

We are going to cope with special $B$'s and $\mathcal Q$'s.

For $\mathbf v=(v_i)_{i=1}^{r_1+r_2},\mathbf w=(w_i)_{i=1}^{r_1+r_2}\in(\mathbb R_+\cup\{0\})^{\{1,\cdots,r_1+r_2\}}$, $\mathbf v>\mathbf w$ means $v_i>w_i, \forall i$. Let $\mathbf 1_J=\mathbf 1_{i\in J}\in (\mathbb R_+\cup\{0\})^{\{1,\cdots,r_1+r_2\}}$ be the charateristic vector of $J\subset \{1,\cdots,r_1+r_2\}$ and simply denote $\mathbf 1=\mathbf 1_{\{1,\cdots,r_1+r_2\}}$, $\mathbf 1_i=\mathbf 1_{\{i\}}$.

For $\mathbf w\in (\mathbb R_+\cup\{0\})^{\{1,\cdots,r_1+r_2\}}$, we define a family of boxes whose sides in $V_i$ are of length $2^{-w_i}$.

\begin{definition} For $\mathbf w\in (\mathbb R_+\cup\{0\})^{\{1,\cdots,r_1+r_2\}}$, a set $B\subset\mathbb R^d=\oplus_{i=1}^{r_1+r_2}V_i$ is a {\it $\mathbf w$-box} if $B=\prod_{i=1}^{r_1+r_2}B_i$ where:

$B_i=[0,2^{-w_i}\mathcal S_\Gamma)$ inside $V_i$, if $V_i\cong\mathbb R$;

$B_i=\{x+\mathrm iy|x,y\in [0,2^{-w_i}\mathcal S_\Gamma)\}\subset V_i$, i.e. the square $[0,2^{-w_i}\mathcal S_\Gamma)\times[0,2^{-w_i}\mathcal S_\Gamma)$, if $V_i\cong\mathbb C$.

The family of all $\mathbf w$-boxes are denoted by $\Pi^\mathbf w$.\end{definition}

\begin{definition} For $\mathbf t\in(\mathbb N\cup\{0\})^{\{1,\cdots,r_1+r_2\}}$.

Given $B\in\Pi^\mathbf w$, set $F_{\mathbf t}B=f_{\mathbf t}(B)$ where $f_{\mathbf t}=\oplus_{i=1}^{r_1+r_2}2^{-t_i}\mathrm{Id}|_{V_i}$ rescales each $V_i$ by $2^{-t_i}$.

Then $F_{\mathbf t}B\in\Pi^{\mathbf w+\mathbf t}$ is a subset of $B$ and there is a unique partition of $B$ into $2^{\sum_{i=1}^{r_1+r_2}d_i\cdot t_i}$ different translates of $F_{\mathbf t}(B)$. Call this partition by $\mathcal Q_\mathbf tB$.\end{definition}

$F_\mathbf t$ maps $\Pi^\mathbf w$ to $\Pi^{\mathbf w+\mathbf t}$. It is an immediate observation that \begin{equation}\label{cuttwice}F_\mathbf t\circ F_\mathbf s=F_{\mathbf t+\mathbf s}.\end{equation}

\begin{lemma}\label{injectiveradius}Any box in $\Pi^{\mathbf w}$ is projected injectively by $\pi$ if $\mathbf w\geq\log(\sqrt d\mathcal M_\psi)\mathbf 1$.\end{lemma}
\begin{proof} If $B\in\Pi^{\mathbf w}$, then all side of it are at most of length $2^{-\log(\sqrt d\mathcal M_\psi)}\mathcal S_\Gamma$. So $\forall x,x'\in B$, $|x-x'|<\sqrt d\cdot 2^{-\log(\sqrt d\mathcal M_\psi)}\mathcal S_\Gamma=\mathcal M_\psi^{-1}\mathcal S_\Gamma$. Hence $|\psi^{-1}(x-x')|\leq\|\psi^{-1}\|\mathcal M_\psi^{-1}\mathcal S_\Gamma\leq 1$ by Definition \ref{isouni}. Thus unless $x=x'$, $\psi^{-1}(x-x')\notin\mathbb Z^d$, or equivalently, $x-x'\notin\psi(\mathbb Z^d)=\Gamma$ and $x$, $x'$ project to different points in $X=\mathbb R^d/\Gamma$.\end{proof}

\begin{lemma}\label{sumcuts} If $\pi|_B$ is injective, $\mu$ is a Borel probability measure on $X$ and $\mathbf t^1,\cdots,\mathbf t^n\in(\mathbb N\cup\{0\})^{\{1,\cdots,r_1+r_2\}}$, then $$\overline H_\mu(F_{\sum_{s=1}^n\mathbf t^s}B)=\overline H_\mu(B)+\sum_{s=1}^n\overline H_\mu(F_{\sum_{l=1}^{s-1}\mathbf t^l}B,\mathcal Q_{\mathbf t^s}F_{\sum_{l=1}^{s-1}\mathbf t^l}B).$$\end{lemma}

\begin{proof} The statement is trivial for $n=0$ by (\ref{cuttwice}), to conclude by induction it suffices to do the $n=1$ case. In this case there is only one vector $\mathbf t=\mathbf t^1$. By Lemma \ref{transinv}, $\overline H_\mu(F_\mathbf tB)=H_\mu(Q),\forall Q\in\mathcal Q_\mathbf tB$, therefore 

\begin{align*}\overline H_\mu(F_\mathbf tB)=&\frac1{|\mathcal Q_\mathbf tB|}\sum_{Q\in\mathcal Q_\mathbf tB}\overline H_\mu(Q)\\
=&\frac1{|\mathcal Q_\mathbf tB|}\sum_{Q\in\mathcal Q_\mathbf tB}\frac{|\mathcal Q_\mathbf tB|}{\mathrm m_X(B)}\int_{x\in X}-\mu(x+Q)\log\mu(x+Q)\mathrm{dm}_X(x)\\
=&\frac1{\mathrm m_X(B)}\int_{x\in X}\sum_{Q\in\mathcal Q_\mathbf tB}-\mu(x+Q)\log\mu(x+Q)\mathrm{dm}_X(x).\end{align*}

By decomposing $-\mu(x+Q)\log\mu(x+Q)$ into the sum of $-\mu(x+Q)\log\mu(x+B)$ and $-\mu(x+Q)\log\frac{\mu(x+Q)}{\mu(x+B)}$, we obtain

\begin{align*}
&\overline H_\mu(F_\mathbf tB)\\
=&\frac1{\mathrm m_X(B)}\int_{x\in X}-\mu(x+B)\log\mu(x+B)\mathrm{dm}_X(x)\\
&\hskip\parindent+\frac1{\mathrm m_X(B)}\int_{x\in X}\sum_{Q\in\mathcal Q_\mathbf tB}-\mu(x+Q)\log\frac{\mu(x+Q)}{\mu(x+B)}\mathrm{dm}_X(x)\\
=&\overline H_\mu(B)+\overline H_\mu(B,\mathcal Q_\mathbf tB).\end{align*}
This establishes the lemma.\end{proof}

\subsection{Positive entropy in an eigenspace}

\hskip\parindent In this part we will show that when both Conditions \ref{condmu} and \ref{condG'} are satisfied, there exists an index $i$ such that the measure $\tau=\psi_*\mu$ has positive entropy in the direction of $V_i$ with respect to some scale that we will specify.

Let us start with a box $B_0\in\Pi^{R_0\mathbf 1}$ where \begin{equation}\label{R0def}R_0:=\log\frac{\sqrt d}\eta=\log\frac{\sqrt d\mathcal M_\psi }\epsilon,\end{equation} then $\mathrm{diam}(B_0)=\sqrt d 2^{-R_0}\mathcal S_\Gamma=\eta\mathcal S_\Gamma$. As $\epsilon<1$, $R_0\geq\log(\sqrt d\mathcal M_\psi)$ always holds and $\pi|_{B_0}$ is injective by Lemma \ref{injectiveradius}. We now try to give an estimate for the average entropy $\overline H_\tau(B_0)$.

\declareconst{posentropydeltaconst}
\begin{lemma}\label{posR0} Assuming Conditions \ref{condmu} and \ref{condG'}, if a constant $\ref{posentropydeltaconst}=\ref{posentropydeltaconst}(d)$ is sufficiently large then for any $\delta\geq\frac {\ref{posentropydeltaconst}\mathcal M_\psi^d}{R_0}$ and box $B_0\in\Pi^{R_0\mathbf 1}$, $$\overline H_\tau(B_0)\geq (\alpha-\delta)dR_0,$$ where $\tau=\psi_*\mu$ with $\psi$ and $\mu$ coming respectively from Conditions \ref{condmu} and \ref{condG'}\end{lemma}

\begin{proof} Let $b_i=\psi(e_i)$, then $|b_i|\leq\|\psi\|\leq\mathcal M_\psi \mathcal S_\Gamma$. The parallelepiped $D=\{\sum_{j=1}^d\theta_jb_j|\theta_j\in[0,1),\forall j\}$ spanned by the linear basis $\{b_1,\cdots ,b_d\}$ for the lattice $\Gamma$ is a fundamental domain of the projection $\pi$. 

The preimage $\psi^{-1}(B_0)\subset \mathbb R^d$ is inside a ball of radius $\|\psi^{-1}\|\mathrm{diam}(B_0)\leq\mathcal M_\psi\mathcal S_\Gamma^{-1}\cdot\eta\mathcal S_\Gamma=\mathcal M_\psi\eta=\epsilon$. Hence $\psi^{-1}(B_0-B_0)=\psi^{-1}(B_0)-\psi^{-1}(B_0)$ is covered by the ball of radius $2\epsilon$ centered at the origin.  Notice $\psi^{-1}(D)=[0,1)^d\subset\mathbb R^d$, hence \begin{equation}\label{fdboudary}\psi^{-1}(D+B_0-B_0)\subset (-2\epsilon,1+2\epsilon)^d,\end{equation} which can be covered by at most $5^d$ translates of $\psi^{-1}(D)$ as $\epsilon<1$. Thus $D+B_0-B-0$ is covered by at most $5^d$ translates of $D$, each of which is a fundamental domain of $\pi$. Hence $\forall x\in  X$, $|\pi^{-1}(x)\cap(D+B_0-B_0)|\leq 5^d$.

We may tile $\mathbb R^d$ by translating $B_0$: $\mathbb R^d=\sqcup_{y\in\Sigma}(y+B_0)$ where $\Sigma\subset \mathbb R^d$ is a lattice. Denote $Y=\{y\in\Sigma|(y+B_0)\cap D\neq\emptyset\}$. Then $D\subset D_+\subset D+B_0-B_0$ where $D_+=\sqcup_{y\in Y}(y+B_0)$.

Let $\mathcal C$ be the collection $\{\pi(y+B_0)|y\in Y\}$ then it is a cover of $X$ as $\cup_{C\in\mathcal C}C\supset \cup_{y\in Y}\pi(y+B_0)=\pi(D_+)\supset\pi(D)= X$. Furthermore every point in $X$ is covered by at most $5^d$ pieces from $\mathcal C$ as $|\{C\in\mathcal C|x\in C\}|=|\{y\in Y|x\in\pi(y+B_0)\}|=|\pi^{-1}(x)\cap D_+|\leq |\pi^{-1}(x)\cap (D+B_0-B_0)|\leq 5^d$.

Denote $x+\mathcal C=\{x+C|C\in\mathcal C\}$ for all $x\in X$. $x+\mathcal C$ is the translate of $\mathcal C$ by $x$ and hence remains a cover of $X$ whose multiplicity at any given point is at most $5^d$. Then 
\begin{equation}\label{entBentC1}\begin{split}&\int_{x\in X}H_\tau(x+\mathcal C)\mathrm{dm}_X(x)\\
=&\int_{x\in X}\sum_{C\in\mathcal C}-\tau(x+C)\log \tau(x+C)\mathrm{dm}_X(x)\\
=&\int_{x\in X}\sum_{y\in Y}-\tau(x+y+B_0)\log \tau(x+y+B_0)\mathrm{dm}_X(x).\end{split}\end{equation}
As $\mathrm{m}_X(y+B_0)=\mathrm{m}_X(B_0)$ and $\overline H_\tau(y+B_0)=\overline H_\tau(B_0)$ for all $y\in Y$, by definition of $\overline H_\tau(y+B_0)$,
\begin{equation}\label{entBentC}\begin{split}
(\ref{entBentC1})=&\sum_{y\in Y}\mathrm m_X(B_0)\overline H_\tau(y+B_0)=|\mathcal C|\mathrm m_X(B_0) \overline H_\tau(B_0)\\
=&\mathrm m_X(D_+) \overline H_\tau(B_0),\end{split}\end{equation}

The cover $\mathcal C$ induces a partition $\mathcal P_\mathcal C$ of $X$: $$\mathcal P_\mathcal C=\{\pi\big((y+B_0)\cap D\big)|y\in Y\}.$$ Again let $x+\mathcal P_\mathcal C=\{x+P|P\in\mathcal P_\mathcal C\}$ denote its translate by $x$. There is a one-to-one correspondence between $\mathcal C$ and $\mathcal P_\mathcal C$: $P_C=\pi\big((y+B_0)\cap D\big)\leftrightarrow C=\pi(y+B_0)$ and for all $C$ we have $P_C\subset C$. As $C\in\mathcal C$ is always a translate of $B_0$ we know the diameter of $\mathcal P_\mathcal C$ is at most $\eta\mathcal S_\Gamma$ and so is that of $x+\mathcal P_\mathcal C$, $\forall x\in  X$. So by Lemma \ref{condtau}, \begin{equation}\label{entC1}H_\tau(x+\mathcal P_\mathcal C)\geq\alpha d\log\frac1\eta-d\mathcal M_\psi.\end{equation}

On the other hand, notice
\begin{equation}\label{entC2}\begin{split}&H_\tau(x+\mathcal C)\\
=&\sum_{C\in\mathcal C}-\tau(x+C)\log\tau(x+C)\\
\geq&\sum_{C\in\mathcal C}-\tau(x+P_C)\log\tau(x+C)\\
=&\sum_{C\in\mathcal C}\big(-\tau(x+P_C)\log\tau(x+P_C)+\tau(x+P_C)\log\frac{\tau(x+P_C)}{\tau(x+C)}\big)\\
=&H_\tau(x+\mathcal P_\mathcal C)-\sum_{C\in\mathcal C}\tau(x+C)\big(-\frac{\tau(x+P_C)}{\tau(x+C)}\log\frac{\tau(x+P_C)}{\tau(x+C)}\big).\end{split}\end{equation}

It follows from (\ref{entC1}), (\ref{entC2}), the inequality $-u\log u\leq\frac1{e\ln2},\forall u\in [0,1]$ and the fact that the multiplicity of the cover $x+\mathcal C$ is at most $5^d$ that 
\begin{equation}\label{entC}\begin{split}H_\tau(x+\mathcal C)\geq&\alpha d\log\frac1\eta-d\log\mathcal M_\psi -\frac1{e\ln2}\sum_{C\in\mathcal C}\tau(x+C)\\
\geq&\alpha dR_0-\alpha d\log\sqrt d-d\log\mathcal M_\psi -\frac1{e\ln2}\cdot 5^d,\end{split}\end{equation}

which, together with (\ref{entBentC}), implies \begin{equation}\label{entB1}\overline H_\tau(B_0)\geq\frac1{\mathrm m_X(D_+)}\big(\alpha dR_0-O_{d}(\log\mathcal M_\psi )-O_d(1)\big).\end{equation}

By (\ref{fdboudary}), the volume of $\psi^{-1}\big((D+B_0-B_0)\backslash D\big)$ is bounded by $(1+4\epsilon)^d-1$, hence
\begin{equation}\label{entB2}\begin{split}&\mathrm m_X(D_+)-1\\
\leq&\mathrm m_X(D+B_0-B_0)-\mathrm m_X(D)=\mathrm m_X\big((D+B_0-B_0)\backslash D\big)\\
=&\mathrm m\Big(\psi^{-1}\big((D+B_0-B_0)\backslash D\big)\Big)\\
\leq&(1+4\epsilon)^d-1=O_d(\epsilon),\end{split}\end{equation}
 where the implied dimensional constant can be made explicit. 

From (\ref{entB1}) and (\ref{entB2}) we conclude \begin{equation}\label{entB}\overline H_\tau(B_0)\geq\frac1{1+O_d(\epsilon)}\big(\alpha dR_0-O_{d}(\log\mathcal M_\psi)-O_d(1)\big),\end{equation} with effective implied constants.

So when $\delta\in[\frac {\ref{posentropydeltaconst}\mathcal M_\psi^d}{R_0},\frac\alpha{10}]$, as $\mathcal M_\psi\geq 1$, $$\epsilon=\frac1{\sqrt  d\mathcal M_\psi}2^{-R_0}\lesssim_{d} 2^{-R_0}\lesssim \frac1{R_0}\leq\frac\delta{\ref{posentropydeltaconst}}<\frac\delta{\ref{posentropydeltaconst}\alpha};$$ in addition $$O_{d}(\log\mathcal M_\psi)+O_d(1)=O_d(\mathcal M_\psi^d)=O_d(\frac{\delta R_0}{\ref{posentropydeltaconst}}).$$ So by (\ref{entB}),
$$\overline H_\tau(B_0)\geq\frac1{1+O_d(\frac\delta{\ref{posentropydeltaconst}\alpha})}\big(\alpha dR_0-O_d(\frac{\delta R_0}{\ref{posentropydeltaconst}})\big)\geq(\alpha-\delta)dR_0,$$  if $\ref{posentropydeltaconst}$ is sufficiently large (depending on $d$).\end{proof}

The following fact is a consequence to additivity (Lemma \ref{sumcuts}) 

\begin{lemma}\label{posT} For a sufficiently large effective constant $\ref{posentropydeltaconst}=\ref{posentropydeltaconst}(d)$, suppose $\delta\geq\frac {\ref{posentropydeltaconst}\mathcal M_\psi^d}{R_0}$ and $T\leq \frac{\delta R_0}2$ be an integer. Then there exist a real number $R\in[\delta R_0, R_0-T]$ and a box $B\in\Pi^{R\mathbf 1}$ such that $$\overline H_\tau(B,\mathcal Q_{T\mathbf 1}B)\geq(\alpha-3\delta)dT.$$\end{lemma}

\begin{proof} Write $R_0=S+pT$ where $S\in[\delta R_0,\delta R_0+T)$ and $p\in\mathbb N$. Note $S\geq\ref{posentropydeltaconst}\mathcal M_\psi^d$ by the assumption about $\delta$. We can always take $\ref{posentropydeltaconst}=\ref{posentropydeltaconst}(d)$ large to make sure $\ref{posentropydeltaconst}\mathcal M_\psi^d\geq\log(\sqrt d\mathcal M_\psi)$ for all $\mathcal M_\psi\geq 1$. Hence there is $S_-\in[\log(\sqrt d\mathcal M_\psi),\log(\sqrt d\mathcal M_\psi)+1]$ such that $S-S_-=q\in\mathbb N\cup\{0\}$.

Fix $B_0\in\Pi^{R_0\mathbf 1}$, $B_{S_-}\in\Pi^{S_-\mathbf 1}$ and $B_{S}\in\Pi^{S\mathbf 1}$ such that $B_0=F_{pT\mathbf 1}B_S$ and $B_S=F_{q\mathbf 1}B_{S_-}$. By Lemma \ref{sumcuts} and \ref{posR0}
\begin{equation}\label{entseqdecomp}\begin{split}&\overline H_\tau(B_{S_-})+\overline H_\tau(B_{S_-},\mathcal Q_{q\mathbf 1}B_{S_-})+\sum_{r=0}^{p-1}\overline H_\tau(F_{rT\mathbf 1}B_S,\mathcal Q_{T\mathbf 1}F_{rT\mathbf 1}B_S)\\=&\overline H_\tau(B_0)\geq(\alpha-\delta)dT.\end{split}\end{equation}

On the other hand, using again $-u\log u\leq\frac1{e\ln 2}$, \begin{equation}\label{entbase}\begin{split}\overline H_\tau(B_{S_-})=&\frac1{\mathrm m_X(B_{S_-})}\int_{x\in X}-\nu(x+B_{S_-})\log\nu(x+B_{S_-})\mathrm{dm}_X(x)\\\leq&\frac{(2\sqrt d)^d\mathcal M_\psi^d}{e\ln 2},\end{split}\end{equation}
because $\mathrm m_X(B_{S_-})=\mathcal S_\Gamma^{-d}\mathrm{vol}(B_{S_-})=\mathcal S_\Gamma^{-d}\big(2^{-S_-}\mathcal S_\Gamma)^d=2^{-d\mathcal S_-}$ has lower bound $2^{-d\big(\log(\sqrt d\mathcal M_\psi)+1\big)}=(2\sqrt d)^{-d}\mathcal M_\psi^{-d}$.

Moreover by Corollary \ref{sizeQ} \begin{equation}\label{entbase'}\begin{aligned}\overline H_\tau(B_{S_-},\mathcal Q_{q\mathbf 1}B_{S_-})\leq&\log|\mathcal Q_{q\mathbf 1}B_{S_-}|\leq dq\leq dS\leq d(\delta R_0+T)\\
\leq &d(\delta R_0+\frac{\delta R_0}2).\end{aligned}\end{equation}

Notice by taking the explicit constant $\ref{posentropydeltaconst}(d)$ large enough we have $d\delta R_0\geq d\ref{posentropydeltaconst}\mathcal M_\psi^d>2\cdot \frac{(2\sqrt d)^d\mathcal M_\psi^d}{e\ln 2}$. 

So by (\ref{entseqdecomp}), (\ref{entbase}) and (\ref{entbase'}), $\exists r\in\{0,1,\cdots,p-1\}$ such that 
\begin{equation}\begin{split}&\overline H_\tau(F_{rT\mathbf 1}B_S,\mathcal Q_{T\mathbf 1}F_{rT\mathbf 1}B_S)\\
\geq&\frac1p\big((\alpha-\delta)dR_0-d(\delta R_0+\frac{\delta R_0}2)-\frac{(2\sqrt d)^d\mathcal M_\psi^d}{e\ln 2}\big)\\
\geq&\frac1p\big((\alpha-\delta)dR_0-d(\delta R_0+\frac{\delta R_0}2)-\frac{d\delta R_0}2\big)\\
=&(\alpha-3\delta)d\frac{R_0}p\\
\geq&(\alpha-3\delta)dT\end{split}\end{equation}

$F_{rT\mathbf 1}B_S\in\Pi^{R\mathbf 1}$ where $R=S+rT$. As $R\geq S\geq\delta R_0$ and $R\leq S+(p-1)T=R_0-T$, the lemma follows. \end{proof}

The lemma basically claims $\tau$ has positive entropy at scale $2^{-R}\mathcal S_\Gamma$. It follows from subadditivity (Corollary \ref{sizeQ}) that the projection of $\tau$ to at least one of the $V_i$'s has positive entropy at the same scale, which is characterized by:

\begin{corollary}There is an index $i\in\{1,\cdots,r_1+r_2\}$ such that $$\overline H_\tau(B,\mathcal Q_{T\mathbf 1_i}B)\geq (\alpha-3\delta)d_iT.$$\end{corollary}

\begin{proof}Since $\mathcal Q_{T\mathbf 1}B=\bigvee_{i=1}^{r_1+r_2}\mathcal Q_{T\mathbf 1_i}B$, by Corollary \ref{sizeQ} $$\sum_{i=1}^{r_1+r_2}\overline H_\tau(B,\mathcal Q_{T\mathbf 1_i}B)\geq \overline H_\tau(B,\mathcal Q_{T\mathbf 1}B)\geq (\alpha-3\delta)dT.$$ Recall $d_i=\dim V_i$; as $\sum_{i=1}^{r_1+r_2}d_i=d$, the corollary is proved.\end{proof}

\subsection{From entropy to $L^2$-norm}

\hskip\parindent Using the positive entropy of $\tau$ in $V_i$ direction, we now construct a new measure dominated by $\tau$ so that its total mass is bounded from below but a certain $L^2$-norm in $V_i$ direction is bounded from above.

Let $T$ and $B$ be as in Lemma \ref{posT}. For all $x \in X$ such that $\tau(x+B)\neq 0$, define a probability measure $\tau_x=\frac{\tau|_{x+B}}{\tau(x+B)}$, supported on $x+B$. Set $\mathcal S_x=\{Q\in\mathcal Q_{T\mathbf 1_i}B|\frac{\tau(x+Q)}{\tau(x+B)}\leq 2^{-d_i\delta T}\}=\{Q\in\mathcal Q_{T\mathbf 1_i}B|\tau_x(x+Q)\leq 2^{-d_i\delta T}\}$ and $\mathcal S'_x=\mathcal Q\backslash\mathcal S_x$.

Notice $\tau_x$, $\mathcal S_x$ and $\mathcal S'_x$ are defined for $\mathrm m_{\tau,B}$-a.e. $x$, where $\mathrm m_{\tau,B}$ is defined as in Lemma \ref{tauprob}. Quoting Remark \ref{entexp} and Lemma \ref{posT} we have \begin{equation}\label{singdecomp}\begin{split}&(\alpha-3\delta)d_iT\\
\leq&\overline H_\tau(B,\mathcal Q_{T\mathbf 1_i}B)\\
=&\mathbb E_{\mathrm m_{\tau,B}(x)}H_{\tau_x}(x+\mathcal Q_{T\mathbf 1_i}B)\\
=&\mathbb E_{\mathrm m_{\tau,B}(x)}H_{\tau_x}\Big(\big\{\bigcup_{Q\in\mathcal S_x}(x+Q),\bigcup_{Q\in\mathcal S'_x}(x+Q)\big\}\Big)\\
&\hskip.5cm+\mathbb E_{\mathrm m_{\tau,B}(x)}H_{\tau_x}\Big(x+\mathcal Q_{T\mathbf 1_i}B\big|\big\{\bigcup_{Q\in\mathcal S_x}(x+Q),\bigcup_{Q\in\mathcal S'_x}(x+Q)\big\}\Big)\\
=&\mathbb E_{\mathrm m_{\tau,B}(x)}H_{\tau_x}\Big(\big\{\bigcup_{Q\in\mathcal S_x}(x+Q),\bigcup_{Q\in\mathcal S'_x}(x+Q)\big\}\Big)\\
&\hskip.5cm+\mathbb E_{\mathrm m_{\tau,B}(x)}\sum_{Q\in\mathcal S_x}\tau_x(x+Q)\cdot\big(-\log\frac{\tau_x(x+Q)}{\tau_x(\bigcup_{Q\in\mathcal S_x}(x+Q))}\big)\\
&\hskip1cm+\mathbb E_{\mathrm m_{\tau,B}(x)}\sum_{Q\in\mathcal S'_x}\tau_x(x+Q)\cdot\big(-\log\frac{\tau_x(x+Q)}{\tau_x(\bigcup_{Q\in\mathcal S'_x}(x+Q))}\big)\end{split}\end{equation}

Observe all three terms on the right-hand side can be bounded from above:

By Corollary \ref{sizeQ}, $H_{\tau_x}\Big(\big\{\bigcup_{Q\in\mathcal S_x}(x+Q),\bigcup_{Q\in\mathcal S'_x}(x+Q)\big\}\Big)\leq\ln 2;$

Moreover, \begin{align*}&\sum_{Q\in\mathcal S_x}\tau_x(x+Q)\cdot\big(-\log\frac{\tau_x(x+Q)}{\tau_x(\bigcup_{Q\in\mathcal S_x}(x+Q))}\big)\\
	=&\tau_x(\bigcup_{Q\in\mathcal S_x}(x+Q))H_{\frac{\tau_x|_{\bigcup_{Q\in\mathcal S_x}(x+Q)}}{\tau_x(\bigcup_{Q\in\mathcal S_x}(x+Q))}}(\mathcal S_x)\\
	\leq&\tau_x(\bigcup_{Q\in\mathcal S_x}(x+Q))\cdot\log|\mathcal S_x|\\
	\leq&\tau_x(\bigcup_{Q\in\mathcal S_x}(x+Q))\cdot\log|\mathcal Q_{T\mathbf 1_i}B|\\
	=&\tau_x(\bigcup_{Q\in\mathcal S_x}(x+Q))\cdot\log2^{d_iT}\\
	=&d_iT\tau_x(\bigcup_{Q\in\mathcal S_x}(x+Q)).\end{align*}

Last, \begin{align*}&\sum_{Q\in\mathcal S'_x}\tau_x(x+Q)\cdot\big(-\log\frac{\tau_x(x+Q)}{\tau_x(\bigcup_{Q\in\mathcal S'_x}(x+Q))}\big)\\
	\leq&\sum_{Q\in\mathcal S'_x}\tau_x(x+Q)\cdot\big(-\log\tau_x(x+Q)\big)\\
	\leq&\sum_{Q\in\mathcal S'_x}\tau_x(x+Q)\cdot(-\log2^{-d_i\delta T})\\
	=&d_i\delta  T\sum_{Q\in\mathcal S'_x}\tau_x(x+Q)\\=&d_i\delta  T\end{align*}

So (\ref{singdecomp}) gives \begin{equation}\label{singdecomp'}(\alpha-3\delta)d_iT\leq\ln 2+d_iT\mathbb E_{\mathrm m_{\tau,B}(x)}\tau_x(\bigcup_{Q\in\mathcal S_x}(x+Q))+d_i\delta  T.\end{equation}

In consequence, if \begin{equation}\label{deltaTassump}\delta T\geq 1\end{equation} then $d_i\delta T\geq\delta T\geq\ln 2$ and by (\ref{singdecomp'}), $d_iT\tau_x(\bigcup_{Q\in\mathcal S_x}(x+Q))\geq(\alpha-3\delta)d_iT-\ln 2-d_i\delta T\geq (\alpha-5\delta)d_iT$, so we get \begin{equation}\label{nonsing}\mathbb E_{\mathrm m_{\tau,B}(x)}\tau_x(\bigcup_{Q\in\mathcal S_x}(x+Q))\geq \alpha-5\delta.\end{equation}

For all $x$ such that $\tau(x+B)\neq 0$ set $$\nu_x:=\sum_{Q\in\mathcal S_x}\tau_x|_{x+Q},$$ which is supported on $\bigcup_{Q\in\mathcal S'_x}(x+Q)\subset x+B$. Then \begin{equation}\label{subprob}|\nu_x|\leq 1\end{equation} as $\nu_x$ is bounded by the probability measure $\tau_x$.

If We define a new measure $$\nu:=\mathbb E_{\mathrm m_{\tau,B}(x)}\nu_x,$$ then it follows from (\ref{nonsing}) that the total mass \begin{equation}\label{numass}|\nu|\geq\alpha-5\delta.\end{equation} Moreover, observe \begin{align*}\mathbb E_{\mathrm m_{\tau,B}(x)}\tau_x=&\frac1{\mathrm m_X(B)}\int_{x\in X}\frac{\tau|_{x+B}}{\tau(x+B)}\cdot\tau(x+B)\mathrm{dm}_X(x)\\
=&\frac1{\mathrm m_X(B)}\int_{x\in X}\tau\cdot\mathbf 1_{x+B}\mathrm{dm}_X(x)\\
=&\frac\tau{\mathrm m_X(B)}\int_{x\in X}\mathbf 1_{x+B}\mathrm{dm}_X(x)\\
=&\frac\tau{\mathrm m_X(B)}\cdot\mathrm m_X(B)\mathbf 1_X=\tau.\end{align*} Hence \begin{equation}\label{nutau}\nu=\mathbb E_{\mathrm m_{\tau,B}(x)}\nu_x\leq \mathbb E_{\mathrm m_{\tau,B}(x)}\tau_x=\tau.\end{equation}
Further, remark \begin{equation}\label{L2norm}\begin{split}&\mathbb E_{\mathrm m_{\tau,B}(x)}\sum_{Q\in\mathcal Q_{T\mathbf 1_i}B}\nu_x^2(x+Q)\\
\leq&\mathbb E_{\mathrm m_{\tau,B}(x)}\big(\max_{Q\in\mathcal Q_{T\mathbf 1_i}B}\nu_x(x+Q)\big)\big(\sum_{Q\in\mathcal Q_{T\mathbf 1_i}B}\nu_x(x+Q)\big)\\
=&\mathbb E_{\mathrm m_{\tau,B}(x)}\big(\max_{Q\in\mathcal S_x}\nu_x(x+Q)\big)\nu_x(x+B),\end{split}\end{equation}
where the last step was because $\nu_x$ is supported on $x+\cup_{Q\in\mathcal S_x}Q$. Since $\forall Q\in\mathcal S_x$, restricted to $x+Q$, $\nu_x$ is identical to $\tau_x$, 
\begin{equation}\label{L2normbound}\begin{split}(\ref{L2norm})=&\mathbb E_{\mathrm m_{\tau,B}(x)}\big(\max_{Q\in\mathcal S_x}\tau_x(x+Q)\big)\nu_x(x+B)\\
\leq&\mathbb E_{\mathrm m_{\tau,B}(x)}2^{-d_i\delta T}\nu_x(x+B)\\
\leq&2^{-d_i\delta T}.\end{split}\end{equation}

To sum up before finishing Section \ref{vecinVi}, what has been proved so far (\ref{subprob}, \ref{numass}, \ref{nutau}, \ref{L2normbound}) is the following proposition.

\begin{proposition}\label{dominated} There is a constant $\ref{posentropydeltaconst}$ depending effectively on the dimension $d$ such that if Conditions \ref{condG'} and \ref{condmu} are satisfied with $\delta\in[\frac{\ref{posentropydeltaconst}\mathcal M_\psi^d}{R_0},\frac\alpha{10}]$, then $\forall T\in[\frac1\delta,\frac{\delta R_0}2]$, where $\sqrt d2^{-R_0}=\eta=\mathcal M_\psi ^{-1}\epsilon$, there exist:

(i). A number $R\in[\delta R_0,R_0-T]$;

(ii).  A box $B\in\Pi^{R\mathbf 1}$;

(iii).  A probability measure $\mathrm m_{\tau,B}$ on $X$;

(iv).  A family of measures $\nu_x$ supported on $x+B$ defined for $\mathrm m_{\tau,B}$-a.e. $x$;

(v).  An index $i\in \{1,2,\cdots,r_1+r_2\}$;\\
such that the measure $\nu$ given by $\nu=\mathbb E_{\mathrm m_{\tau,B}(x)}\nu_x$ satisfies:

(1). $\nu\leq\tau$, where $\tau=\psi_*\mu$;

(2). $|\nu_x|\leq 1$, $\mathrm m_{\tau,B}$-a.e. $x\in X$;

(3). $|\nu|\geq\alpha-5\delta$;

(4).  $\mathbb E_{\mathrm m_{\tau,B}(x)}\sum_{Q\in\mathcal Q_{T\mathbf 1_i}B}\nu_x^2(x+Q)\leq 2^{-d_i\delta T}$, where $d_i=\dim V_i$.

\end{proposition}

\section{Group action on a single eigenspace}\label{GVi}

\hskip\parindent In the special direction $V_i$ which is isomorphic to either $\mathbb R$ or $\mathbb C$ we obtained from last section, the group $G$ acts multiplicatively: $\forall v\in V_i, g\in G$, $\times_{\phi(g)}.v=\sigma_i(\phi(g))v=\zeta^i_gv$.

From now on let $i$ be fixed. We are going to discuss several types of behavior of the $G$-action on $V_i$ that can help us.

\subsection{Expansion}

\hskip\parindent This is the simplest behavior of the $G$-action that we study. When the $j$-th eigenvalue $\zeta_g^j$ of an element $g\in G$ is of absolute value greater than 1 (resp. less than 1), we say $g$ {\it expands} (resp. {\it contracts}) $V_i$, which is equivalent to that $\mathcal L(\phi(g))$ lies in $\{w\in W|w_i>0\}$ (resp. $\{w\in W|w_i<0\}$), where $\mathcal L$ and $W$ are defined as in \S \ref{prelimfield}.

\begin{definition} In an abelian group, two elements $x$, $y$ are said to be {\bf multiplicatively indenpendent} if there doesn't exist $(p,q)\in\mathbb Z^2\backslash\{(0,0)\}$ such that $x^py^q$ is identity.\end{definition}

It follows easily from Dirichlet's Unit Theorem that there exists an element from $G$ that expands $V_i$ and contracts all $V_j, j\neq i$. Actually we can say more.

\begin{proposition}\label{expand} Given any $i\in \{1,\cdots,r_1+r_2\}$, there is a pair of multiplicatively independent elements $u,\tilde u$ from $G$ such that $$\left\{\begin{aligned}&h^\mathrm{Mah}(\phi(u)),h^\mathrm{Mah}(\phi(\tilde u))<9d(\frac r2+1)\mathcal F_{\phi(G)};\\
&|\zeta_u^i|\geq 2^{d(\frac r2+1)\mathcal F_{\phi(G)}}, |\zeta_u^j|<1, \forall j\neq i;\\
&|\zeta_{\tilde u}^i|\geq 2^{\frac14d(\frac r2+1)\mathcal F_{\phi(G)}}.\end{aligned}\right.$$ Here $h$ is the logarithmic Mahler measure introduced in Definition \ref{logh}.

Moreover, in case that $V_i\cong\mathbb C$ it can be required that $\zeta_{u^n}^i\notin\mathbb R, \forall n\in\mathbb N$ and if $\zeta_{u^p\tilde u^q}^i\in\mathbb R$ and $|\zeta_{u^p\tilde u^q}^i|>1$ for some pair $(p,q)\in\mathbb Z^2$,  then $|\zeta_{u^p\tilde u^q}^i|\geq2^{\frac14d(\frac r2+1)\mathcal F_{\phi(G)}}$.\end{proposition}

It should be pointed out that the proof of proposition could be much shorter if the last paragraph was dropped from statement. However this restriction is going to be an essential ingredient in obtaining effectiveness in \S \ref{escapeline}.
\begin{proof} Define a linear norm \begin{equation}\label{formh0}h_0(w)=\frac12\sum_{j=1}^{r_1+r_2}d_j|w_j|=\sum_{j=1}^{r_1+r_2}d_j\max(w_j,0)=-\sum_{j=1}^{r_1+r_2}d_j\min(w_j,0)\end{equation} on $W$ and let $\Omega=\{w\in W|h_0(w)\leq 1\}$. Observe $\forall w\in W,\forall j$, \begin{equation}\label{maxh0}|w_j|\leq h_0(w).\end{equation} Moreover, for a unit $t\in U_K$, $h^\mathrm{Mah}(t)=h_0(\mathcal L(t))$.

We also define a positive definite quadratic form $Q_0$ on $W$ by \begin{equation}\label{formQ0}Q_0(w,z)=\frac12\sum_{j=1}^{r_1+r_2}d_jw_jz_j\end{equation}.

If we denote $w_{r_1+r_2+j}=w_{r_1+j}$ for $1\leq j\leq r_2$ then $\sum_{j=1}^dw_d=0$, $Q_0(w,z)=\frac12\sum_{j=1}^{d}w_jz_j$, $h_0(w)=\sum_{j=1}^{d}\max(w_j,0)$. So \begin{equation}\label{Q0h0}\begin{split}Q_0(w,w)=&\frac12\sum_{j=1}^{d}w_j^2\\
=&\frac12(\sum_{1\leq j\leq d,w_j>0}w_j^2+\sum_{1\leq j\leq d,w_j<0}w_j^2)\\
\leq&\frac12\big((\sum_{1\leq j\leq d,w_j>0}w_j)^2+(\sum_{1\leq j\leq d,w_j<0}w_j)^2\big)\\
=&\frac12\big(h_0^2(w)+h_0^2(w)\big)\\
=&h_0^2(w).\end{split}\end{equation} And by Cauchy-Schwarz inequality $|Q_0(w,z)|\leq h_0(w)h_0(z)$.

If $V_i\cong\mathbb C$ and $F:=\sigma_i^{-1}(\mathbb R)$ is not $\mathbb Q$, then by Lemma \ref{maxreal} there are $1\leq k,l\leq r_1+r_2$ such that $k,l,i$ are distinct and for any unit $t$ from $U_F$, $\mathcal L(t)$ is in the hyperplane \begin{equation}\label{realunits}\forall t\in U_F,\ \mathcal L(t)\in W_{k,l}:=\{w_k=w_l\}. \end{equation}

If $V_i\cong\mathbb R$ or $\sigma_i^{-1}(\mathbb R)=\mathbb Q$, then we fix any pair $(k,l)$ so that $k,l,i$ are distinct, which is possible as $r_1+r_2=r+1\geq 3$. So in case that $V_i\cong\mathbb C$ and $F:=\sigma_i^{-1}(\mathbb R)=\mathbb Q$, as $U_F=\{\pm 1\}$, (\ref{realunits}) still holds.

Construct $a,b\in W$ by $$a_i=\frac1{d_i},\ a_j=-\frac1{d-d_i}, \forall j\neq i;$$ $$b_k=\frac1{d_k},\ b_k=-\frac1{d_l},\ b_j=0,\forall j\neq k,l.$$ It is easy to compute $h_0(a)=h_0(b)=1$, moreover $$Q_a:=Q_0(a,a)=\frac12(\frac1{d_i}+\frac1{d-d_i})\geq\frac14,\ Q_b:=Q_0(b,b)=\frac12(\frac1{d_k}+\frac1{d_l})\geq\frac12$$ and $Q_0(a,b)=0$.

Notice the hyperplane $W_{k,l}$ is the orthogonal complement to $b$ with respect to $Q_0$. So $\forall t\in U_F$, $Q_0(\mathcal L(t),b)=0$ if $V_i\cong C$, $F=\sigma_i^{-1}(\mathbb R)$.

$\phi(G)$, a finite-index subgroup of $U_K$, is embedded as a full-rank lattice in $W$ by $\mathcal L$. Let $0<m_1\leq m_2\leq \cdots\leq m_{r}$ be the successive minima of $\mathcal L(\phi(G))$ with respect to $\Omega$ then by definition $m_r=\mathcal F_{\phi(G)}$.

By a theorem of Jarn\'{i}k on inhomogeneous minimum (cf. \cite[p99]{GL87}), for any point $z$ in $W$, there is a lattice point $w$ from $\mathcal L(\phi(G))$ such that $w-z\in m^*\Omega$ where $m^*=\frac{\sum_{j=1}^{r}m_i}2\leq\frac 12r\mathcal F_{\phi(G)}$. Set \begin{equation}m=m^*+m_r\leq(\frac r2+1)\mathcal F_{\phi(G)}.\end{equation}

Taking $z$ to be $Za$ and $Zb$ respectively where \begin{equation}\label{4dminima}Z=4d(\frac r2+1)\mathcal F_{\phi(G)}\geq 4dm,\end{equation} we see there exist $a^*\in(Za+m^*\Omega)\cap \mathcal L(\phi(G))$ and $b^*\in(Zb+m^*\Omega)\cap \mathcal L(\phi(G))$.

By definition of successive minima there are $f_1,\cdots,f_r\in m_r\Omega$ that form a basis of $\mathcal L(\phi(G))$. Hence the $r+1$ vectors $a^*, a^*+f_1, \cdots, a^*+f_r\in Za+m^*\Omega+m_r\Omega\subset Za+m\Omega$ span $W$. At least one of them, denoted by $a'$, lies out of the hyperplane $W_{k,l}$.

Denote $b'=b^*\in Za+m^*\Omega\subset Za+m\Omega$. Write $a'=Za+m\epsilon_a$, $b'=Zb+m\epsilon_b$ where $h_0(\epsilon_a), h_0(\epsilon_b)\leq 1$. Recall $a'$, $b'$ are both in the lattice $\mathcal L(\phi(G))$.

First of all, remark $a'$ and $b'$ are linearly independent. Actually $|(\epsilon_a)_j|\leq h_0(\epsilon_a)=1, \forall j$, so
\begin{equation}\left\{\begin{aligned}
a'_i=&\frac Z{d_i}-m|(\epsilon_a)_i|\geq\frac Z2-m\geq\frac{2d-1}{4d}Z;\\
|a'_j|\leq &\frac Z{d-d_i}+m|(\epsilon_a)_j|\leq Z+m\leq\frac{4d+1}{4d}Z, \forall j\neq i.\end{aligned}\right.\end{equation} But similarly
\begin{equation}\left\{\begin{aligned}
|b'_i|=&m|(\epsilon_b)_i|\leq m\leq\frac 1{4d}Z;\\
|b'_k|\geq&\frac Z{d_k}-m|(\epsilon_b)_k|\geq\frac Z2-m\geq \frac {2d-1}{4d}Z.\end{aligned}\right.\end{equation}
Hence $\frac{|a'_i|}{|a'_k|}\geq\frac{2d-1}{4d+1}>\frac 1{2d-1}\geq\frac{|b'_i|}{|b'_k|}$ when $d\geq 3$, which implies linear independence.

Second, suppose the 2-dimensional subspace $\mathbb Ra'\oplus\mathbb Rb'$ has a non-empty intersection with $W_{k,l}=\{w|Q_0(w,b)=0\}$, then we claim that for all $w$ in this intersection, $|w_i|\gtrsim h_0(w)$. Suppose $w=W_aa'+W_bb'=W_aZa+W_am\epsilon_a+W_bZb+W_bm\epsilon_b$. Then since $Q_0(a,b)=0$ and by (\ref{Q0h0}), $|Q_0(\epsilon_a,b)|, |Q_0(\epsilon_a,a)|, |Q_0(\epsilon_b,b)|$ and $|Q_0(\epsilon_b,b)|$ are all bounded by 1, so \begin{align*}0=&Q_0(w,b)=Q_0(W_aZa+W_am\epsilon_a+W_bZb+W_bm\epsilon_b,b)\\
\geq&|W_bZQ_0(b,b)|-|W_amQ_0(\epsilon_a,b)|-|W_bmQ_0(\epsilon_b,b)|\\
\geq&|W_b|ZQ_b-|W_a|m-|W_b|m,\end{align*} thus \begin{equation}\label{WaWb}\frac{|W_b|}{|W_a|}\leq\frac{m}{ZQ_b-m}\leq\frac m{\frac12Z-m}\leq\frac m{2dm-m}=\frac 1{2d-1}.\end{equation} So \begin{equation}\label{Q0wabound}\begin{split}|Q_0(w,a)|=&|Q_0(W_aZa+W_am\epsilon_a+W_bZb+W_bm\epsilon_b,a)|\\
\geq&|W_aZQ_0(a,a)|-|W_amQ_0(\epsilon_a,a)|-|W_bmQ_0(\epsilon_b,a)|\\
\geq&|W_a|ZQ_a-|W_a|m-|W_b|m\\
\geq&\frac14|W_a|Z-|W_a|\cdot\frac Z{4d}-\frac{|W_a|}{2d-1}\cdot\frac Z{4d}\\
=&\frac{2d-3}{8d-4}|W_a|Z.\end{split}\end{equation}

But on the other hand \begin{equation}\label{Q0wa}\begin{split}2Q_0(w,a)=&d_i\cdot\frac1{d_i}w_i+\sum_{1\leq j\leq r_1+r_2,j\neq i}d_jw_j(-\frac1{d-d_i})\\
=&\frac d{d-d_i}w_i-d_iw_i\cdot \frac1{d-d_i}-\sum_{1\leq j\leq r_1+r_2,j\neq i}d_jw_j\cdot \frac1{d-d_i}\\
=&\frac d{d-d_i}w_i-\frac1{d-d_i}\sum_{j=1}^{r_1+r_2}d_jw_j\\
=&\frac d{d-d_i}w_i.\end{split}\end{equation}
Furthermore by (\ref{4dminima}) and (\ref{WaWb}), \begin{equation}\label{h0w}\begin{split}h_0(w)\leq& |W_a|Z+|W_b|Z+|W_a|m+|W_b|m\\
\leq& (|W_a|+\frac1{2d-1}|W_a|)(Z+\frac1{4d}Z)=\frac{4d+1}{4d-2}|W_a|Z.\end{split}\end{equation} As $d\geq 3$ we eventually deduce from (\ref{Q0wabound}), (\ref{Q0wa}) and (\ref{h0w}) that \begin{equation}\label{heightconcen}\begin{split}|w_i|=&2\cdot\frac{d-d_i}d|Q_0(w,a)|\\
\geq&2\cdot\frac{d-2}d\cdot\frac{2d-3}{8d-4}|W_a|Z\\
\geq&2\cdot\frac{d-2}d\cdot\frac{2d-3}{8d-4}\cdot\frac{4d-2}{4d+1}h_0(w)\\
=&\frac{(d-2)(2d-3)}{d(4d+1)}h_0(w)\\
\geq&\frac1{13}h_0(w).\end{split}\end{equation}

As $a',b'$ are in the lattice $\mathcal L(\phi(G))$, we can choose $u$ and $\tilde u$ as following:

First let $u$ be such that $\mathcal L(\phi(u))=a'$, then \begin{equation}\label{ucontrol}\left\{\begin{aligned}&h^\mathrm{Mah}(\phi(u))&=&h_0(a')\leq Z+m\leq(4d+1)(\frac r2+1)\mathcal F_{\phi(G)};\\
&\log|\zeta_u^i|&=&a'_i\geq\frac{2d-1}{4d}Z\geq (2d-1)(\frac r2+1)\mathcal F_{\phi(G)};\\
&\log|\zeta_u^j|&=&a'_j\leq Z(-\frac1{d-d_i})+m\\
&&\leq&(-\frac{4d}{d-d_i}+1)m\\
&&\leq&-3m<0, \forall j\neq i.\end{aligned}\right.\end{equation}
If $V_i\cong\mathbb C$, then $\zeta_{u^n}^i\notin\mathbb R$ as $\mathcal L(\phi(u^n))=na'\notin W_{k,l}$ because of the way $a'$ was chosen.

\noindent{\bf Case 1.} If $(\mathbb Za'\oplus\mathbb Zb')\cap W_{k,l}=\{0\}$. Choose $\tilde u\in G$ so that $\mathcal L(\phi(\tilde u))=a'+b'$. Then because $a'$ and $b'$ are linearly independent, $u$ and $\tilde u$ are multiplicatively independent.

In this case, by remarks we made at the beginning of proof, either $V_i$ is real or there is no non-trivial unit of the form $u^p\tilde u^q, (p,q)\neq(0,0)$ such that $\zeta_{u^p\tilde u^q}^i\in\mathbb R$ (because otherwise $pa'+q(a'+b')=p\mathcal L(\phi(u))+q\mathcal L(\phi(\tilde u))\in W_{k,l}$, but by linear independence it doesn't vanish, thus contradictis the assumption $(\mathbb Za'\oplus\mathbb Zb')\cap W_{k,l}=\{0\}$).

Notice $|b'_i|=|m(\epsilon_b))_i|\leq m\leq(\frac r2+1)\mathcal F_{\phi(G)}$. Then  \begin{equation}\label{tildeucontrol}\left\{\begin{aligned}&h^\mathrm{Mah}(\phi(\tilde u))&\leq&(Z+m)+(Z+m)\leq (8d+2)(\frac r2+1)\mathcal F_{\phi(G)},\\
&\log|\zeta_{\tilde u}^i|&=&a'_i+b'_i\geq a'_i-|b'_i|\geq(2d-2)(\frac r2+1)\mathcal F_{\phi(G)}.\end{aligned}\right.\end{equation}

\noindent{\bf Case 2.} If $(\mathbb Za'\oplus\mathbb Zb')\cap W_{k,l}\neq\{0\}$ then it is isomorphic to either $\mathbb Z$ or $\mathbb Z^2$. Choose $u$ so that $\mathcal L(\phi(u))=a'$. But as $a'\notin W_{k,l}$, $(\mathbb Ra'\oplus\mathbb Rb')\cap W_{k,l}$ has dimension $1$ and $(\mathbb Za'\oplus\mathbb Zb')\cap W_{k,l}$ has to be cyclic. Let $a''$ be a generator, then by (\ref{heightconcen}), $|a''_i|>\frac1{13}h^\mathrm{Mah}(a'')>0$, without loss of generality assume $a''_i>0$. 

In this case, if there is $\zeta_{u^p\tilde u^q}^i\in\mathbb R$ for a non-trivial pair $(p,q)$ then $p\mathcal L(\phi(u))+q\mathcal L(\phi(\tilde u))\in(\mathbb Za'\oplus\mathbb Zb')\cap W_{k,l}$, thus has to be a multiple of $a''_i$. 

{\bf Case 2.i.} If $h^\mathrm{Mah}(a'')\geq Z$ then we choose $\tilde u$ so that $\mathcal L(\phi(\tilde u))=a'+b'$.  Then $u$ and $\tilde u$ are always multiplicatively independent and the inequalities (\ref{tildeucontrol}) still hold as in case 1. Moreover, if $\zeta_{u^p\tilde u^q}^i\in\mathbb R$ and $|\zeta_{u^p\tilde u^q}^i|>1$, then $\log|\zeta_{u^p\tilde u^q}^i|=p\mathcal L(\phi(u))+q\mathcal L(\phi(\tilde u))=na''_i$ for some $n\in \mathbb N$, thus \begin{equation}\label{nonisometric}\log|\zeta_{u^p\tilde u^q}^i|\geq a''_i\geq\frac1{13}h^\mathrm{Mah}(a'')\geq\frac1{13}Z\geq\frac d4(\frac r2+1)\mathcal F_{\phi(G)}.\end{equation}

{\bf Case 2.ii.} If $h^\mathrm{Mah}(a'')<Z$ then there is $N\in\mathbb N$ such that $Nh^\mathrm{Mah}(a'')\in (Z,2Z]$. There exists $\tilde u\in G$ such that  $\mathcal L(\phi(\tilde u))=Na''$. Then $u,\tilde u$ are multiplicatively independent because $a'\notin W_{k,l}$ and $Na''\in W_{k,l}$ are linearly independent.

In this case, if $\zeta_{u^p\tilde u^q}^i$ is real with $|\zeta_{u^p\tilde u^q}^i|\geq 1$ then $\log|\zeta_{u^p\tilde u^q}^i|=pa'+qNa''\in W_{k,l}$. Because $a'\notin W_{k,l}$ and $a''\in W_{k,l}$, $p=0$ and $q>0$. Thus \begin{equation}\log|\zeta_{u^p\tilde u^q}^i|\geq Na''_i\geq \frac1{13}Nh^\mathrm{Mah}(a'')\geq\frac1{13}Z\geq\frac d4(\frac r2+1)\mathcal F_{\phi(G)}.\hskip2cm(\ref{nonisometric}\text{'})\nonumber\end{equation}\addtocounter{equation}{1} Further, \begin{equation}\hskip1cm\left\{\begin{aligned}&h^\mathrm{Mah}(\phi(\tilde u))&=&Nh^\mathrm{Mah}(a'')\leq 2Z\leq 8d(\frac r2+1)\mathcal F_{\phi(G)},\\&\log|\zeta_{\tilde u}^i|&=&Na''_i\geq\frac d4(\frac r2+1)\mathcal F_{\phi(G)}.\end{aligned}\hskip2.6cm (\ref{tildeucontrol}\text{'})\right.\nonumber\end{equation}\addtocounter{equation}{1}

The proposition is established by combining (\ref{ucontrol}), (\ref{tildeucontrol}), (\ref{tildeucontrol}'), (\ref{nonisometric}) and (\ref{nonisometric}').\end{proof}

\subsection{Approximation of an arithmetic progression}

\hskip\parindent We will construct a sequence inside $\{\zeta_g^i|g\in G\}$ which resembles an arithmetic progression. To achieve this the following important result from Diophantine geometry is needed:

\begin{theorem}\label{BWthm} {\rm(Baker-W\"ustholz\cite{BW93})} Suppose $\alpha_k\in\mathbb C, k=1,\cdots,N$ are non-zero algebraic numbers belonging to the same degree $d$ number field and $\vartheta_k\in\mathbb C$ is a fixed natural logarithm of $\alpha_k$ for all $k$. Let $$h'_k=\max(h^\mathrm{Mah}(\alpha_k),|\log\vartheta_k|,1).$$ Suppose a non-zero integer vector $(p_1,p_2,\cdots,p_N)\in\mathbb Z^N$ satisfies $\sum_{k=1}^Np_k\vartheta_k\neq 0$, then $$-\log |\sum_{k=1}^Np_k\vartheta_k|\lesssim_{N,d}\big(\prod_{k=1}^Nh'_k\big)\max(\log|p_1|,\log|p_2|,\cdots,\log|p_N|,\frac1d),$$ where the implied constant is explicitly determined by $N$ and $d$.\end{theorem}

Observe that, if $\mathrm{Im}\vartheta_k\in[0,2\pi]$, then as $\log\vartheta_k=\log|\alpha_k|+\mathrm{Im}\vartheta_k$, $h'_k\lesssim\max(h^\mathrm{Mah}(\alpha_k),\log|\alpha_k|,1)$. However $\log|\alpha_k|\leq \sum_{i=1}^d\log_+|\sigma_i(\alpha_k)|=h^\mathrm{Mah}(\alpha_k)$. Hence (a special case of) Baker-W\"ustholz Theorem can be reformulated as:

\begin{lemma}\label{BWlog}Suppose $\alpha_k\in\mathbb C, k=1,\cdots,N$ are non-zero algebraic numbers belonging to the same degree $d$ number field and $\vartheta_k\in\mathbb C$ is a fixed natural logarithm of $\alpha_k$ with $\mathrm{Im}\vartheta_k\in[0,2\pi]$ for all $k=1,\cdots,N$. For any non-zero integer vector $(p_1,p_2,\cdots,p_N)\in\mathbb Z^N$ such that $\sum_{k=1}^Np_k\vartheta_k\neq 0$, $$-\log |\sum_{k=1}^Np_k\vartheta_k|\lesssim_{N,d} \big(\prod_{k=1}^N\max(h^\mathrm{Mah}(\alpha_k),1)\big)\log\max(|p_1|,|p_2|,\cdots,|p_N|,2),$$ where the implied constant is an explicit function in $N$ and $d$.\end{lemma}

\declareconst{arithprolengthconst} \declareconst{arithproexpoconst}
\begin{proposition}\label{arithpro} For some effective constants $\ref{arithprolengthconst}(d)$ and $\ref{arithproexpoconst}(d)$, for all integers $s\geq\ref{arithprolengthconst}\mathcal F_{\phi(G)}$, there exist $a_0,a_1,\cdots,a_{s-1}\in G$ and a number $\Delta$ (which is either real or complex depending on whether $V_i\cong\mathbb R$ or $\mathbb C$) such that:

(1). $s^{-\ref{arithproexpoconst}\mathcal F_{\phi(G)}^2}\leq|\Delta|\leq s^{-3}$;

(2). $h^\mathrm{Mah}(\phi(a_t))\leq s^{10}, \forall t$;

(3). $|\zeta_{a_t}^i-(1+t\Delta)|\leq s^{-1}|\Delta|, \forall t$.\end{proposition}

\begin{proof} Let $u$, $\tilde u$ be defined as in Proposition \ref{expand}. Denote $\zeta_u^i=e^{\theta+2\pi\beta\mathrm i},\zeta_{\tilde u}^i=e^{\tilde \theta+2\pi\tilde \beta\mathrm i}$, where $\theta,\tilde\theta>0$ and $\beta,\tilde\beta\in[0,1)$. 

Let $\gamma=\theta/{\tilde\theta}$. Since $\frac{\tilde\theta}{\ln 2}=\log|\zeta_{\tilde u}^i|\geq\frac14d(\frac r2+1)\mathcal F_{\phi(G)}$ and $\frac{\theta}{\ln 2}=\log|\zeta_{u}^i|\leq h^\mathrm{Mah}(\phi(u))\leq 9d(\frac r2+1)\mathcal F_{\phi(G)}$, $\gamma\leq 36$.

Then  $\forall n,m\in\mathbb N$, the $i$-th eigenvalue of $u^n\tilde u^{-m}$ is given by \begin{equation}\label{moduargu}\begin{split}\zeta_{u^n\tilde u^{-m}}^i=&e^{(\theta n-\tilde\theta m)+2\pi(\beta n-\tilde\beta m)\mathrm i}\\
=&e^{(\gamma\tilde\theta n-\tilde\theta m)+2\pi(\beta n-\gamma\tilde\beta n+\tilde\beta(\gamma n-m))\mathrm i}\\
=&e^{(\tilde\theta+2\pi\tilde\beta\mathrm i)(\gamma n-m)+2\pi(\beta-\gamma\tilde\beta)n\mathrm i}.\end{split}\end{equation}

Let $\Delta_0=(\gamma,\beta-\gamma\tilde\beta)\in\mathbb R^2$. We slightly abuse notation and denote the projection of $\Delta_0$ in $\mathbb T^2=(\mathbb R/\mathbb Z)^2$ by $\Delta_0$ as well. Then $\forall n\in\mathbb N$, $n\Delta_0\in\mathbb T^2$ is represented by $(\gamma n,(\beta-\gamma\tilde\beta)n)$. For all $P\geq 10$, $\mathbb T^2$ can be covered by at most $P$ squares of side length $\frac{1}{\lfloor\sqrt P\rfloor}\leq\frac 2{\sqrt P}$. By pigeonhole principle there exist two different $n_1,n_2\in\{0,1,\cdots,P\}$ such that $n_1\Delta_0$ and $n_2\Delta_0$ are in the same square. Let $n=|n_1-n_2|$, then $1\leq n\leq P$ and $n\Delta_0$ is in the neighborhood $\pi\big([-\frac 2{\sqrt P},\frac 2{\sqrt P}]^2\big)$ of $(0,0)\in\mathbb T^2$.

Now we are going to effectively bound the distance between $n\Delta_0$ and the origin using Baker-W\"usholtz Theorem (Lemma \ref{BWlog}). We know for some $(m_1,m_2)\in\mathbb Z^2$,  \begin{equation}\omega_1:=|\gamma n-m_1|\leq\frac 2{\sqrt P}, \omega_2:=|(\beta-\gamma\tilde\beta)n-m_2|\leq\frac 2{\sqrt P}.\end{equation} As $\frac 2{\sqrt P}\leq 1$,  \begin{equation}\label{m1m2bound}|m_1|\leq|\gamma n|+1\leq\gamma P+1;\ |m_2|\leq|(\beta-\gamma\tilde\beta)n|+1\leq |\beta-\gamma\tilde\beta|P+1.\end{equation} Notice $$|\beta n-\tilde\beta m_1-m_2|=|\tilde\beta(\gamma n-m_1)+(\beta -\tilde\beta \gamma) n-m_2|\leq\tilde\beta\omega_1+\omega_2.$$ Multiply both sides by $2\pi$, we get $$|n\cdot 2\pi\beta \mathrm i-m_1\cdot 2\pi\tilde\beta\mathrm i-m_2\cdot 2\pi i|\leq 2\pi\tilde(\beta\omega_1+\omega_2).$$ On the other hand, $$|n\theta-m_1\tilde\theta|=|\tilde\theta(\gamma n-m_1)|\leq\tilde\theta\omega_1.$$
Add these two inequalities together: 
\begin{equation}\label{arithproBWsum}|n(\theta+2\pi\beta\mathrm i)-m_1(\tilde\theta+2\pi\tilde\beta\mathrm i)-m_2\cdot2\pi\mathrm i|\leq(\tilde\theta+2\pi\tilde\beta)\omega_1+2\pi\omega_2.\end{equation}

But $\theta+2\pi\beta\mathrm i$, $\tilde\theta+2\pi\tilde\beta\mathrm i$ and $2\pi\mathrm i$ are respectively natural logarithms of $\zeta_u^i$,$\zeta_{\tilde u}^i$ and $1$.

Moreover we claim $n(\theta+2\pi\beta\mathrm i)-m_1(\tilde\theta+2\pi\tilde\beta\mathrm i)-m_2\cdot2\pi\mathrm i\neq 0$. Actually if this is not the case then $\zeta_{u^n\tilde u^{m_1}}^i=(\zeta_u^i)^n(\zeta_{\tilde u}^i)^{m_1}=1$, but this would imply all of eigenvalues of $u^n\tilde u^{m_1}$ are equal to 1 since they are conjugate to each other. Thus $u^n\tilde u^{m_1}$ is identity, so $n=m_1=0$ by the multiplicative independence; but $n$ is supposed to be positive.

Apply Lemma \ref{BWlog} to (\ref{arithproBWsum}):
\begin{equation}\label{BakerAP}\begin{aligned}
&&&-\log\big((\tilde\theta+2\pi\tilde\beta)\omega_1+2\pi\omega_2\big)\\
&\leq&&-\log|n(\theta+2\pi\beta\mathrm i)-m_1(\tilde\theta+2\pi\tilde\beta\mathrm i)-m_2\cdot2\pi\mathrm i|\\
&\lesssim_d&&\max(h^\mathrm{Mah}(\zeta_u^i),1)\max(h^\mathrm{Mah}(\zeta_{\tilde u}^i),1)\log\max(|n|,|m_1|,|m_2|,2)\\
&=&&\max(h^\mathrm{Mah}(\phi(u)),1)\max(h^\mathrm{Mah}(\phi(\tilde u)),1)\\
&&&\hskip3cm\cdot\log\max(|n|,|m_1|,|m_2|,2)\\
&\lesssim_{d}&&\mathcal F_{\phi(G)}^2\log\max(P,\gamma P+1,|\beta-\gamma\tilde\beta|P+1,2).\end{aligned}\end{equation}
We used in the last inequality the fact that $h^\mathrm{Mah}(\phi(u)),h^\mathrm{Mah}(\phi(\tilde u))\lesssim_{d}\mathcal F_{\phi(G)}$ from Proposition \ref{expand} (using $r\leq d-1$) and $\mathcal F_{\phi(G)}\geq\ref{voutierconst}(d)$ (by Remark \ref{voutier}).

When $P\geq 1$, $\max(P,\gamma P+1,|\beta-\gamma\tilde\beta|P+1, 2)\leq 40P$ as $\gamma\leq 36$ and $\beta,\tilde \beta<1$, thus (\ref{BakerAP}) implies there is an effective constant \declaresubconst{arithprosubconst} $\ref{arithprosubconst}=\ref{arithprosubconst}(d)$ such that  \begin{equation}\label{expBakerAP}(\tilde\theta+2\pi\tilde\beta)\omega_1+2\pi\omega_2\geq P^{-\ref{arithprosubconst}\mathcal F_{\phi(G)}^2}.\end{equation} 

Define a sequence \begin{equation}a_t=(u^n\tilde u^{-m_1})^t\in G,\ t=0,1,\cdots.\end{equation} By (\ref{moduargu}), $\zeta_{a_t}^i=e^{t\Delta}$ with $\Delta={(\tilde\theta+2\pi\tilde\beta\mathrm i)(\gamma n-m_1)+2\pi((\beta-\gamma\tilde\beta)n-m_2)\mathrm i}$. Equivalently \begin{equation}\label{Deltadef}\left\{\begin{aligned}&\mathrm{Re}\Delta=\tilde\theta(\gamma n-m_1)\\&\mathrm{Im}\Delta=2\pi\tilde\beta(\gamma n-m_1)+2\pi((\beta-\gamma\tilde\beta)n-m_2).\end{aligned}\right.\end{equation}

We have two possibilities: either $(\tilde\theta+2\pi\tilde\beta)\omega_1\geq\frac13((\tilde\theta+2\pi\tilde\beta)\omega_1+2\pi\omega_2)$. Then
$|\mathrm{Re}\Delta|=\tilde\theta\omega_1\geq\frac{\tilde\theta}{\tilde\theta+2\pi\tilde\beta}\cdot\frac13((\tilde\theta+2\pi\tilde\beta)\omega_1+2\pi\omega_2)$. Since
$\tilde\beta\in[0,1)$, $\tilde\theta=\log|\zeta_{\tilde u}^i|\gtrsim_{d}\mathcal F_{\phi(G)}\gtrsim_d 1$, we conclude by (\ref{expBakerAP}) in this case that $$|\mathrm{Re}\Delta|\gtrsim_{d}P^{-\ref{arithprosubconst}\mathcal F_{\phi(G)}^2}.$$

Or $2\pi\omega_2>\frac23((\tilde\theta+2\pi\tilde\beta)\omega_1+2\pi\omega_2)$, then $2\pi\tilde\beta\omega_1\leq\frac{2\pi\tilde\beta}{\tilde\theta+2\pi\tilde\beta}\cdot \frac13((\tilde\theta+2\pi\tilde\beta)\omega_1+2\pi\omega_2)\leq \frac13((\tilde\theta+2\pi\tilde\beta)\omega_1+2\pi\omega_2).$ And $$|\mathrm{Im}\Delta|\geq 2\pi\omega_2-2\pi\tilde\beta\omega_1\geq \frac13((\tilde\theta+2\pi\tilde\beta)\omega_1+2\pi\omega_2)\geq \frac13P^{-\ref{arithprosubconst}\mathcal F_{\phi(G)}^2}.$$

So in any case, there is an effective constant \declaresubconst{arithprosubconst1}$\ref{arithprosubconst1}(d)$ such that  \begin{equation}\label{Deltalower} P^{-\ref{arithprosubconst}\mathcal F_{\phi(G)}^2}\leq\ref{arithprosubconst1}|\Delta|.\end{equation}

On the other hand, because $\omega_1,\omega_2\leq\frac2{\sqrt P}$, $\tilde\beta<1$, and $$\tilde\theta\leq h^\mathrm{Mah}(\phi(\tilde u))\lesssim_d\mathcal F_{\phi(G)},$$ by (\ref{Deltadef}) there is some \declaresubconst{arithprosubconst2}$\ref{arithprosubconst2}(d)$ such that \begin{equation} \label{Deltaupper}|\Delta|\leq\ref{arithprosubconst2}\mathcal F_{\phi(G)}P^{-\frac12}.\end{equation}

Let $P=s^8$ where $s$ is an integer such that $s\geq\ref{arithprolengthconst}\mathcal F_{\phi(G)}$ where $\ref{arithprolengthconst}=\ref{arithprolengthconst}(d)\geq \max(\ref{arithprosubconst2},\frac{\ref{arithprosubconst1}}{\ref{voutierconst}},1)$. 

Now we verify all of the three claims from proposition.

(1). By (\ref{Deltalower}) there exists an effective constant $\ref{arithproexpoconst}=\ref{arithproexpoconst}(d)$ such that \begin{align*}|\Delta|\geq&\ref{arithprosubconst1}^{-1}P^{-\ref{arithprosubconst}\mathcal F_{\phi(G)}^2}\geq\frac1{\ref{voutierconst}\ref{arithprolengthconst}}P^{-\ref{arithprosubconst}\mathcal F_{\phi(G)}^2}\geq s^{-1}s^{-8\ref{arithprosubconst}\mathcal F_{\phi(G)}^2}\\
\geq&s^{-\ref{arithproexpoconst}\mathcal F_{\phi(G)}^2},\end{align*} where we used $\mathcal F_{\phi(G)}\geq\ref{voutierconst}$.

On the other hand by as $s\geq \ref{arithprosubconst2}\mathcal F_{\phi(G)}$. (\ref{Deltaupper}) implies $|\Delta|\leq sP^{-\frac12}=s^{-3}.$

(2). We take only the first $s$ terms in the sequence $\{a_t\}$. For all $0\leq t\leq s-1$, \begin{align*}h^\mathrm{Mah}(\phi(a_t))&=\ h^\mathrm{Mah}(\phi((u^n\tilde u^{-m_1})^t))\leq t\big (|n|h^\mathrm{Mah}(\phi( u))+|m_1|h^\mathrm{Mah}(\phi(\tilde u))\big )\\
&\lesssim_ds(|n|+|m_1|)\mathcal F_{\phi(G)}.\end{align*}
As $1\leq n\leq P$, by (\ref{m1m2bound}) $h^\mathrm{Mah}(\phi(a_t))\lesssim_{d}sP\mathcal F_{\phi(G)}=s^9\mathcal F_{\phi(G)}$. When $\ref{arithprolengthconst}$ is made sufficiently large (but still effective), $h^\mathrm{Mah}(\phi(a_t))\leq s^{10}.$ 

(3). Finally, if $|\Delta|\leq s^{-3}$ and $t<s$ then $|t\Delta|<1$. So $$|\zeta_{a_t}^i-(1+t\Delta)|=|e^{t\Delta}-(1+t\Delta)|=\Big|\sum_{k=2}^\infty\frac{(t\Delta)^k}{k!}\Big|\leq(\sum_{k=2}^\infty\frac1{k!})|t\Delta|^2\leq|t\Delta|^2,$$ which is bounded by $s^2|\Delta|^2=(s^2|\Delta|)|\Delta|\leq s^{-1}|\Delta|$. This completes the proof.\end{proof}

\subsection{Escape from a fixed line}\label{escapeline}

\hskip\parindent In the last section we obtained a collection of eigenvalues of different elements from $G$ in $V_i$ which approximate the arithmetic progression $\{1+t\Delta\}_{t=1}^s\subset\mathbb C$. The direction of this progression is given by $\Delta\in C$.

When $V_i\cong\mathbb C$, it will turn out in a later part of this paper that in certain situations it may appear that some directions in $\mathbb C$ are ``bad'' in the sense that while embeded in $X=(\mathbb R^{r_1}\times\mathbb C^{r_2})/\Gamma$, they don't have the desired equidistribution property. So we will have to find a way to move our arithmetic progressions away from these exceptional directions, which is going to be the topic of the following discussion.

\declareconst{nearlinebdheightconst} \declareconst{nearlinebdexpoconst}
\begin{proposition}\label{nearlinebd} There exist effective constants $\ref{nearlinebdheightconst}(d)$ and $\ref{nearlinebdexpoconst}(d)$ such that $\forall l\geq 2$, there are $b_1,b_2,\cdots,b_l\in G$ satisfying

(1). $\forall k=1,\cdots,l$, $h^\mathrm{Mah}(\phi(b_k))\leq\ref{nearlinebdheightconst}\mathcal F_{\phi(G)}^2l$;

(2). $1\leq|\zeta_{b_k}^i|\leq l^{\ref{nearlinebdexpoconst}\mathcal F_{\phi(G)}^2}$, $\forall k=1,\cdots,l$;

(3). For any non-zero $\mathbb R$-linear form $f$ on $V_i$, the number of the $b_k$'s such that $f(\zeta_{b_k}^i)\leq\|f\|$ is bounded by $100$.\end{proposition}

We need a couple of lemmas to prove this result.

\begin{lemma}\label{nearargu} Assume $V\cong\mathbb C$ and $f$ is a non-zero real form on $V$. For $\lambda>0$, if complex numbers $v$, $w$ satisfy $|v|,|w|\geq c>0$ and $|f(v)|,|f(w)|\leq\lambda\|f\|$ then $\arg v$ and $\arg w$ are $\frac{\pi\lambda}c$-close up to a multiple of $\pi$, i.e. $|\arg v-\arg w-m\pi|<\frac{\pi\lambda}c$ for some  $m\in\mathbb Z$.\end{lemma}

\begin{proof} Without loss of generality, it suffices to prove in the case where $\|f\|=1$.

Suppose $f(v)=f_1\mathrm{Re}v+f_2\mathrm{Im}v$, then $f_1^2+f_2^2=1$. Let $z=f_2+f_1\mathrm i$. Observe $$f(v)=\mathrm{Im}(zv)=|zv||\sin(\arg z+\arg v)|=|v||\sin(\arg z+\arg v)|.$$ So $|\sin(\arg z+\arg v)|\leq\frac\lambda{|v|}\leq\frac\lambda c$. Take $m_v\in\mathbb Z$ such that $\arg z+\arg v-m_v\pi\in[-\frac\pi2,\frac\pi2)$ then $|\arg z+\arg v-m_v\pi|\leq\frac\pi2|\sin(\arg z+\arg v-m_v\pi)|=\frac\pi2|\sin(\arg z+\arg v)|\leq\frac{\pi\lambda}{2 c}$. For the same reason $|\arg z+\arg w-m_w\pi|\leq\frac{\pi\lambda}{2 c}$ for some integer $m_w$. Take the difference, we get $|\arg v-\arg w-(m_v-m_w)\pi|\leq\frac{\pi}c$.\end{proof}

\begin{lemma}\label{absvalextension} Suppose $F\subset\mathbb C$ is a number field of degree $d$ which is already embedded into $\mathbb C$ and $\zeta, \rho\in F$ then there is a number field $\tilde F\subset \mathbb C$ of degree no more than $4d^2$ such that $\zeta/|\zeta|, \rho/|\rho|\in\tilde F$.\end{lemma}

\begin{proof}$\overline\zeta$ and $\overline\rho$ belong to $\overline F\subset \mathbb C$, the complex conjugate of $F$ which has degree $d$ as well. The number field $F'$ generated by $F$ and $\overline F$ together is of degree no more than $d^2$. Since $|\zeta|^2=\zeta\overline\zeta\in F'$ we see $|\zeta|$ is a quadratic element over $F'$, and so is $|\rho|$. Hence the extension $\tilde F:=F'(|\zeta|,|\rho|)$ of $F$ is of degree at most $4$. So $D:=|\tilde F:\mathbb Q|\leq 4d^2$. It is clear that $\zeta/|\zeta|, \rho/|\rho|\in\tilde F$.\end{proof}

\begin{proof}[Proof of Proposition \ref{nearlinebd}] \noindent{\bf Case 1 }(trivial case). When $V_i\cong\mathbb R$, it suffices to take a single element $u\in G$, which was constructed in Proposition \ref{expand}. Let $b_k=u, \forall 1\leq k\leq l$. Then (1) is easily verified for sufficiently large $\ref{nearlinebdheightconst}(d)$ since $h^\mathrm{Mah}(\phi(b_k))=h^\mathrm{Mah}(\phi(u))\lesssim_d\mathcal F_{\phi(G)}$ and $\mathcal F_{\phi(G)}\gtrsim_{d}1$. We know $1<|\zeta_u^i|\leq 2^{9d(\frac r2+1)\mathcal F_{\phi(G)}}$, so (2) is satisfied if $\ref{nearlinebdexpoconst}=\ref{nearlinebdexpoconst}(d)$ is large enough. Last, since any real linear form $f$ is actually a scalar multiplication, $|f(\zeta_{b_k}^i)|=|f(\zeta_u^i)|=\|f\|\cdot|\zeta_u^i|>\|f\|$ for all $k$, part (3) follows.

\noindent{\bf Case 2} (main purpose). Suppose from now on $V_i\cong\mathbb C$. Take $u$ and $\tilde u$ from Proposition \ref{expand}. As in the proof of Proposition \ref{arithpro}, set $\zeta_u^i=e^{\theta+2\pi\beta\mathrm i},\zeta_{\tilde u}^i=e^{\tilde \theta+2\pi\tilde \beta\mathrm i}$, where $\theta,\tilde\theta>0$ and $\beta,\tilde\beta\in[0,1)$. Let $\theta=\gamma {\tilde\theta}$ then $\gamma\leq 36$ as before.

Construct a sequence of group elements \declaresubconst{nearlinebdsubconst}\begin{equation}b_k=u^{k+J}\tilde u^{-\lceil \gamma k\rceil},\  k=1,\cdots,l,\end{equation} where \begin{equation}J:=\lceil\theta^{-1}\ref{nearlinebdsubconst}\mathcal F_{\phi(G)}^2\ln l\rceil,\end{equation} $\ref{nearlinebdsubconst}(d)$ being a sufficiently large effective constant to be decided later. Moreover we are going to suppose \begin{equation}\label{lassumption}l^{\frac12\ref{nearlinebdsubconst}\mathcal F_{\phi(G)}^2}\geq e^{\tilde\theta};\end{equation}
if $\ref{nearlinebdsubconst}$ is large enough and $l\geq 2$, then this is always valid because $\tilde\theta\lesssim_{d}\mathcal F_{\phi(G)}$.

Since $u$ and $\tilde u$ are multiplicatively independent, the $b_k$'s are distinct. The action of $b_k$ on $V_i$ is characterized by \begin{equation}\label{bkpolar}\zeta_{b_k}^i=e^{(k\theta+J\theta-\lceil \gamma k\rceil\tilde\theta)+2\pi(k\beta+J\beta-\lceil \gamma k\rceil\tilde\beta)\mathrm i}.\end{equation}

Now we check  Proposition \ref{nearlinebd} part by part.

(1). Observe $h^\mathrm{Mah}(\phi(b_k))\leq (k+J)h^\mathrm{Mah}(\phi(u))+\lceil \gamma k\rceil h^\mathrm{Mah}(\phi(\tilde u))$. Moreover,  $\theta=\ln|\zeta_u^i|\gtrsim_{d}\mathcal F_{\phi(G)}$ effectively, which implies $J\lesssim_d\mathcal F_{\phi(G)}\ln l$. Since $h^\mathrm{Mah}(\phi(u)),h^\mathrm{Mah}(\phi(\tilde u))\lesssim_{d}\mathcal F_{\phi(G)}$,  $k\leq l$ and $\gamma\leq 36$, we conclude $$h^\mathrm{Mah}(\phi(b_k))\lesssim_{d} \mathcal F_{\phi(G)}^2l,$$ with an effective implied constant $\ref{nearlinebdheightconst}=\ref{nearlinebdheightconst}(d)$ for all $1\leq k\leq l$.

(2). Because $\lceil \gamma k\rceil\tilde\theta\in[\gamma k\tilde\theta,\gamma k\tilde\theta+\tilde\theta)=[k\theta,k\theta+\tilde\theta)$ and $$J\theta\in[\ref{nearlinebdsubconst}\mathcal F_{\phi(G)}^2\ln l,\ref{nearlinebdsubconst}\mathcal F_{\phi(G)}^2\ln l+\theta),$$ by (\ref{bkpolar}) $\ln|\zeta_{b_k}^i|$ is \begin{equation}\label{band}k\theta+J\theta-\lceil \gamma k\rceil\tilde\theta\in[\ref{nearlinebdsubconst}\mathcal F_{\phi(G)}^2\ln l-\tilde\theta,\ref{nearlinebdsubconst}\mathcal F_{\phi(G)}^2\ln l+\theta).\end{equation} So $$\log|\zeta_{b_k}^i|\lesssim_{d}\mathcal F_{\phi(G)}^2\ln l$$ and the second condition is satisfied by some effective $\ref{nearlinebdexpoconst}$.

(3). Suppose there are more than one $b_k$ such that $f(\zeta_{b_k}^i)\leq\|f\|$ (otherwise we are done). Apply Lemma \ref{nearargu} to two such elements $\zeta_{b_k}^i$ and $\zeta_{b_{k'}}^i$, by (\ref{band}) there is $m\in\mathbb Z$ such that $\big|2\pi(k\beta+J\beta-\lceil \gamma k\rceil\tilde\beta)-2\pi(k'\beta+J\beta-\lceil \gamma k'\rceil\tilde\beta)-m\pi\big|$ is at most $$\frac\pi{\min(|\zeta_{b_k}^i|,|\zeta_{b_{k'}}^i|)}\leq \pi e^{-(\ref{nearlinebdsubconst}\mathcal F_{\phi(G)}^2\ln l-\tilde\theta)}=\pi e^{\tilde\theta}l^{-\ref{nearlinebdsubconst}\mathcal F_{\phi(G)}^2}.$$

If we denote $n_1=k-k'$ and $n_2=-\lceil \gamma k\rceil+\lceil \gamma k'\rceil$, this rewrites \begin{equation}\label{nearlinebdBWsum}|n_1\cdot2\pi\beta\mathrm i+n_2\cdot2\pi\tilde\beta\mathrm i-m\cdot\pi\mathrm i|\leq\pi e^{\tilde\theta}l^{-\ref{nearlinebdsubconst}\mathcal F_{\phi(G)}^2}.\end{equation}

Assume first $n_1\cdot 2\pi\beta+n_2\cdot 2\pi\tilde\beta-m\pi\neq0$. 

It is easy to see \begin{equation}\label{n1n2bound}|n_1|\leq l;\ |n_2|\leq\gamma l+1\leq 40l,\end{equation} (as $\gamma\leq 36$). Hence by assumption (\ref{lassumption}), \begin{equation}\label{n1n2mbound}\begin{split}|m|\leq&\pi^{-1}(2\pi|n_1|+2\pi|n_2|+\pi e^{\tilde\theta}l^{-\ref{nearlinebdsubconst}\mathcal F_{\phi(G)}^2})\\\leq& \pi^{-1}(2\pi|n_1|+2\pi|n_2|+\pi)\leq100l.\end{split}\end{equation}

Notice $2\pi\beta\mathrm i$, $2\pi\tilde\beta\mathrm i$ and $\pi i$ are respectively natural logarithms of the numbers $\zeta_u^i/{|\zeta_u^i|}$, $\zeta_{\tilde u}^i/|\zeta_{\tilde u}^i|$ and $-1$.

By Lemma \ref{absvalextension}, $\zeta_u^i/{|\zeta_u^i|}$ and $\zeta_{\tilde u}^i/|\zeta_{\tilde u}^i|$ are algebraic numbers in a number field $\tilde K\subset\mathbb C$ of degree $D\leq 4d^2$.

Let $\tilde h^\mathrm{Mah}$ be the logarithmic Mahler measure on $\tilde K$. Then for $\lambda\in K$, the absolute logarithmic height $h(d)=\frac{\tilde h^\mathrm{Mah}(\lambda)}{D}=\frac{h^\mathrm{Mah}(\lambda)}d$. Hence $\tilde h^\mathrm{Mah}(\zeta_u^i)=\tilde h^\mathrm{Mah}(\overline{\zeta_u^i})=\frac Ddh^\mathrm{Mah}(\phi(u))\leq\frac{4d^2}d
\cdot 9d(\frac r2+1)\mathcal F_{\phi(G)}\lesssim_d\mathcal F_{\phi(G)}$ and by (\ref{heightrules}) $\tilde h^\mathrm{Mah}(|\zeta_u^i|)=\frac12\tilde h^\mathrm{Mah}(\zeta_u^i\overline{\zeta_u^i})\leq \frac12(\tilde h^\mathrm{Mah}(\zeta_u^i)+\tilde h^\mathrm{Mah}(\overline{\zeta_u^i}))\lesssim_d\mathcal F_{\phi(G)}$. So $\tilde h^\mathrm{Mah}(\zeta_u^i/|\zeta_u^i|)\lesssim_d\mathcal F_{\phi(G)}$, similarly $\tilde h^\mathrm{Mah}(\zeta_{\tilde u}^i/|\zeta_{\tilde u}^i|)\lesssim_d\mathcal F_{\phi(G)}$. All the implied constants are effective here.

So Lemma \ref{BWlog} applies, by (\ref{nearlinebdBWsum}), \begin{equation}\label{Bakerescape}\begin{aligned}
&&&-\log(\pi e^{\tilde\theta}l^{-\ref{nearlinebdsubconst}\mathcal F_{\phi(G)}^2})\\
&\leq&&-\log|n_1\cdot2\pi\beta\mathrm i+n_2\cdot2\pi\tilde\beta\mathrm i-m\cdot\pi\mathrm i|\\
&\lesssim_D&&\max(\tilde h^\mathrm{Mah}(\frac{\zeta_u^i}{|\zeta_u^i|}),1)\max (\tilde h^\mathrm{Mah}(\frac{\zeta_{\tilde u}^i}{|\zeta_{\tilde u}^i|}),1)\\
&&&\hskip3cm\cdot\log\max(|n_1|,|n_2|,|m|,2)\\
&\lesssim_{d}&&\mathcal F_{\phi(G)}^2\log(100l),\end{aligned}\end{equation} with an effective implied constant. In the last step we used the fact there there are only finitely many choices for $D\leq 4d^2$ for a fixed $d$ and estimates (\ref{n1n2bound}), (\ref{n1n2mbound}).

By (\ref{Bakerescape}) and assumption (\ref{lassumption}), \declaresubconst{nearlinebdsubconst2} $l^{-\frac12\ref{nearlinebdsubconst}\mathcal F_{\phi(G)}^2}\geq e^{\tilde\theta}l^{-\ref{nearlinebdsubconst}\mathcal F_{\phi(G)}^2}\geq \frac1\pi l^{-\ref{nearlinebdsubconst2} \mathcal F_{\phi(G)}^2}$, where both $\ref{nearlinebdsubconst2} =\ref{nearlinebdsubconst2} (d)$ is an effective constant arising from (\ref{Bakerescape}). But if we pick $l\geq2$ and $\ref{nearlinebdsubconst}$ sufficently large with respect to $\ref{nearlinebdsubconst2}$ then the inequality cannot be true. Contradiction.

Therefore $n_1\cdot2\pi\beta+n_2\cdot2\pi\tilde\beta-m\pi$ has to vanish.

Since $n_1\cdot2\pi\beta+n_2\cdot2\pi\tilde \beta$ is one argument of $(\zeta_u^i)^{n_1}(\zeta_{\tilde u}^i)^{n_2}$, so this restriction actually says $\zeta_{b_kb_{k'}^{-1}}^i=\zeta_{u^{n_1}\tilde u^{n_2}}^i$ is real.

Consider the set $\Lambda=\{(p,q)\in\mathbb Z^2|(\zeta_u^i)^p(\zeta_{\tilde u}^i)^q\in\mathbb R\}$, which is a subgroup of $\mathbb Z^2$. So $\Lambda$ is isomorphic to either $\mathbb Z$ or $\mathbb Z^2$ if it is not trivial. Assume $\Lambda\cong\mathbb Z^2$ then it is a lattice in $\mathbb Z^2$ and there exists a number $n$ such that $(n,0)\in\Lambda$, i.e. $\zeta_{u^n}^i=(\zeta_{h}^i)^n\in\mathbb R$, which contradicts Proposition \ref{expand}. Therefore $\Lambda$ has to be infinite cyclic. Fix a generator $(p,q)\in \mathbb Z^2\backslash\{(0,0)\}$ of $\Lambda$, then $(n_1,n_2)$ is a non-trivial multiple of $(p,q)$. So $\log|\zeta_{u^{n_1}\tilde u^{n_2}}^i|$ is a non-trivial multiple of $\log|\zeta_{u^p\tilde u^q}^i|$, which is at least $\frac14d(\frac r2+1)\mathcal F_{\phi(G)}$ by Proposition \ref{expand}.

But by (\ref{band}), $\log|\zeta_{b_kb_{k'}^{-1}}^i|\leq\log|e^{\tilde\theta}e^\theta|=\log|\zeta_u^i|+\log|\zeta_{\tilde u}^i|\leq h^\mathrm{Mah}(\phi(u))+h^\mathrm{Mah}(\phi(\tilde u))<18d(\frac r2+1)\mathcal F_{\phi(G)}$. Thus there are at most $72$ choices of $(n_1,n_2)=(k-k',-\lceil\gamma k\rceil+-\lceil\gamma k'\rceil).$ Hence if we fix $k$, there are at most $72$ other indices $k'$ such that $f(\zeta_{b_{k'}}^i)\leq\|f\|$, this establishes the last part of proposition.\end{proof}

\section{Measure-theoretical results}\label{measresults}

\hskip\parindent   We will effectively construct some element in the convex hull of $\{g.\tau,g\in G\}$ which approximately dominates some positive multiple of the Lebesgue measure $\mathrm m$ on $X$.

\subsection{Fourier coefficients on $X$}

\hskip\parindent Assuming Conditions \ref{condG'}, \ref{condmu}, in this part we are going to study the behavior of the measure $\nu$ in Proposition \ref{dominated} under the $G$ action. We will show in an effective way that certain averages of measures in the orbit $G.\nu$ will approach a positive multiple of the Lebesgue measure. For this purpose we will need the $L^2$-bound from Proposition \ref{dominated} as well as the phenomena described in Section \ref{GVi}.

Recall $\psi\in M_d(\mathbb R)$ is a linear conjuagtion between $\mathbb T^d=\mathbb R^d/\mathbb Z^d$ and $X=\mathbb R^d/\Gamma$ as described in Proposition \ref{KphiGamma}.

\begin{definition}Given a measure $\gamma$ on $X$, the Fourier coefficient of $\gamma$ at frequency $\xi\in X^*=\psi_*(\mathbb Z^d)$ is $$\hat\gamma(\xi)=\int_{x\in X}e(\xi(x))\mathrm d\gamma(x)$$ where $e(\beta)=e^{-2\pi\beta\mathrm i}$.\end{definition} When $\gamma$ comes from a measure on $\mathbb T^d$, the following lemma is obvious.

\begin{lemma}\label{Fourierstar} For all measures $\mathbb\gamma$ on $\mathbb T^d$ and all $\mathbf q\in\mathbb Z^d$, $\widehat{\psi_*\gamma}(\psi_*\mathbf q)=\hat\psi(\mathbf q)$.\end{lemma}

\begin{notation}\label{convexcomb} Given a measure $\nu$ on $X$ and an index $i\in I$, for $n,s,l\in\mathbb N$ such that $s\geq\ref{arithprolengthconst}\mathcal F_{\phi(G)}$, $l\geq2$ where $\ref{arithprolengthconst}$ is as in Propositions \ref{arithpro}, we write \begin{equation}\label{measureaverage}\nu_{n,s,l}^i=\frac1{sl}\sum_{t=0}^{s-1}\sum_{k=1}^l\times_{\phi(u^na_tb_k)}.\nu,\end{equation} with $u,a_t,b_k\in G$ respectively defined in Propositions \ref{expand}, \ref{arithpro} and \ref{nearlinebd} and $\times_{\phi(u^na_tb_k)}.\nu$ denoting the pushforward of $\nu$ by $\times_{\phi(u^na_tb_k)}$ on $X$. \end{notation}

With $\nu$ and $i$ constructed in Proposition \ref{dominated}, we hope to control the sizes of the Fourier coefficients of $\nu_{n,s,l}^i$ when the parameters $n$, $s$, and $l$ are carefully chosen. The first step to do this is the following estimate:

\begin{proposition}\label{fcupper}Assuming Conditions \ref{condG'}, \ref{condmu}, let $\nu$ and $i$ be as in Proposition \ref{dominated} and adopt other notations from that proposition as well. Suppose \begin{equation}\label{xiupper}\mathcal M_\psi A|\zeta_u^i|^n|\Delta|l^{\ref{nearlinebdexpoconst}\mathcal F_{\phi(G)}^2}2^{-R}\leq\frac14,\end{equation} then $|\widehat{\nu_{n,s,l}^i}(\xi)|^2\leq\sum_{c=1}^5L_c$ where \begin{equation}\label{fcdecomp}\begin{aligned}L_1=&9\cdot2^{-d_i\delta T}&\hskip2cm(\ref{fcdecomp}a)\\
L_2=&2\pi\sqrt d\mathcal M_\psi A2^{s^{10}+\ref{nearlinebdheightconst}\mathcal F_{\phi(G)}^2l-R}&(\ref{fcdecomp}b)\\
L_3=&2\pi\sqrt d\mathcal M_\psi A|\zeta_u^i|^ns^{-1}|\Delta|l^{\ref{nearlinebdexpoconst}\mathcal F_{\phi(G)}^2}2^{-R}&(\ref{fcdecomp}c)\\
L_4=&\frac d{2}2^{\frac{d(d-1)}2\ref{totirrheightconst}\mathcal F_{\phi(G)}}\mathcal M_\psi ^{d-1}A^{d-1}|\zeta_u^i|^{-n}s^{-1}|\Delta|^{-1}2^{R+T}&(\ref{fcdecomp}d)\\
L_5=&\frac{100}l&(\ref{fcdecomp}e)
\end{aligned}\nonumber\end{equation}\addtocounter{equation}{1}\end{proposition}

\begin{lemma}\label{expsumlemma}For $\beta\in(-\frac14,\frac14)$ and $s\in\mathbb N$, $|\sum_{t=0}^{s-1}e(t\beta)|\leq(2|\beta|)^{-1}$.\end{lemma}
\begin{proof}It is an easy fact that when $b\in(-\frac\pi2,\frac\pi2)$, $|\sin b|\geq\frac2\pi |b|$. Therefore $|\sum_{t=0}^{s-1}e(t\beta)|=|\frac{1-e(s\beta)}{1-e(\beta)}|\leq\frac2{|1-\exp(2\pi\beta\mathrm i)|}\leq\frac2{|\sin2\pi\beta|}\leq\frac2{\frac2\pi\cdot2\pi|\beta|}=(2|\beta|)^{-1}$.\end{proof}

\begin{proof}[Proof of Proposition \ref{fcupper}]
Recall $\nu=\mathbb E_{\mathrm m_{\tau,B}(x)}\nu_x$ where $\mathrm m_{\tau,B}$ is a probability measure on $X$ and each $\nu_x$ is supported on $x+B$.

For any fixed positive number $A$, for all $\xi\in X^*\backslash\{0\}$ with $|\xi|\leq A$, \begin{equation*}\begin{split}|\widehat{\nu_{n,s,l}^i}(\xi)|^2=&\Big|\int_Xe(\xi(y))\cdot \frac1{sl}\mathrm d(\sum_{t=0}^{s-1}\sum_{k=1}^l\times_{\phi(u^na_tb_k)}.\nu)(y)\Big|^2\\
=&\Big|\int_X\frac1{sl}\sum_{t=0}^{s-1}\sum_{k=1}^le\big(\xi(\times_{\phi(u^na_tb_k)}.y)\big)\mathrm d\nu (y)\Big|^2\\
=&\Big|\mathbb E_{\mathrm m_{\tau,B}(x)}\frac1{sl}\sum_{t=0}^{s-1}\sum_{k=1}^l\int_{x+B}e\big(\xi(\times_{\phi(u^na_tb_k)}.y)\big)\mathrm d\nu_x (y)\Big|^2\end{split}\end{equation*}
By Cauchy-Schwarz,

\begin{align*}
|\widehat{\nu_{n,s,l}^i}(\xi)|^2\leq&\mathbb E_{\mathrm m_{\tau,B}(x)}\frac1{sl}\sum_{t=0}^{s-1}\sum_{k=1}^l\Big|\int_{x+B}e\big(\xi(\times_{\phi(u^na_tb_k)}.y)\big)\mathrm d\nu_x (y)\Big|^2\\
=&\mathbb E_{\mathrm m_{\tau,B}(x)}\frac1{sl}\sum_{t=0}^{s-1}\sum_{k=1}^l\Big(\int_{x+B}e\big(\xi(\times_{\phi(u^na_tb_k)}.y)\big)\mathrm d\nu_x (y)\Big)\\
&\hskip3cm \overline{\Big(\int_{x+B}e\big(\xi(\times_{\phi(u^na_tb_k)}.y)\big)\mathrm d\nu_x (y)\Big)}\end{align*}
Using the simple fact that $\overline{e(\beta)}=e(-\beta), \forall\beta\in\mathbb R$, we obtain 
\begin{equation}\begin{split}\label{fcdecomp1}|\widehat{\nu_{n,s,l}^i}(\xi)|^2=&\mathbb E_{\mathrm m_{\tau,B}(x)}\frac1{sl}\sum_{t=0}^{s-1}\sum_{k=1}^l\\
&\hskip1cm\iint_{y,z\in x+B}e\Big(\xi\big(\times_{\phi(u^na_tb_k)}.(y-z)\big)\Big)\mathrm d\nu_x (y)d\nu_x (z).\end{split}\end{equation}

Next, we decompose $x+B$ by $x+\mathcal Q_{T\mathbf1_i}B$:
\begin{equation}\label{adjandnonadj}\begin{split}
&|\widehat{\nu_{n,s,l}^i}(\xi)|^2\\
=&\mathbb E_{\mathrm m_{\tau,B}(x)}\sum_{P,Q\in\mathcal Q_{T\mathbf1_i}B}\\
&\hskip.5cm\iint_{\substack{y\in x+P\\z\in x+Q}}\frac1{sl}\sum_{t=0}^{s-1}\sum_{k=1}^le\Big(\xi\big(\times_{\phi(u^na_tb_k)}.(y-z)\big)\Big)\mathrm d\nu_x (y)d\nu_x (z)\\
=&\mathbb E_{\mathrm m_{\tau,B}(x)}\sum_{\substack{P,Q\in\mathcal Q_{T\mathbf1_i}B\\\bar P\cap\bar Q\neq\emptyset}}\mathcal I_{P,Q,x}+\mathbb E_{\mathrm m_{\tau,B}(x)}\sum_{\substack{P,Q\in\mathcal Q_{T\mathbf1_i}B\\\bar P\cap\bar Q=\emptyset}}\mathcal I_{P,Q,x},\end{split}\end{equation}
where \begin{equation}\label{IntPQ}\mathcal I_{P,Q,x}=\iint_{\substack{y\in x+P\\z\in x+Q}}\frac1{sl}\sum_{t=0}^{s-1}\sum_{k=1}^le\Big(\xi\big(\times_{\phi(u^na_tb_k)}.(y-z)\big)\Big)\mathrm d\nu_x (y)d\nu_x (z)\geq 0.\end{equation}

From the construction of the partition $\mathcal Q_{T\mathbf1_i}B$ we know for each atom $P$ in it there are at most $9=3^2$ atoms $Q$ such that the closures $\bar P$ and $\bar Q$ intersect (when $V_i\cong\mathbb R$, there are actually at most 3 such atoms). So each atom appear in at most 9 adjacent pairs $(P,Q)$ as $P$ (and as $Q$ as well). Since $$\Big|\frac1{sl}\sum_{t=0}^{s-1}\sum_{k=1}^le\Big(\xi\big(\times_{\phi(u^na_tb_k)}.(y-z)\big)\Big)\Big|\leq\Big|\frac1{sl}\sum_{t=0}^{s-1}\sum_{k=1}^l1\Big|=1,$$ we have \begin{equation}\label{adjupper}\begin{split}&\mathbb E_{\mathrm m_{\tau,B}(x)}\sum_{\substack{P,Q\in\mathcal Q_{T\mathbf1_i}B\\\bar P\cap\bar Q\neq\emptyset}}\mathcal I_{P,Q,x}\\
\leq&\mathbb E_{\mathrm m_{\tau,B}(x)}\sum_{\substack{P,Q\in\mathcal Q_{T\mathbf1_i}B\\\bar P\cap\bar Q\neq\emptyset}}\nu_x(x+P)\nu_x(x+Q)\\
\leq&\mathbb E_{\mathrm m_{\tau,B}(x)}\sum_{\substack{P,Q\in\mathcal Q_{T\mathbf1_i}B\\\bar P\cap\bar Q\neq\emptyset}}\nu^2_x(x+P)\hskip.5cm(\text{Cauchy-Schwarz)}\\
\leq&\mathbb E_{\mathrm m_{\tau,B}(x)}9\sum_{P\in\mathcal Q_{T\mathbf1_i}B}\nu^2_x(x+P)\\
\leq&9\cdot2^{-d_i\delta T}=L_1,\end{split}\end{equation} where the last step follows from the $L^2$-bound in Propoperty \ref{dominated}.

From now on we suppose two points $y,z\in X$ and two atoms $P,Q\in \mathcal Q_{T\mathbf1_i}B$ are such that \begin{equation}\label{nonadj}\bar P\cap\bar Q=\emptyset,\ y\in x+P,\  z\in x+Q. \end{equation}

In this case the orthogonal projection of the distance vector $y-z$ in subspace $V_i$, denoted by $(y-z)_i$, has size at least $2^{-R-T}\mathcal S_\Gamma$, which is the length of the sides of atoms from $\mathcal Q_{T\mathbf1_i}B$ in $V_i$-direction. On the other hand $|y-z|\leq\sqrt d2^{-R}\mathcal S_\Gamma$ because $y$, $z$ belongs to the same $d$-dimensional cube $x+B$ whose sides have length $2^{-R}\mathcal S_\Gamma$.

Thus the difference $(y-z)'\overset{\text{def}}=(y-z)-(y-z)_i$, which is the orthogonal projection into $\oplus_{j\neq i}V_j$, has size at most $\sqrt d2^{-R}\mathcal S_\Gamma$. Recall $|\zeta_u^i|\geq2^{d(\frac r2+1)\mathcal F_{\phi(G)}}$ and $|\zeta_u^j|\leq 1, \forall j\neq i$; $\log|\zeta_{a_t}^j|\leq \log h^\mathrm{Mah}(\phi(a_t))\leq s^{10}$ and similarly $|\zeta_{b_k}^j|\leq\ref{nearlinebdheightconst}\mathcal F_{\phi(G)}^2l$. Thus $$|\times_{\phi(u^na_tb_k)}.(y-z)'|\leq2^{s^{10}+\ref{nearlinebdheightconst}\mathcal F_{\phi(G)}^2l}\cdot \sqrt d2^{-R}.$$

$e:\mathbb R\mapsto\mathbb S^1\subset\mathbb C$ is Lipschitz continuous with $2\pi$ as Lipschitz constant. As $\|\xi\|\leq |\xi|\cdot\|\psi^{-1}\|\leq A\mathcal M_\psi \mathcal S_\Gamma^{-1}$, $e(\xi(\cdot))$ has Lipschitz constant $2\pi A\mathcal M_\psi \mathcal S_\Gamma^{-1}$ and \begin{align*}&\Big|e\Big(\xi\big(\times_{\phi(u^na_tb_k)}.(y-z)\big)\Big)-e\Big(\xi\big(\times_{\phi(u^na_tb_k)}.(y-z)_i\big)\Big)\Big|\\\leq& 2\pi A\mathcal M_\psi \mathcal S_\Gamma^{-1}\cdot|\times_{\phi(u^na_tb_k)}.(y-z)'|\\\leq&2\pi\sqrt d\mathcal M_\psi A2^{s^{10}+\ref{nearlinebdheightconst}\mathcal F_{\phi(G)}^2l-R}=L_2.\end{align*} Hence \begin{equation}\label{nonadjexpsum}\begin{split}&\Big|\frac1{sl}\sum_{t=0}^{s-1}\sum_{k=1}^le\Big(\xi\big(\times_{\phi(u^na_tb_k)}.(y-z)\big)\Big)\Big|\\\leq&\Big|\frac1{sl}\sum_{t=0}^{s-1}\sum_{k=1}^le\Big(\xi\big(\times_{\phi(u^na_tb_k)}.(y-z)_i\big)\Big)\Big|+L_2.\end{split}\end{equation}

We now study $\Big|\frac1{sl}\sum_{t=0}^{s-1}\sum_{k=1}^le\Big(\xi\big(\times_{\phi(u^na_tb_k)}.(y-z)_i\big)\Big)\Big|$. First of all observe since $(y-z)_i$ lies in $V_i$, which is regarded as $\mathbb R$ or $\mathbb C$, $\times_{\phi(u^na_tb_k)}.(y-z)_i=\zeta_u^n\zeta_{a_t}^i\zeta_{b_k}^i(y-z)_i$.

By Proposition \ref{arithpro} $|\zeta_{a_t}^i-(1+t\Delta)|\leq s^{-1}|\Delta|$ with $s^{-\ref{arithproexpoconst}}\mathcal F_{\phi(G)}^2\leq|\Delta|\leq s^{-3}$. And the other factor $|\zeta_{b_k}^i|\leq l^{\ref{nearlinebdexpoconst}\mathcal F_{\phi(G)}^2}$. Once again because of Lipschitz continuity, \begin{align*}&\Big|e\Big(\xi\big(\times_{\phi(u^na_tb_k)}.(y-z)_i\big)\Big)-e\Big(\xi\big((\zeta_u^i)^n(1+t\Delta)\zeta_{b_k}^i(y-z)_i\big)\Big)\Big|\\\leq&2\pi A\mathcal M_\psi \mathcal S_\Gamma^{-1}\cdot|\zeta_u^i|^n|\zeta_{a_t}^i-(1+t\Delta)||\zeta_{b_k}^i|\cdot|(y-z)_i|\\\leq&2\pi\sqrt d\mathcal M_\psi A|\zeta_u^i|^ns^{-1}|\Delta|l^{\ref{nearlinebdexpoconst}\mathcal F_{\phi(G)}^2}2^{-R}=L_3.\end{align*} So \begin{equation}\label{aperr}(\ref{nonadjexpsum})\leq\Big|\frac1{sl}\sum_{t=0}^{s-1}\sum_{k=1}^le\Big(\xi\big((\zeta_u^i)^n(1+t\Delta)\zeta_{b_k}^i(y-z)_i\big)\Big)\Big|+L_2+L_3.\end{equation}

Define a $\mathbb R$-linear form $f(v):=\xi\big((\zeta_u^i)^n\Delta v(y-z)_i\big)$ on either $\mathbb R$ or $\mathbb C$ depending on whether $V_i\cong\mathbb R$ or $\mathbb C$. Then $\|f\|=\big\|\xi\big|_{V_i}\big\|\cdot|\zeta_u^i|^n|\Delta|\cdot|(y-z)_i|$. However, by Lemma \ref{quantirr}, $\big\|\xi\big|_{V_i}\big\|\geq d^{-1}2^{-\frac{d(d-1)}2\ref{totirrheightconst}\mathcal F_{\phi(G)}}\mathcal M_\psi ^{-(d-1)}A^{-(d-1)}\mathcal S_\Gamma^{-1}$. Thus by Proposition \ref{nearlinebd}, there is a subset $J\subset\{1,\cdots,l\}$ of cardinality at most $100$ such that $\forall k\in\{1,\cdots,l\}\backslash J$, 
\begin{equation}\label{exceptionalJ}\begin{split}&|\xi\big((\zeta_u^i)^n\Delta\zeta_{b_k}^i(y-z)_i\big)|\\
=&|f(\zeta_{b_k}^i)|\geq\|f\|\\
\geq& \|\xi|_{V_i}\|\cdot|\zeta_u^i|^n|\Delta|\cdot|(y-z)_i|\mathcal S_\Gamma\\
\geq&d^{-1}2^{-\frac{d(d-1)}2\ref{totirrheightconst}\mathcal F_{\phi(G)}}\mathcal M_\psi ^{-(d-1)}A^{-(d-1)}|\zeta_u^i|^n|\Delta|2^{-R-T}\\=&\frac1{2s}L_4^{-1}.\end{split}\end{equation}
Here we used the fact $|(y-z)_i|\geq 2^{-R-T}\mathcal S_\Gamma$.

On the other hand, it follows from assumption (\ref{xiupper}) that \begin{equation}\label{xiquarter}\begin{split}|\xi((\zeta_u^i)^n\Delta\zeta_{b_k}^i(y-z)_i))|\leq& A\mathcal M_\psi \mathcal S_\Gamma^{-1}|(\zeta_u^i)^n\Delta\zeta_{b_k}^i(y-z)_i|\\=&\mathcal M_\psi A|\zeta_u^i|^n|\Delta|l^{\ref{nearlinebdexpoconst}\mathcal F_{\phi(G)}^2}2^{-R}\leq\frac14.\end{split}\end{equation}
Moreover notice that  \begin{equation}\label{expsum}\begin{split}&\Big|\frac1{sl}\sum_{t=0}^{s-1}\sum_{k=1}^le\Big(\xi\big((\zeta_u^i)^n(1+t\Delta)\zeta_{b_k}^i(y-z)_i\big)\Big)\Big|\\
 \leq& \frac1l\sum_{k=1}^l\Big|\frac1s\sum_{t=0}^{s-1}e\Big(\xi\big((\zeta_u^i)^n(1+t\Delta)\zeta_{b_k}^i(y-z)_i\big)\Big)\Big|\\
=&\frac1l\sum_{k=1}^l\Big|e\Big(\xi\big((\zeta_u^i)^n\zeta_{b_k}^i(y-z)_i\big)\Big)\cdot\frac1s\sum_{t=0}^{s-1}e\Big(\xi\big((\zeta_u^i)^nt\Delta\zeta_{b_k}^i(y-z)_i\big)\Big)\Big|\\
=&\frac1l\sum_{k=1}^l\Big|\frac1s\sum_{t=0}^{s-1}e\Big(t\xi\big((\zeta_u^i)^n\Delta\zeta_{b_k}^i(y-z)_i\big)\Big)\Big|,\end{split}\end{equation}
by separating the exception set $J$ from generic indices we get
\begin{equation}\label{expsum1}\begin{split}
(\ref{expsum})\leq&\frac1l\Big(100+\sum_{k\in\{1,\cdots,l\}\backslash J}\Big|\frac1s\sum_{t=0}^{s-1}e\Big(t\xi\big((\zeta_u^i)^n\Delta\zeta_{b_k}^i(y-z)_i\big)\Big)\Big|\Big)\\
\leq&\frac1l\Big(100+\sum_{k\in\{1,\cdots,l\}\backslash J}\frac1s\Big(2\Big|\xi\big((\zeta_u^i)^n\Delta\zeta_{b_k}^i(y-z)_i)\big)\Big|\Big)^{-1}\Big),\end{split}\end{equation}
where Lemma \ref{expsumlemma} was applied in the last step in light of condition (\ref{xiquarter}).

Making use of (\ref{exceptionalJ}), we obtain from (\ref{expsum1}) that 
\begin{equation}\label{expsum2}(\ref{expsum})\leq \frac1l\big(100+l\cdot\frac1s(2\cdot\frac1{2s}L_4^{-1})^{-1}\big)=\frac{100}l+L_4=L_4+L_5\end{equation}

Plug (\ref{expsum2}) into (\ref{aperr}), we get \begin{equation}\label{nonadjexpsumbd}(\ref{nonadjexpsum})\leq L_2+L_3+L_4+L_5.\end{equation}  under the assumptions (\ref{xiupper}) and (\ref{nonadj}).

Plug (\ref{nonadjexpsumbd}) into (\ref{IntPQ}), it follows from the fact $|\nu_x|\leq 1$ almost everywhere (see Proposition \ref{dominated}) that $\mathcal I_{P,Q,x}\leq L_2+L_3+L_4+L_5$ for a.e. $x$ if $\bar P\cap \bar Q=0$ thus  \begin{equation}\label{nonadjupper}\mathbb E_{\mathrm m_{\tau,B}(x)}\sum_{\substack{P,Q\in\mathcal Q_{T\mathbf1_i}B\\\bar P\cap\bar Q=\emptyset}}\mathcal I_{P,Q,x}\leq L_2+L_3+L_4+L_5.\end{equation}

The proposition is verified by combining (\ref{adjandnonadj}), (\ref{adjupper}) and (\ref{nonadjupper}).\end{proof}

\subsection{Manipulation of parameters}

\hskip\parindent We want to show that if the parameters $\delta$, $T$ and $n$, $s$, $l$ are properly chosen then \begin{equation}\label{fcAinverse}|\widehat{\nu_{n,s,l}^i}(\xi)|^2\lesssim_{d}A^{-1}\end{equation} for any non-trivial frequency $\xi$ with $|\xi|\leq A$ as long as $A$ is a large number but still sufficiently small compared to $R$.

Actually by Proposition \ref{fcupper} this would follow from the following collection of inequalities:

\begin{equation}\label{collineq}\left\{\begin{aligned}
&\mathcal M_\psi A|\zeta_u^i|^n|\Delta|l^{\ref{nearlinebdexpoconst}\mathcal F_{\phi(G)}^2}2^{-R}\leq\frac14&\hskip1.5cm(\ref{xiupper})\\
&2^{-d_i\delta T}\lesssim A^{-1}&(\ref{collineq}a)\\
&\mathcal M_\psi A2^{s^{10}+\ref{nearlinebdheightconst}\mathcal F_{\phi(G)}^2l-R}\lesssim A^{-1}&(\ref{collineq}b)\\
&\mathcal M_\psi A|\zeta_u^i|^ns^{-1}|\Delta|l^{\ref{nearlinebdexpoconst}\mathcal F_{\phi(G)}^2}2^{-R}\lesssim A^{-1}&(\ref{collineq}c)\\
&2^{\frac{d(d-1)}2\ref{totirrheightconst}\mathcal F_{\phi(G)}}\mathcal M_\psi ^{d-1}A^{d-1}|\zeta_u^i|^{-n}s^{-1}|\Delta|^{-1}2^{R+T}\lesssim A^{-1}&(\ref{collineq}d)\\
&l^{-1}\lesssim A^{-1}&(\ref{collineq}e)
\end{aligned}\right.\nonumber\end{equation}\addtocounter{equation}{1}

First of all, set \begin{equation}\label{lAsetting}l=A=\lceil2^{d_i\delta T}\rceil\end{equation} so that (\ref{collineq}a) and (\ref{collineq}e) are satisfied. Multiplying (\ref{xiupper}) by (\ref{collineq}d) gives $$2^{\frac{d(d-1)}2\ref{totirrheightconst}\mathcal F_{\phi(G)}}\mathcal M_\psi ^dA^ds^{-1}l^{\ref{nearlinebdexpoconst}\mathcal F_{\phi(G)}^2}2^T\lesssim A^{-1}.$$ So $s$ should be such that \begin{align*}s\gtrsim& 2^{\frac{d(d-1)}2\ref{totirrheightconst}\mathcal F_{\phi(G)}}\mathcal M_\psi ^dA^{d+1}l^{\ref{nearlinebdexpoconst}\mathcal F_{\phi(G)}^2}2^T\\\approx& 2^{\frac{d(d-1)}2\ref{totirrheightconst}\mathcal F_{\phi(G)}}\mathcal M_\psi ^d2^{\big(1+(d+1+\ref{nearlinebdexpoconst}\mathcal F_{\phi(G)}^2)d_i\delta\big)T}.\end{align*}
We take the critical setting \begin{equation}\label{ssetting}s=\lceil2^{\frac{d(d-1)}2\ref{totirrheightconst}\mathcal F_{\phi(G)}}\mathcal M_\psi ^d2^{(1+(d+1+\ref{nearlinebdexpoconst}\mathcal F_{\phi(G)}^2)d_i\delta)T}\rceil.\end{equation}
Now the parameter $\Delta$ is determined by $s$ together with the index $i$. Take $n$ so that both sides of (\ref{xiupper}) are almost equal up to a multiplicative error: \begin{equation}\label{nsetting}\mathcal M_\psi A|\zeta_u^i|^n|\Delta|l^{\ref{nearlinebdexpoconst}\mathcal F_{\phi(G)}^2}2^{-R}\in(\frac1{4|\zeta_u^i|},\frac1{4}].\end{equation}
This together with the way $s$ was determined implies that (\ref{collineq}d) holds. Notice $s\gtrsim A$, therefore (\ref{collineq}c) actually follows from (\ref{xiupper}) by multiplying both sides respectively by $s^{-1}$ and $A^{-1}$. It only remains to check (\ref{collineq}b), which would follow if \begin{equation}2^R\gtrsim\mathcal M_\psi A^22^{s^{10}+\ref{nearlinebdheightconst}\mathcal F_{\phi(G)}^2l}\sim \mathcal M_\psi 2^{s^{10}+\ref{nearlinebdheightconst}\mathcal F_{\phi(G)}^2l+2d_i\delta T}.\end{equation} It suffices to assume \begin{equation}\label{howbigRis}\begin{split}R\geq&\log\mathcal M_\psi +s^{10}+\ref{nearlinebdheightconst}\mathcal F_{\phi(G)}^2l+2d_i\delta T\\\sim&\log\mathcal M_\psi +2^{5d(d-1)\ref{totirrheightconst}\mathcal F_{\phi(G)}}\mathcal M_\psi ^{10d}2^{(10+10(d+1+\ref{nearlinebdexpoconst}\mathcal F_{\phi(G)}^2)d_i\delta)T}\\&\hskip1cm+\ref{nearlinebdheightconst}\mathcal F_{\phi(G)}^22^{d_i\delta T}+2d_i\delta T\end{split}\end{equation}

Remark (\ref{howbigRis}) would also assure that $\mathcal M_\psi A|\Delta|l^{\ref{nearlinebdexpoconst}\mathcal F_{\phi(G)}^2}2^{-R}\ll 1$ (recall $|\Delta|\leq s^{-3}<1$), so that the definition of the positive integer $n$ in (\ref{nsetting}) makes sense.

\declareconst{RTboundconst}
For a sufficiently large constant $\ref{RTboundconst} =\ref{RTboundconst} (d)$, if $T\geq 2$ and \begin{equation}\label{RTbound}R\geq \cdot2^{\ref{RTboundconst} \mathcal F_{\phi(G)}^2(T-1)}\mathcal M_\psi ^{10d}\end{equation} then (\ref{howbigRis}) holds and therefore so do all the inequalities (\ref{xiupper}) and (\ref{collineq}a-e), as a consequence (\ref{fcAinverse}) holds whenever $|\xi|\leq A$, $\xi\neq0$.

So far we have almost proved the following result:

\declareconst{fcAinXdeltaconst} \declareconst{fcAinXmetricballconst} \declareconst{fcAinXboundconst}
\begin{proposition}\label{fcAinX}We can find effective constants $\ref{fcAinXdeltaconst}(d)$, $\ref{fcAinXmetricballconst}(d)$, $\ref{fcAinXboundconst}(d)$ such that:

Assuming Conditions \ref{condG'} and \ref{condmu}; if 
\begin{equation}\label{epsilonfcAinX}\log\frac1\epsilon>\max(\mathcal M_\psi ^{30d},4),\end{equation} and
\begin{equation}\label{deltafcAinX}\delta\in[\ref{fcAinXdeltaconst}\mathcal F_{\phi(G)}^2(\log\log\frac1\epsilon)^{-1},\frac\alpha{10}],\end{equation}
then there exist:

(1) A number $A\geq(\log\frac1\epsilon)^{\ref{fcAinXdeltaconst}^{-1}\mathcal F_{\phi(G)}^{-2}\delta}$;

(2) A measure $\tau'$ with total mass $|\tau'|\geq\alpha-5\delta$ which is dominated by some $\tau''$ in the convex hull of $B_G^\mathrm{Mah}(\ref{fcAinXmetricballconst}\log\frac1\epsilon).\tau:=\{\times_{\phi(g)}.\tau|m(g)\leq\ref{fcAinXmetricballconst}\log\frac1\epsilon\}$, where $\tau=\psi_*\mu$ and $m$ is the logarithmic Mahler measure;\\
such that $|\widehat{\tau'}(\xi)|^2\leq \ref{fcAinXboundconst}A^{-1}$ for any non-trivial character $\xi\in X^*$ with $|\xi|\leq A$, where $|\xi|$ is defined in Definition \ref{absxi}.\end{proposition}

\begin{proof} As in Proposition \ref{dominated}, put $\sqrt d2^{-R_0}=\eta=\mathcal M_\psi ^{-1}\epsilon$. Set
\begin{equation}\label{Tsetting}T=\lceil\ref{RTboundconst} ^{-1}\mathcal F_{\phi(G)}^{-2}\log(\mathcal M_\psi ^{-10d}\delta\log\frac1\epsilon)\rceil.\end{equation}
We claim Proposition \ref{dominated} applies, to see this it suffices to verify the conditions $\delta\geq\frac{\ref{posentropydeltaconst}\mathcal M_\psi^d}{R_0}$ and $T\in[\frac1\delta,\frac{\delta R_0}2]$.

Set $\ref{fcAinXdeltaconst}=3\ref{RTboundconst} $. When $\log\frac1\epsilon\geq 4$, $\log\log\frac1\epsilon\lesssim(\log\frac1\epsilon)^\frac13$; furthermore we assumed $\mathcal M_\psi ^{10d}\leq (\log\frac1\epsilon)^\frac13$ as well. Thus as $\mathcal F_{\phi(G)}\gtrsim_d1$, when $\ref{RTboundconst} =\ref{RTboundconst} (d)$ is large enough  we have $\ref{fcAinXdeltaconst}\mathcal F_{\phi(G)}^2\geq 1$ and
\begin{equation}\label{Tdeltabound}\begin{split}T\geq&\ref{RTboundconst} ^{-1}\mathcal F_{\phi(G)}^{-2}\log(\ref{fcAinXdeltaconst}\mathcal F_{\phi(G)}^2\mathcal M_\psi ^{-10d}\frac{\log\frac1\epsilon}{\log\log\frac1\epsilon})\\
\geq&\ref{RTboundconst} ^{-1}\mathcal F_{\phi(G)}^{-2}\log\sqrt[3]{\log\frac1\epsilon}\\
\geq&\ref{RTboundconst} ^{-1}\mathcal F_{\phi(G)}^{-2}\cdot \frac13(\ref{fcAinXdeltaconst}^{-1}\mathcal F_{\phi(G)}^{-2}\delta)^{-1}\\
=&\delta^{-1}.\end{split}\end{equation}

On the other hand as (\ref{deltafcAinX}) implicitly gives $$\ref{RTboundconst} ^{-1}\mathcal F_{\phi(G)}^{-2}\log\log\frac1\epsilon=3\big(\ref{fcAinXdeltaconst}\mathcal F_{\phi(G)}^2(\log\log\frac1\epsilon)^{-1}\big)^{-1}\geq 3\delta^{-1}\geq 3,$$ when $\ref{RTboundconst} $ is sufficiently large  \begin{equation}\begin{split}T\leq&\ref{RTboundconst} ^{-1}\mathcal F_{\phi(G)}^{-2}\log(M^{-10d}\delta\log\frac1\epsilon)+1\\
\leq&\ref{RTboundconst} ^{-1}\mathcal F_{\phi(G)}^{-2}\log\log\frac1\epsilon+1\\
\leq&\ref{RTboundconst} ^{-1}\mathcal F_{\phi(G)}^{-2}\cdot 2\log\log\frac1\epsilon\\
\leq&\frac{\log\frac1\epsilon}{2\log\log\frac1\epsilon}\leq\frac{\delta\log\frac1\epsilon}2\leq\frac{\delta R_0}2.\end{split}\end{equation}

By similar arguments, it is quite easy to deduce $\delta\geq\frac{\ref{posentropydeltaconst}\mathcal M_\psi^d}{R_0}$ using  $R_0\geq\log\frac1\epsilon>\mathcal M_\psi^{30d}$ and $\mathcal F_{\phi(G)}\gtrsim_d1$.

Thus we can obtain $R$, $\nu$ and $i$ from Proposition \ref{dominated}. Notice $R\geq\delta R_0\geq\delta\log\frac1\epsilon$, therefore (\ref{RTbound}) follows from (\ref{Tsetting}). And we may define the numbers $n$, $s$, $l$, $A$ as in (\ref{lAsetting}), (\ref{ssetting}) and (\ref{nsetting}). Then (\ref{fcAinverse}) was already proved for all non trivial characters $\xi\in X^*$ with $|\xi|\leq A$. Denote the implied constant by $\ref{fcAinXboundconst}$, which is effective.

Let $\tau'$ be the measure $\nu_{n,s,l}^i$ defined in Notation \ref{convexcomb}. Since $\tau'$ is the average of a certain collection of pushforwards of $\nu$, $|\tau'|=|\nu|\geq\alpha-5\delta$. Because $\nu\leq\tau$, $\tau'$ is dominated by the probability measure $\tau'':=\frac1{sl}\sum_{t=0}^{s-1}\sum_{k=1}^l(u^na_tb_k).\tau$, which lies in the convex hull of $$\{g.\tau,\ h^\mathrm{Mah}(\phi(g))\leq nh^\mathrm{Mah}(\phi(u))+\max_th^\mathrm{Mah}(\phi(a_t))+\max_kh^\mathrm{Mah}(\phi(b_k))\big)\}$$ inside the space of probability measures on $X$.

Observe $h^\mathrm{Mah}(\phi(u))=\log|\zeta_u^i|$ as $V_i$ is the only expanding subspace under $\times_u$. By Propositions in \S \ref{GVi}, $nh^\mathrm{Mah}(\phi(u))+\max_th^\mathrm{Mah}(\phi(a_t))+\max_kh^\mathrm{Mah}(\phi(b_k))\leq n\log|\zeta_u^i|+s^{10}+\ref{nearlinebdheightconst}\mathcal F_{\phi(G)}^2l$, which we will show  can be bounded by a multiple of $\log\frac1\epsilon$. 

It follows from (\ref{nsetting}) that $|\zeta_u^i|^n\leq\mathcal M_\psi ^{-1}|\Delta|^{-1}2^R\leq \mathcal M_\psi ^{-1}s^{\ref{arithproexpoconst}}2^R$ and by (\ref{howbigRis}) $s^{10}\mathcal F_{\phi(G)}^2l\leq R-\log\mathcal M_\psi $. So $n\log|\zeta_u^i|+s^{10}+\ref{nearlinebdheightconst}\mathcal F_{\phi(G)}^2l\leq(\ref{arithproexpoconst}\log s+R-\log\mathcal M_\psi )+(R-\log\mathcal M_\psi )\lesssim_{d}s^{10}+R-\log\mathcal M_\psi \lesssim R_0-\log\mathcal M_\psi =\log\frac{\sqrt d}\epsilon$. Thus by Lemma \ref{mahlerheight} there exists an effective constant $\ref{fcAinXmetricballconst}(d)$ such that $\tau''$ is in the the convex hull of $B_G^\mathrm{Mah}(\ref{fcAinXmetricballconst}\log\frac1\epsilon).\tau$.

It only remains to obtain the lower bound for $A$. As we already remarked in (\ref{Tdeltabound}), $T\geq\ref{RTboundconst} ^{-1}\mathcal F_{\phi(G)}^{-2}\log\sqrt[3]{\log\frac1\epsilon}$. Hence $A\geq 2^{d_i\delta T}\geq 2^{\delta \ref{RTboundconst} ^{-1}\mathcal F_{\phi(G)}^{-2}\log\sqrt[3]{\log\frac1\epsilon}}=(\log\frac1\epsilon)^{\frac13 \ref{RTboundconst} ^{-1}\mathcal F_{\phi(G)}^{-2}\delta}=(\log\frac1\epsilon)^{\ref{fcAinXdeltaconst}^{-1}\mathcal F_{\phi(G)}^{-2}\delta}$.\end{proof}

Now we may pull everything from $X$ back to $\mathbb T^d$. The following corollary is immediate.

\begin{corollary}\label{fcA} For the same constants $\ref{fcAinXdeltaconst}$, $\ref{fcAinXmetricballconst}$ and $\ref{fcAinXboundconst}$ as in Proposition \ref{fcAinX}, if Conditions \ref{condG'} and \ref{condmu} hold and
$$\log\frac1\epsilon>\max(\mathcal M_\psi ^{30d},4);$$ $$\delta\in[5\ref{fcAinXdeltaconst}\mathcal F_{\phi(G)}^2(\log\log\frac1\epsilon)^{-1},\frac\alpha2],$$ then there exist:

(1) A number $A\geq (\log\frac1\epsilon)^{\frac15\ref{fcAinXdeltaconst}^{-1}\mathcal F_{\phi(G)}^{-2}\delta}$;

(2) A measure $\mu'$ with total mass $|\mu'|\geq\alpha-\delta$ which is dominated by some $\mu''$ in the convex hull of $B_G^\mathrm{Mah}(\ref{fcAinXmetricballconst}\log\frac1\epsilon).\mu$;\\
such that $|\widehat{\mu'}(\mathbf q)|^2\leq \ref{fcAinXboundconst}A^{-1}$ for all non-trivial characters $\mathbf q\in\mathbb Z^d$ with $|\mathbf q|\leq A$.
\end{corollary}
\begin{proof}Apply Proposition \ref{fcAinX} to $\frac\delta5$ instead of $\delta$ and pull back from $X$ back to $\mathbb T^d$ by $\psi$.\end{proof}

\subsection{The effective measure-theoretical theorem}

\hskip\parindent In Corollary \ref{fcA} we constructed a measure $\mu'$ whose Fourier coefficients are small except the trivial one. Since for the Lebesgue measure $\mathrm m$ on $\mathbb T^d$, all the Fourier coefficients at non-trivial frequencies vanish. It is natural that we can effectively describe how close $\mu'$ is to a multiple of $\mathrm m$.

\begin{definition}For a function $f\in \mathcal C^\infty(\mathbb T^d)$ and $w>0$, $\|f\|_{\dot H^w}$ denotes the Sobolev norm $\big(\sum_{\mathbf q\in\mathbb Z^d}\big||\mathbf q|^w\hat f(\mathbf q)\big|^2\big)^{\frac12}$.\end{definition}

\declareconst{sobolevconst}
\begin{lemma}\label{sobolev} If a measure $\gamma$ on $\mathbb T^d$ satisfies $|\gamma|\leq 1$ and $|\hat\gamma(\mathbf q)|\leq CA^{-\frac12}$,  $\forall\mathbf q\in\mathbb Z^d\backslash\{0\}, |\mathbf q|\leq A$ for some fixed positive number $C$, then there is an explicit constant $\ref{sobolevconst}=\ref{sobolevconst}(d)$ such that $\forall f\in\mathcal C^\infty(\mathbb T^d)$, $$\big|\gamma(f)-|\gamma|\mathrm m(f)\big|\leq\ref{sobolevconst}\cdot(C+1)A^{-\frac12}\|f\|_{\dot H^{\frac{d+1}2}}.$$\end{lemma}

\begin{proof} Denote $B_{\mathbb Z^d}(L)=\{\mathbf q\in\mathbb Z^d,|\mathbf q|\leq L\}$. If $f$ is sufficiently differentiable then the spherical Fourier sums $$S_L(f)(x):=\sum_{\mathbf q\in B_{\mathbb Z^d}(L)}\hat f(\mathbf q)e(\mathbf q\cdot x), x\in\mathbb T^d$$ coverge uniformly to $f$ as $L\rightarrow\infty$ (see Il'in \cite{I68}).

Hence $\gamma(f)=\lim_{L\rightarrow\infty}\gamma(S_L(f))=\lim_{L\rightarrow\infty}\sum_{\mathbf q\in B_{\mathbb Z^d}(L)}\hat\gamma(\mathbf q)\hat f(\mathbf q)$ and
\begin{equation}\label{fcboxsum}\begin{split}&\big|\gamma(f)-|\gamma|\mathrm m(f)\big|\\=&|\gamma(f)-\hat\gamma(0)\hat f(0)|\\=&|\lim_{L\rightarrow\infty}\sum_{\mathbf q\in B_{\mathbb Z^d}(L)\backslash\{0\}}\hat\gamma(\mathbf q)\hat f(\mathbf q)|\\\leq&\overline\lim_{L\rightarrow\infty}\big\||\mathbf q|^{-\frac{d+1}2}\hat\gamma(\mathbf q)\big\|_{l^2(B_{\mathbb Z^d}(L)\backslash\{0\})}\big\||\mathbf q|^{\frac{d+1}2}\hat f(\mathbf q)\big\|_{l^2(B_{\mathbb Z^d}(L)\backslash\{0\})}\\&\hskip5cm\text{(Cauchy-Schwarz)}\\\leq&\big\||\mathbf q|^{-\frac{d+1}2}\hat\gamma(\mathbf q)\big\|_{l^2(\mathbb Z^d\backslash\{0\})}\big\||\mathbf q|^{\frac{d+1}2}\hat f(\mathbf q)\big\|_{l^2(\mathbb Z^d)}\\\leq&\big(\big\||\mathbf q|^{-\frac{d+1}2}\hat\gamma(\mathbf q)\big\|_{l^2(B_{\mathbb Z^d}(A)\backslash\{0\})}+\big\||\mathbf q|^{-\frac{d+1}2}\hat\gamma(\mathbf q)\big\|_{l^2(\mathbb Z^d\backslash B_{\mathbb Z^d}(A))}\big)\|f\|_{\dot H^{\frac{d+1}2}}.\end{split}\end{equation}

However it follows from the assumption that $\big\||\mathbf q|^{-\frac{d+1}2}\hat\gamma(\mathbf q)\big\|_{l^2(B_{\mathbb Z^d}(A)\backslash\{0\})}\leq CA^{-\frac12}\big\||\mathbf q|^{-\frac{d+1}2}\big\|_{l^2(B_{\mathbb Z^d}(A)\backslash\{0\})}\leq C\big\||\mathbf q|^{-\frac{d+1}2}\big\|_{l^2(\mathbb Z^d\backslash\{0\})}\cdot A^{-\frac12}$. Remark the norm $\big\||\mathbf q|^{-\frac{d+1}2}\big\|_{l^2(\mathbb Z^d\backslash\{0\})}$ is a finite constant relying only on $d$.

On the other hand, as $\hat\gamma(\mathbf q)\leq|\gamma|\leq 1$ for any $\mathbf q$, $\big\||\mathbf q|^{-\frac{d+1}2}\hat\gamma(\mathbf q)\big\|_{l^2(\mathbb Z^d\backslash B_{\mathbb Z^d}(A))}$ is bounded by $\big\||\mathbf q|^{-\frac{d+1}2}\big\|_{l^2(\mathbb Z^d\backslash B_{\mathbb Z^d}(A))}\sim_d\|y^{-\frac{d+1}2}\|_{L^2(\mathbb R^d\backslash B_{\mathbb R^n}(A))}$, which has order $(\int_{y\in\mathbb R^d,|y|>A}y^{-d-1}\mathrm dy)^\frac12\sim_d (A^{-1})^{\frac12}=A^{-\frac12}$ where $B_{\mathbb R^n}(A)$ denotes the Euclidean ball of radius $A$ in $\mathbb R^n$.

By plugging these bounds into (\ref{fcboxsum}), it is clear that $\big|\gamma(f)-|\gamma|\mathrm m(f)\big|\lesssim_d(CA^{-\frac12}+A^{-\frac12})\|f\|_{\dot H^{\frac{d+1}2}}=(C+1)A^{-\frac12}\|f\|_{\dot H^{\frac{d+1}2}}$.
\end{proof}

Applying lemma to the measure $\mu'$ from Corollary \ref{fcA}, we obtain the following:

\declareCONST{effmeas'deltaconst}
\begin{proposition}\label{effmeas'}There exist effective constants $\ref{efftopometricballCONST} (d)$, $\ref{effmeascoeffCONST} (d)$ and $\ref{effmeas'deltaconst}(d)$,  such that:

If $G$ and $\mu$ satisfy Condition \ref{condG'} and \ref{condmu} respectively, 
$\log\frac1\epsilon\geq\max(\mathcal M_\psi ^{30d},4)$ and $\delta\in[\ref{effmeas'deltaconst}\mathcal F_{\phi(G)}^2(\log\log\frac1\epsilon)^{-1},\frac\alpha2]$, then there exists a measure $\mu'$ with total mass $|\mu'|\geq\alpha-\delta$ which is dominated by some $\mu''$ in the convex hull of $B_G^\mathrm{Mah}(\ref{efftopometricballCONST} \log\frac1\epsilon).\mu$ such that $$\big|\mu'(f)-|\mu'|\mathrm m(f)\big|\leq \ref{effmeascoeffCONST} (\log\frac1\epsilon)^{-\frac12{\ref{effmeas'deltaconst}}^{-1}\mathcal F_{\phi(G)}^{-2}\delta}\|f\|_{\dot H^{\frac{d+1}2}}, \forall f\in\mathcal C^\infty(\mathbb T^d).$$\end{proposition}

\begin{proof}Let $\ref{effmeas'deltaconst}=5\ref{fcAinXdeltaconst}$, $\ref{efftopometricballCONST} =\ref{fcAinXmetricballconst}$, $\ref{effmeascoeffCONST} =\ref{sobolevconst}(d)(\sqrt{\ref{fcAinXboundconst}}+1)$ and apply Lemma \ref{sobolev} to $\mu'$ from Corollary \ref{fcA}. Note all constants are effective and depend only on the dimension $d$.\end{proof}

Proposition \ref{effmeas'} trivially implies the measure-theoretical form (Theorem \ref{effmeas}) of our main result.

\begin{proof}[Proof of Theorem \ref{effmeas}] By Theorem \ref{condGfield}, Condition \ref{condG'} is implied by Condition \ref{condG}, so Theorem \ref{effmeas'} applies. 

Let $\ref{efftopoepsilonCONST} =2^{-\max(\mathcal M_\psi ^{30d},4)}$, $\ref{effmeasdeltaCONST} =\ref{effmeas'deltaconst}\mathcal F_{\phi(G)}^2$, and $\ref{efftopometricballCONST} $, $\ref{effmeascoeffCONST} $ be the same as in Theorem \ref{effmeas'}. All these constants depend effectively on $d$, $\mathcal M_\psi $ and $\mathcal F_{\phi(G)}$, therefore evenually only on the group $G$.\end{proof}

\section{Topological results}\label{toporesults}

\hskip\parindent We now prove the main result of this paper, as well as two corollaries which give a clearer picture of the $G$-action on the torus.

\subsection{Density of the orbit of a dispersed set}

\hskip\parindent In parallel to the measure-theoretical Proposition \ref{effmeas'}, we prove first a topological result which assumes Condition \ref{condG'} and specifies how the constants depend on number-theoretical features of group $G$.

A subset $E$ of a metric space is said to be $\epsilon$-separated if any pair of points in $E$ are at least of distance $\epsilon$ from each other.

\declareCONST{efftopo'alpha} \declareCONST{efftopo'density}
\begin{proposition}\label{efftopo'}There are effective constants $\ref{efftopometricballCONST} (d)$, $\ref{efftopo'alpha}(d)$  and $\ref{efftopo'density}(d)$ such that:

Suppose $G<SL_d(\mathbb Z)$ satisfies Condition \ref{condG'}, if $\log\frac1\epsilon\geq\max(\mathcal M_\psi ^{30d},4)$ and some $\epsilon$-separated set $E$ has size $|E|\geq\epsilon^{-\alpha d}$ where $\alpha\in[\ref{efftopo'alpha}\mathcal F_{\phi(G)}^2\frac{\log\log\log\frac1\epsilon}{\log\log\frac1\epsilon},1)$, then $B_G^\mathrm{Mah}(\ref{efftopometricballCONST} \log\frac1\epsilon).E$ is $(\log\frac1\epsilon)^{-\ref{efftopo'density}\mathcal F_{\phi(G)}^{-2}\alpha}$-dense.\end{proposition}

\begin{proof} Let $\mu=\frac1{|E|}\sum_{x\in E}\delta_x$ be the uniform probability measure on $E$. For any partition $\mathcal P$ of the torus such that $\mathrm{diam}\mathcal P\leq\epsilon$, each atom contains at most one point from $E$. Thus $H_\mu(\mathcal P)=\sum_{P\in\mathcal P}-\mu(P)\log\mu(P)=\sum_{x\in E}-\frac1{|E|}\log\frac1{|E|}=\log|E|\geq\alpha d\log\frac1\epsilon$ satisfies Condition \ref{condmu}.

Fix a positive bump function $\theta\in\mathcal C^\infty(\mathbb R^d)$ supported on the ball $B(0,1)$ centered at the origin with radius $1$ such that $\int_{y\in\mathbb R^d}\theta(y)\mathrm d y=1$. For any $x\in\mathbb T^d$ and any positive number $\rho<\frac12$, on the ball $B(x,\rho)\subset\mathbb T^d$ centered at $x$ with radius $\rho$ we define $f(x+v)=\rho^{-d}\theta(\frac v\rho), v\in\mathbb R^d, |v|\leq r$; on the rest of $\mathbb T^d$ let $f$ be equal to $0$. Then $f$ is a positive function in $\mathcal C^\infty(\mathbb T^d)$ and $\mathrm m(f)=1$.

Remark \begin{align*}
\|f\|_{\dot H^{\frac{d+1}2}}\lesssim_d&\big(\sum_{\mathbf q\in\mathbb Z^d}\sum_{j=1}^d|q_j|^{d+1}|\hat f(\mathbf q)|^2\big)^{\frac12}\\
\leq\ &\Big(\sum_{j=1}^d\sum_{\mathbf q\in\mathbb Z^d}\big(q_j^{\lceil\frac{d+1}2\rceil}\hat f(\mathbf q)\big)^2\Big)^{\frac12}\\
=\ &\big(\sum_{j=1}^d\|(-\frac{\partial}{\partial x_j})^{\lceil\frac{d+1}2\rceil}f\|_{L^2(\mathbb T^d)}^2\big)^\frac12.
\end{align*}
Moreover, an easy fact is that the last expression is proportional to $\rho^{-\frac d2-\lceil\frac{d+1}2\rceil}$ once $\theta$ is fixed. Thus there is an effective dimensional constant \declaresubconst{efftoposubconst} $\ref{efftoposubconst}(d)$ such that $\|f\|_{\dot H^{\frac{d+1}2}}\leq\ref{efftoposubconst}\rho^{-\frac d2-\lceil\frac{d+1}2\rceil}$.

If $\alpha\geq \ref{efftopo'alpha}\mathcal F_{\phi(G)}^2\frac{\log\log\log\frac1\epsilon}{\log\log\frac1\epsilon}$ with $\ref{efftopo'alpha}=\max(8\ref{effmeas'deltaconst},3\ref{effmeascoeffCONST} \ref{efftoposubconst}\ref{voutierconst}^{-2})$,  which depends effectively on $d$, then as $\log\log\log\frac1\epsilon\geq\log\log4=1$,  we can apply Proposition \ref{effmeas'} to $\mu$ with $\delta=\frac\alpha2$ and get a measure $\mu'$ such that \begin{equation}\label{measurepositive}\begin{split}\mu'(f)\geq&(\alpha-\delta)\mathrm m(f)-\ref{effmeascoeffCONST} (\log\frac1\epsilon)^{-\frac12{\ref{effmeas'deltaconst}}^{-1}\mathcal F_{\phi(G)}^{-2}\delta}\cdot\ref{efftoposubconst}\rho^{-\frac d2-\lceil\frac{d+1}2\rceil}\\
=&\frac\alpha2-\ref{effmeascoeffCONST} \ref{efftoposubconst}(\log\frac1\epsilon)^{-\frac14{\ref{effmeas'deltaconst}}^{-1}\mathcal F_{\phi(G)}^{-2}\alpha}\cdot \rho^{-\frac d2-\lceil\frac{d+1}2\rceil}.\end{split}\end{equation}

Take \begin{equation}\rho_0=(\log\frac1\epsilon)^{-\frac18(\frac d2+\lceil\frac{d+1}2\rceil)^{-1}{\ref{effmeas'deltaconst}}^{-1}\mathcal F_{\phi(G)}^{-2}\alpha}.\end{equation}
Then (\ref{measurepositive}) rewrites \begin{equation}\label{positivesupport}\begin{split}\mu'(f)\geq&\frac\alpha2-\ref{effmeascoeffCONST} \ref{efftoposubconst}(\log\frac1\epsilon)^{-\frac18{\ref{effmeas'deltaconst}}^{-1}\mathcal F_{\phi(G)}^{-2}\alpha}\\
\geq&\frac12\ref{efftopo'alpha}\mathcal F_{\phi(G)}^2\frac{\log\log\log\frac1\epsilon}{\log\log\frac1\epsilon}-\ref{effmeascoeffCONST} \ref{efftoposubconst}(\log\frac1\epsilon)^{-\frac{\ref{efftopo'alpha}}{8\ref{effmeas'deltaconst}}\cdot\frac{\log\log\log\frac1\epsilon}{\log\log\frac1\epsilon}}\\
\geq&\frac12\ref{efftopo'alpha}\ref{voutierconst}^2\frac{1}{\log\log\frac1\epsilon}-\ref{effmeascoeffCONST} \ref{efftoposubconst}(\log\frac1\epsilon)^{-\frac{\log\log\log\frac1\epsilon}{\log\log\frac1\epsilon}}\\
\geq&(\frac12\ref{efftopo'alpha}\ref{voutierconst}^2-\ref{effmeascoeffCONST} \ref{efftoposubconst})\cdot\frac1{\log\log\frac1\epsilon}.
\end{split}\end{equation}

So as $\ref{efftopo'alpha}>2\ref{effmeascoeffCONST} \ref{efftoposubconst}\ref{voutierconst}^{-2}$, $\mu'(f)$ is strictly positive. In particular, $\mathrm{supp}\mu'\cap\mathrm{supp}f\neq\emptyset$. Since $\mathrm{supp}\mu=E$ and there is $\mu''$ in the convex hull of $B_F^\mathrm{Mah}(\ref{efftopometricballCONST} \log\frac1\epsilon).\mu$ such that $\mu'\leq\mu''$, $\mathrm{supp}\mu'\subset\mathrm{supp}\mu''\subset B_F^\mathrm{Mah}(\ref{efftopometricballCONST} \log\frac1\epsilon).E$. Moreover $f$ is supported on $B(x,\rho_0)$ for an arbitrary point $x\in\mathbb T^d$, therefore we proved $B_F^\mathrm{Mah}(\ref{efftopometricballCONST} \log\frac1\epsilon).E$ is $\rho_0$-dense. It suffices to let $\ref{efftopometricballCONST} $ be the same as in Proposition \ref{effmeas'} and set $\ref{efftopo'density}=\frac18(\frac d2+\lceil\frac{d+1}2\rceil)^{-1}{\ref{effmeas'deltaconst}}^{-1}$ to conclude.\end{proof}

\begin{proof}[Proof of Theorem \ref{efftopo}]By Theorem \ref{condGfield}, Proposition \ref{efftopo'} applies. 

Let $\ref{efftopoepsilonCONST} =2^{-\max(\mathcal M_\psi ^{30d},4)}$, $\ref{efftopometricballCONST} $ be the same as in Theorem \ref{effmeas'}, $\ref{efftopoalphaCONST} =\ref{efftopo'alpha}\mathcal F_{\phi(G)}^2$ and $\ref{efftopodensityCONST} =\ref{efftopo'density}\mathcal F_{\phi(G)}^{-2}$; all of which are effective and depend on $d$, $\mathcal M_\psi $ and $\mathcal F_{\phi(G)}$, thus eventually only on $G$ itself.\end{proof}

\subsection{Density of the orbit of a large set}

\hskip\parindent In Theorems \ref{efftopo'} and \ref{efftopo}, we assumed the set $E$ is sufficiently dispersed with respect to its size. As a corollary to these theorems, the results in this subsection drop this hypothesis. We are going to show that the full $G$-orbit of any sufficiently large subset $E$ in $\mathbb T^d$ is going to have some density. 

If we care about a metric ball $B_G^\mathrm{Mah}(L).E$ with respect to the logarithmic Mahler measure (or the word metric) on $G$ in the orbit instead to the full orbit, then it is still necessary to assume $E$ is $\epsilon_0$-separated for certain $\epsilon_0$. Nevertheless, $\epsilon_0$ is not going to rely on $|E|$. Moreover, $\epsilon_0$ only affects how large the metric ball has to be but has no impact on the effective density we obtain. 

\declareCONST{largetopo'boundbyMCONST} \declareCONST{largetopo'boundbyFCONST} \declareCONST{largetopometricballCONST} \declareCONST{largetopo'densityCONST} 
\begin{proposition}\label{largetopo'} Suppose $G$ satisifes Condition \ref{condG'} and the image $\psi(E)$ of a finite set $E\subset\mathbb T^d$ under $\psi$ is $\epsilon_0\mathcal S_\Gamma$-separated in $X=\psi(\mathbb T^d)$. If \begin{equation}\label{largeassum} \log\log |E|\geq\max\big(\ref{largetopo'boundbyMCONST} \mathcal M_\psi^{30d},\max(\mathcal F_{\phi(G)},2)^{\ref{largetopo'boundbyFCONST} \mathcal F_{\phi(G)}^2}\big),\end{equation} then the set $B_G^\mathrm{Mah}(\log\frac1{\epsilon_0}+\ref{largetopometricballCONST} \log |E|).E$ is $(\log\log |E|)^{-\ref{largetopo'densityCONST} \mathcal F_{\phi(G)}^2}$-dense, where $\ref{largetopo'boundbyMCONST} ,\cdots,\ref{largetopo'densityCONST} $ are effective constants that depend only on $d$. \end{proposition}

\begin{remark}\label{largermk}Here we adopt the assumption that $\psi(E)$ is $\epsilon_0\mathcal S_\Gamma$ dense in $X$ because this would be the more convenient formulation for later applications. However it is completely fine to replace the assumption by a more natural one, namely $E$ is $\epsilon_0$-separated in $\mathbb T^d$. Actually if this holds then $\psi(E)$ is $\|\psi^{-1}\|^{-1}\epsilon_0$-separated in $X$. However $\|\psi^{-1}\|^{-1}\epsilon_0\geq\mathcal M_\psi^{-1}\epsilon_0\mathcal S_\Gamma$. So the proposition applies with $\mathcal M_\psi^{-1}\epsilon_0$ in place of $\epsilon_0$ and concludes that the set $B_G^\mathrm{Mah}(\log\frac1{\epsilon_0}+\log\mathcal M_\psi+\ref{largetopometricballCONST} \log |E|).E$ is $(\log\log |E|)^{-\ref{largetopo'densityCONST} \mathcal F_{\phi(G)}^2}$-dense. However since $\log|E|\gg\log\log|E|\gtrsim_{d}\mathcal M_\psi^{30d}\gg\log\mathcal M_\psi$, by making $\ref{largetopometricballCONST} $ slightly larger the term $\log\mathcal M_\psi$ can be absorbed into $\ref{largetopometricballCONST} \log |E|$ and the original claim still holds.\end{remark}

\begin{proof}[Proof of Proposition]
As the torus $\mathbb T^d$ can be covered by at most $|E|-1$ balls of radius $\frac{\sqrt d}{\lfloor\sqrt[d]{|E|-1}\rfloor}\lesssim_d |E|^{-\frac1d}$, there are two elements $z,z'\in E$ such that $|z-z'|\lesssim_d |E|^{-\frac1d}$, where $|z-z'|$ denotes the distance in $\mathbb T^d$

We select lifts $\tilde z,\tilde z'\in\mathbb R^d$ of $z,z'\in\mathbb T^d$ such that $|\tilde z-\tilde z'|=|z-z'|$. Denote $\tilde y=\psi(\tilde z-\tilde z')\in \mathbb R^d$, then by the fact $\|\psi\|\leq\mathcal M_\psi\mathcal S_\Gamma$, \declaresubconst{largetopo'subconst1} $|\tilde y|\leq\ref{largetopo'subconst1}\mathcal M_\psi |E|^{-\frac1d}\mathcal S_\Gamma,$ where $\ref{largetopo'subconst1}$ depends effectively on $d$. On the other hand $\tilde y$ is a lift of $\psi(z)-\psi(z')$, thus $|\tilde y|\geq\epsilon_0\mathcal S_\Gamma$ by assumption.

$\tilde y$ can be decomposed as $\sum_{j=1}^{r_1+r_2}\tilde y_j$ where $\tilde y_j$ is the orthogonal projection onto $V_j$. Without loss of generality, assume $|\tilde y_i|\geq|\tilde y_j|, \forall j\neq i$; then \begin{equation}\label{coordlower}|\tilde y_i|\geq (r_1+r_2)^{-\frac12}\epsilon_0 S_\Gamma\end{equation} but \begin{equation}\label{coordupper}|\tilde y_j|\leq |\tilde y_i|\leq\ref{largetopo'subconst1}\mathcal M_\psi |E|^{-\frac1d}\mathcal S_\Gamma, \forall j\neq i.\end{equation}

Let $s$ be the largest positive integer such that \begin{equation}\label{sbound}2^{s^{10}}\ref{largetopo'subconst1}|E|^{-\frac1{3d}}\leq \frac14(r_1+r_2)^{-\frac12},\end{equation} then by the fact $\mathcal F_\phi(G)\gtrsim_d1$, for some effective constant \declaresubconst{largetopo'subconst2} $\ref{largetopo'subconst2}(d)$ under the assumption \begin{equation}\label{largeassum1}\log |E|\geq\ref{largetopo'subconst2}\mathcal F_{\phi(G)}^{10}\end{equation} it is guaranteed that \begin{equation}\label{ssize}s\sim_{d}(\log |E|)^\frac1{10}\end{equation} and $$s\geq\max\big(\ref{arithprolengthconst}\mathcal F_{\phi(G)},4\big),$$ where $\ref{arithprolengthconst}$ was given by Proposition \ref{arithpro}.

Therefore we may apply Propositon \ref{arithpro} to construct a sequence of elements $a_0,\cdots,a_{s-1}\in G$. We are interested in the set $\{\times_{\phi(a_t)}.\tilde y\}_{t=0,\cdots,s-1}$. Notice the $j$-th coordinate of the $t$-th element is $\zeta_{a_t}^j\tilde y_j$.

By Propositon \ref{arithpro}, $|\zeta_{a_t}^i-(1+t\Delta)|\leq s^{-1}|\Delta|\leq\frac14|\Delta|$. Thus for any two distinct $t$, $t'$, $|\zeta_{a_t}^i\tilde y_i-\zeta_{a_{t'}}^i\tilde y_i|$ is bounded from below by $(|\Delta|-2\cdot\frac14|\Delta|)y_i\geq \frac{|\Delta|}2|y_i|$, and from above by $\big((s-1)|\Delta|+2\cdot\frac14|\Delta|\big)|\tilde y_i|\leq s|\Delta|\cdot|\tilde y_i|$, where $s^{-\ref{arithproexpoconst}\mathcal F_{\phi(G)}^2}\leq|\Delta|\leq s^{-3}$.

Assume in addition \begin{equation}\label{largeassum2}|E|\geq\mathcal M_\psi^{3d},\end{equation} then $\mathcal M_\psi\mathcal |E|^{-\frac1d}\leq\mathcal M_\psi^{-1}|E|^{-\frac1{3d}}$.

So as long as (\ref{largeassum1}) holds with sufficiently large $\ref{largetopo'subconst2}$, as $\mathcal F_{\phi(G)}\gtrsim_d 1$ and $r_1+r_2\leq d-1$,
\begin{align*}s|\Delta|\cdot|\tilde y_i|\leq&s\cdot s^{-3}\cdot\ref{largetopo'subconst1}\mathcal M_\psi |E|^{-\frac1d}\mathcal S_\Gamma\\
\leq&\ref{largetopo'subconst1}\mathcal M_\psi^{-1} |E|^{-\frac1{3d}}\mathcal S_\Gamma\\
\leq&\frac12(r_1+r_2)^{-\frac12}\mathcal M_\psi^{-1}\mathcal S_\Gamma.\end{align*}

Let $u\in G$ be given by Proposition \ref{expand}, then $|\zeta_u^i|\geq 2^{d(\frac r2+1)\mathcal F_{\phi(G)}}$ and $|\zeta_u^j|<1$ for all $j\neq i$. By the observation above, we may take an integer $n\geq 0$ such that \begin{equation}\label{nupper}|\zeta_u^i|^{-1}\cdot\frac12(r_1+r_2)^{-\frac12}\mathcal M_\psi^{-1}\mathcal S_\Gamma\leq|\zeta_u^i|^ns|\Delta|\cdot|\tilde y_i|\leq\frac12(r_1+r_2)^{-\frac12}\mathcal M_\psi^{-1}\mathcal S_\Gamma,\end{equation} which this also implies \begin{equation}\label{nlower}|\zeta_u^i|^n\cdot s|\Delta|\cdot|\tilde y_i|\geq|\zeta_u^i|^{-1}\cdot\frac12(r_1+r_2)^{-\frac12}\mathcal M_\psi^{-1}\mathcal S_\Gamma.\end{equation}

For any pair $t\neq t'$, \begin{equation}\label{coordiupper}\begin{split}|\zeta_{u^na_t}^i\tilde y_i-\zeta_{u^na_{t'}}^i\tilde y_i|=&|\zeta_u^i|^n|\zeta_{a_t}^i\tilde y_i-\zeta_{a_{t'}}^i\tilde y_i|\leq|\zeta_u^i|^n\cdot s|\Delta|\cdot|\tilde y_i|\\
\leq&\frac12(r_1+r_2)^{-\frac12}\mathcal M_\psi^{-1}\mathcal S_\Gamma.\end{split}\end{equation} On the other hand, by Propositon \ref{arithpro}, $h^\mathrm{Mah}(\phi(a_t))\leq s^{10}$, hence $|\zeta_{a_t}^j|\leq 2^{s^{10}}, \forall j$. Therefore if $j\neq i$, using $|\zeta_u^j|<1$ and (\ref{coordupper}), (\ref{sbound}), (\ref{largeassum2}), 
\begin{equation}\label{coordjupper}\begin{split}
&|\zeta_{u^na_t}^j\tilde y_j-\zeta_{u^na_{t'}}^j\tilde y_j|\\
\leq&|\zeta_u^j|^n(|\zeta_{a_t}^j|+|\zeta_{a_{t'}}^j|)|\tilde y_j|\leq 1\cdot (2^{s^{10}}+2^{s^{10}})\ref{largetopo'subconst1}\mathcal M_\psi |E|^{-\frac1d}\mathcal S_\Gamma\\
\leq&2\cdot2^{s^{10}}\ref{largetopo'subconst1}\mathcal M_\psi^{-1}|E|^{-\frac1{3d}}\mathcal S_\Gamma\leq2\cdot\frac14(r_1+r_2)^{-\frac12}\mathcal M_\psi^{-1}\mathcal S_\Gamma\\
=&\frac12(r_1+r_2)^{-\frac12}\mathcal M_\psi^{-1}\mathcal S_\Gamma.\end{split}\end{equation}

By (\ref{coordiupper}) and(\ref{coordjupper}), the orthogonal projection of $\times_{\phi(u^na_t)}.\tilde y-\times_{\phi(u^na_{t'})}.\tilde y$ is bounded by $\frac12(r_1+r_2)^{-\frac12}$ in each $V_j$, $1\leq j\leq r_1+r_2$, hence $$|\times_{\phi(u^na_t)}.\tilde y-\times_{\phi(u^na_{t'})}.\tilde y|\leq \frac12\mathcal M_\psi^{-1}\mathcal S_\Gamma.$$ Since $\times_{\phi(u^na_t)}.\tilde y=\psi\big(u^na_t.(\tilde z-\tilde z')\big), \forall t$ and $\|\psi^{-1}\|\leq\mathcal M_\psi\mathcal S_\Gamma^{-1}$, $$|u^na_t.(\tilde z-\tilde z')-u^na_{t'}.(\tilde z-\tilde z')|\leq\|\psi^{-1}\|\cdot|\times_{\phi(u^na_t)}.\tilde y-\times_{\phi(u^na_{t'})}.\tilde y|\leq \frac12.$$

Because $u^na_t.(\tilde z-\tilde z')\in\mathbb R^d$ is the lift of $u^na_t.(z-z')\in \mathbb T^d$ for all $t$, this bound implies the distance $|u^na_t.(z-z')-u^na_{t'}.(z-z')|$ between $u^na_t.(z-z')$ and $u^na_{t'}.(z-z')$ on $\mathbb T^d$ is exactly $|u^na_t.(\tilde z-\tilde z')-u^na_{t'}.(\tilde z-\tilde z')|$. Therefore \begin{equation}\label{deviationinX}\begin{split}
&|u^na_t.(z-z')-u^na_{t'}.(z-z')|\\
=&|u^na_t.(\tilde z-\tilde z')-u^na_{t'}.(\tilde z-\tilde z')|\\
\geq&\|\psi\|^{-1}|\times_{\phi(u^na_t)}.\tilde y-\times_{\phi(u^na_{t'})}.\tilde y|\\
\geq&\mathcal M_\psi^{-1}\mathcal S_\Gamma^{-1}|\times_{\phi(u^na_t)}.\tilde y_i-\times_{\phi(u^na_{t'})}.\tilde y_i|\\
\geq&\mathcal M_\psi^{-1}\mathcal S_\Gamma^{-1}|\zeta_u^i|^n|\times_{\phi(a_t)}.\tilde y_i-\times_{\phi(a_{t'})}.\tilde y_i|\\
\geq&\mathcal M_\psi^{-1}\mathcal S_\Gamma^{-1}|\zeta_u^i|^n\cdot\frac{|\Delta|}2|\tilde y_i|.\end{split}\end{equation}
By (\ref{nlower}),
\begin{align*}(\ref{deviationinX})\geq&\mathcal M_\psi^{-1}\mathcal S_\Gamma^{-1}\cdot\frac12s^{-1}\cdot|\zeta_u^i|^{-1}\cdot\frac12(r_1+r_2)^{-\frac12}\mathcal M_\psi^{-1}\mathcal S_\Gamma\\
\geq&\frac14(r_1+r_2)^{-\frac12}2^{-h^\mathrm{Mah}(\phi(u))}\mathcal M_\psi^{-2}s^{-1}\\
\geq&\frac14(d-1)^{-\frac12}2^{-9d(\frac r2+1)\mathcal F_{\phi(G)}}\mathcal M_\psi^{-2}s^{-1}
\end{align*}

By (\ref{ssize}), \declaresubconst{largetopo'subconst3} $s\geq\ref{largetopo'subconst3}(\log |E|)^{\frac1{10}}$ for an effective $\ref{largetopo'subconst3}=\ref{largetopo'subconst3}(d)$, so under a new assumption \declaresubconst{largetopo'subconst4}\begin{equation}\label{largeassum3}\log |E|\geq\max(2^{\ref{largetopometricballCONST} \mathcal F_{\phi(G)}},\mathcal M_\psi^{30})\end{equation} where a sufficiently large effective constant $\ref{largetopometricballCONST} =\ref{largetopometricballCONST} (d)$ is used, we have $\mathcal M_\psi^2<\ref{largetopo'subconst3}^{-\frac23}s^\frac23$ and $4(d-1)^{\frac12}2^{9d(\frac r2+1)\mathcal F_{\phi(G)}}<\frac12\ref{largetopo'subconst3}^{\frac23}s^{\frac13}$. Thus \begin{equation}(\ref{deviationinX})\geq\ref{largetopo'subconst3}^{\frac23}s^{-\frac23}\cdot2\ref{largetopo'subconst3}^{-\frac23}s^{-\frac13}\cdot s^{-1}\geq 2s^{-2},\end{equation} which can be interpreted as the set $E_1:=\{u^na_t.( z-z')|0\leq t\leq s-1\}$ is $2s^{-2}$-separated in $\mathbb T^d$. 

Let $E_2=\{u^na_t.x|0\leq t\leq s-1, x\in E\}\subset\mathbb T^d$ and observe $$E_1\subset E_2-E_2:=\{w-w'|w,w'\in E_2\}.$$ Take a maximal $s^{-2}$-separated subset $E_3$ of $E_2$, we claim $|E_3|\geq\sqrt s$. Actually, $E_2$ is covered by the union of $|E_3|$ balls $\cup_{v\in E_3}B_v$, where $B_v$ is centered at $v$ with radius $s^{-2}$. Then the set $E_1\subset E_2-E_2$ is covered by $\cup_{v,v'\in E}(B_v-B_{v'})$. But for each pair $v,v'$, the set $B_v-B_v'$ is a ball of radius $2s^{-2}$ centerd at $v-v'$, hence contains at most one point from $E_1$. So the number of pairs, which equals $|E_3|^2$, is at least $|E_1|=s$. This proves the claim.

We want to apply Proposition \ref{efftopo'} to $\epsilon:=s^{-2}$ and $|E_3|\geq s^\frac12=\epsilon^{-\frac14}=\epsilon^{-\alpha d}$ where $\alpha=\frac1{4d}$, for which purpose it is necessary to have $\log\frac1\epsilon\geq\max(\mathcal M_\psi^{30d},4)$ and $\ref{efftopo'alpha}\mathcal F_{\phi(G)}^2\frac{\log\log\log\frac1\epsilon}{\log\log\frac1\epsilon}\leq\frac1{4d}$. Because \begin{equation}\label{sizesep}\frac1\epsilon=s^2\geq\ref{largetopo'subconst3}^2(\log |E|)^\frac15\end{equation} and we assumed $|E|$ is sufficiently large with respect to $d$, these conditions would respectively follow from
\begin{equation}\label{largeassum4}\log\log |E|\geq \ref{largetopo'boundbyMCONST} \max(\mathcal M_\psi^{30d},4);\end{equation} and
\begin{equation}\label{largeassum5}\log\log\log |E|\geq \ref{largetopo'boundbyFCONST} \mathcal F_{\phi(G)}^2\max(\log\mathcal F_{\phi(G)},1),\end{equation}
if $\ref{largetopo'boundbyMCONST} $ and $\ref{largetopo'boundbyFCONST} $ are sufficiently large constants that depend effectively on $d$.

Again by $\mathcal F_{\phi(G)}\gtrsim_d1$, if $\ref{largetopo'boundbyFCONST} $ is large enough with respect to $\ref{largetopo'boundbyMCONST} $ then (\ref{largeassum5}) implies the $\ref{largetopo'boundbyMCONST} \cdot 4$ part in (\ref{largeassum4}). So we can combine the two inequalities into assumption (\ref{largeassum}) in the statement of theorem.

Furthermore, notice that (\ref{largeassum}) actually contains all the earlier assumptions (\ref{largeassum1}), (\ref{largeassum2}) and (\ref{largeassum3}) given that $\ref{largetopo'boundbyMCONST} $ and $\ref{largetopo'boundbyFCONST} $ are large with respect to all previous constants.

So eventually we can apply Proposition \ref{efftopo'} to claim $E_3$ is $(\log\frac1\epsilon)^{-\frac1{4d}\ref{efftopo'density}\mathcal F_{\phi(G)}^2}$-dense. By (\ref{sizesep}) when $|E|$ is sufficiently large as in our situation, $(\log\frac1\epsilon)^{-\frac1{4d}\ref{efftopo'density}\mathcal F_{\phi(G)}^2}\leq (\log\log |E|)^{-\ref{largetopo'densityCONST} \mathcal F_{\phi(G)}^2}$ for some effective constant $\ref{largetopo'densityCONST} =\ref{largetopo'densityCONST} (d)$. 

Thus $E_3$ is $(\log\log |E|)^{-\ref{largetopo'densityCONST} \mathcal F_{\phi(G)}^2}$-dense.

Observe $$E_3\subset E_2=\{u^na_t|0\leq t\leq s-1\}.E.$$ It is clear that \begin{equation}\label{heightat}h^\mathrm{Mah}(\phi(a_t))\leq s^{10}\sim_{d}\big((\log |E|)^\frac1{10}\big)^{10}=\log |E|.\end{equation} Now we investigate how large $u^n$ can be. 

By (\ref{coordlower}) and (\ref{nupper}), $|\zeta_u^i|^n\leq \epsilon_0^{-1}\cdot s^{-1}|\Delta|^{-1}\cdot\frac12\mathcal M_\psi^{-1}\leq \epsilon_0^{-1}s^{\ref{arithproexpoconst}\mathcal F_{\phi(G)}^2-1}$ as $\mathcal M_\psi\geq 1$ and $|\Delta|\geq s^{-\ref{arithproexpoconst}\mathcal F_{\phi(G)}^2}$. Because $|\zeta_u^i|>1$ but $|\zeta_u^j|<1,\forall j\neq i$ the logarithmic Mahler measure $h^\mathrm{Mah}(\phi(u))$ equals $\log|\zeta_u^i|$, so  \begin{equation}\label{heightun}\begin{split}h^\mathrm{Mah}(\phi(u^n))=&nh^\mathrm{Mah}(\phi(u))\leq n\log|\zeta_u^i|\leq \log(\epsilon_0^{-1}s^{\ref{arithproexpoconst}\mathcal F_{\phi(G)}^2-1})\\
\leq&\log\frac1{\epsilon_0}+\ref{arithproexpoconst}\mathcal F_{\phi(G)}^2\log s.\end{split}\end{equation}
Remark $\log s\lesssim_{d}\log\log |E|$ and by (\ref{largeassum}), $\log\log |E|\gg\mathcal F_{\phi(G)}^2$; thus $\ref{arithproexpoconst}\mathcal F_{\phi(G)}^2\log s\ll\ref{arithproexpoconst}(\log\log |E|)^2$. Therefore (\ref{heightat}) and (\ref{heightun}) give a bound to the logarithmic Mahler measure of $u^na_t, \forall t=0,\cdots,s-1$: $$m(u^na_t)=h^\mathrm{Mah}(\phi(u^na_t))\leq \log\frac1{\epsilon_0}+\ref{largetopometricballCONST} \log |E|.$$
So $E_3\subset B_G^\mathrm{Mah}(\log\frac1{\epsilon_0}+\ref{largetopometricballCONST} \log |E|).E$ for some effective constant $\ref{largetopometricballCONST}(d)$, which proves the result.\end{proof}

\declareCONST{largetoposizeCONST} \declareCONST{largetopodensityCONST}
\begin{theorem}\label{largetopo} Suppose $G$ satisies Condition \ref{condG}. Then there are effective constants $\ref{largetopometricballCONST}$, $\ref{largetoposizeCONST} $, $\ref{largetopodensityCONST} $,  which depend only on the group $G$ such that $\forall\epsilon_0$, for any $\epsilon_0$-separated subset $E\subset\mathbb T^d$ of size $|E|\geq \ref{largetoposizeCONST} $, the set $B_G^\mathrm{Mah}(\log\frac1{\epsilon_0}+\ref{largetopometricballCONST} \log|E|).E$ is $(\log\log|E|)^{-\ref{largetopodensityCONST} }$-dense.\end{theorem}

\begin{proof} The theorem follows from Theorem \ref{condGfield}, Proposition \ref{largetopo'} (together with Remark \ref{largermk}) and the fact that $d$, $\mathcal M_\psi$ and $\mathcal F_{\phi(G)}$ are determined by $G$.\end{proof}

In particular, we have the following corollary:

\begin{corollary} If Condition \ref{condG}, any $G$-invariant finite subset $E\subset\mathbb T^d$ of size $|E|\geq \ref{largetoposizeCONST}$ is $(\log\log|E|)^{-\ref{largetopodensityCONST} }$-dense where $\ref{largetoposizeCONST}=\ref{largetoposizeCONST}(G)$, $\ref{largetopodensityCONST}=\ref{largetopodensityCONST}(G)$.\end{corollary}

Remark if $E$ contains an irrational point, then it is dense in $\mathbb T^d$ by Berend \cite{B83}. So the corollary is meaningful only for rational subsets. 

\subsection{Density of the orbit of a single point}

\hskip\parindent We now study a single orbit instead of the orbit of a set.

\declareCONST{dioQ'boundMCONST} \declareCONST{dioQ'boundFCONST} \declareCONST{dioQ'densityCONST}
\begin{proposition}\label{dioQ'} Let $G$ be as in Condition \ref{condG'} and $Q\in\mathbb N$ be such that \begin{equation}\label{dioQassum}\log\log\log Q\geq\max\big(\ref{dioQ'boundMCONST}\mathcal M_\psi^{30d}, \max(\mathcal F_{\phi(G)},2)^{\ref{dioQ'boundFCONST}\mathcal F_{\phi(G)}^2}\big).\end{equation} 

If a point $x\in\mathbb T^d$ satisfies either condition (i) or (ii) in Proposition \ref{dioQ}, Then the set $B_G^{\mathrm{Mah}}\big((k+2)\log Q\big).x$ is $(\log\log\log Q)^{-\ref{dioQ'densityCONST}\mathcal F_{\phi(G)}^2}$-dense. Here $\ref{dioQ'boundMCONST}$, $\ref{dioQ'boundFCONST}$, $\ref{dioQ'densityCONST}$ are effective constants that depend only on $d$ .

\end{proposition}

Before proving the theorem, we construct a metric ball inside the orbit $G.x$ which is large in cardinality and satisfies, although quite weak, some separatedness condition.

\begin{lemma}\label{dioQnonrec} Assuming Condition \ref{condG'}, for the element $g\in G$ and the constant $\ref{totirrheightconst}=\ref{totirrheightconst}(d)$ in Proposition \ref{totirrheight}, if $x\in\mathbb T^d$ meets one of the conditions in Proposition \ref{dioQ} then $\forall m\neq m'$, $0\leq m,m'<\mathcal F_{\phi(G)}^{-1}\log(2^{-d}Q)$, the distance $|\psi(g^m.x)-\psi(g^{m'}.x)|$ between $\psi(g^m.x)$ and $\psi(g^{m'}.x)$ is at least $\mathcal M_\psi^{-1}Q^{-k}\mathcal S_\Gamma$, where $k=1$ if $x$ satisifes condition (ii) in Theorem \ref{dioQ}.\end{lemma}

\begin{proof} We may assume $m>m'$. Let $\tilde x\in \mathbb R^d$ be a lift of $x$ in $\mathbb T^d$. By definition of $\mathcal F_{\phi(G)}$, we may take a toral automorphism $g$ from $G$ with $0<h^\mathrm{Mah}(\phi(g))\leq\mathcal F_{\phi(G)}$. The distance $\mathrm |\psi(g^m.x)-\psi(g^{m'}.x)|$ in $\mathbb T^d$ is given by $|\psi\big((g^m-g^{m'}).\tilde x-\Xi\big)|$ where $\Xi$ is some vector from $\mathbb Z^d$.

As $h^\mathrm{Mah}(\phi(g))>0$, $g^{m-m'}-\mathrm{id}$ is an invertible matrix thus so is $g^m-g^{m'}$. (Otherwise one of the eigenvalues $\zeta_g^i=\sigma_i(\phi(g))$ of $g$ is a root of unity, so $h^\mathrm{Mah}(\phi(g))=h^\mathrm{Mah}(\zeta_g^i)=0$.)

Remark $\tilde x-(g^m-g^{m'})^{-1}\Xi=(g^m-g^{m'})^{-1}.\big((g^m-g^{m'}).\tilde x-\Xi\big)$. Hence since $(g^m-g^{m'})^{-1}$ is simultaneously diagonalizable with $g$ over $\mathbb C$, $$\psi(\tilde x-(g^m-g^{m'})^{-1}\Xi)=\times_{\phi\big((g^m-g^{m'})^{-1}\big)}.\psi\big((g^m-g^{m'}).\tilde x-\Xi\big).$$ Because $\times_{\phi\big((g^m-g^{m'})^{-1}\big)}$ acts as a multiplication on each component of $\mathbb R^d=\oplus_{i=1}^{r_1+r_2}V_i$, we see \begin{equation}\label{deviation}\begin{split}
&\tilde x-(g^m-g^{m'})^{-1}\Xi\\
\leq&\|\psi^{-1}\|\cdot\big|\psi\big(\tilde x-(g^m-g^{m'})^{-1}\Xi\big)\big|\\
\leq&\mathcal M_\psi\mathcal S_\Gamma^{-1}\max_{1\leq i\leq d}\big(\big|(\zeta^i_{g})^m-(\zeta^i_{g})^{m'}\big|^{-1}\big)\big|\psi\big((g^m-g^{m'}).\tilde x-\Xi\big)\big|\\
=&\mathcal M_\psi\mathcal S_\Gamma^{-1}\max_{1\leq i\leq d}\big(\big|(\zeta^i_{g})^m-(\zeta^i_{g})^{m'}\big|^{-1}\big)\big|\psi(g^m.x)-\psi(g^{m'}.x)\big|.\end{split}\end{equation}

It follows from the inequality $\max(|z-1|,1)\leq2\max(|z|,1),\forall z\in\mathbb C$ and the fact $\prod_{1\leq i\leq d}\max(|\zeta^i_{g}|,1)=2^{h^\mathrm{Mah}(\phi(g))}\leq 2^{\mathcal F_{\phi(G)}}$ that 
\begin{equation}\label{detQ}\begin{split}&|\det(g^m-g^{m'})|\\
=&\prod_{1\leq i\leq d}\big(|(\zeta^i_{g})^{m-m'}-1|\cdot|\zeta^i_{g}|^{m'}\big)\\
\leq&\prod_{1\leq i\leq d}\Big(\max\big(|(\zeta^i_{g})^{m-m'}-1|,1\big)\cdot\max(|\zeta^i_{g}|^{m'},1)\Big)\\
\leq& 2^d\prod_{1\leq i\leq d}\big(\max(|\zeta^i_{g}|^{m-m'},1)\cdot\max(|\zeta^i_{g}|^{m'},1)\big)\\
=&2^d\prod_{1\leq i\leq d}\max(|\zeta^i_{g}|,1)^m\leq 2^{d+m\mathcal F_{\phi(G)}}<Q.\end{split}\end{equation}

(i) Suppose $x$ is diophantine generic. 

Since $\det(g^m-g^{m'})\cdot(g^m-g^{m'})^{-1}\in M_d(\mathbb Z)$, $(g^m-g^{m'})^{-1}\Xi$ can be written in the form $\frac vq$ where $v\in\mathbb Z^d$ and $q\in\mathbb N$ is a factor of $\det(g^m-g^{m'})$; in particular, $q\leq Q$. Because $\tilde x$ satisfies the same diophantine property as $x$,
Thus by (\ref{deviation}), there exists $1\leq j\leq d$ such that $$\mathcal M_\psi\mathcal S_\Gamma^{-1}|(\zeta^j_{g})^m-(\zeta^j_{g})^{m'}|^{-1}|\psi(g^m.x)-\psi(g^{m'}.x)|\geq\prod_{1\leq i\leq d}|(\zeta^i_{g})^m-(\zeta^i_{g})^{m'}|^{-k},$$ therefore

\begin{equation}\label{recdevest}\begin{split}&|\psi(g^m.x)-\psi(g^{m'}.x)|\\
\geq&\mathcal M_\psi^{-1}\mathcal S_\Gamma|(\zeta^j_{g})^m-(\zeta^j_{g})^{m'}|\cdot\prod_{1\leq i\leq d}|(\zeta^i_{g})^m-(\zeta^i_{g})^{m'}|^{-k}\\
=&\mathcal M_\psi^{-1}\mathcal S_\Gamma|(\zeta^j_{g})^m-(\zeta^j_{g})^{m'}|^{-(k-1)}\cdot\prod_{1\leq i\leq d,i\neq j}|(\zeta^i_{g})^m-(\zeta^i_{g})^{m'}|^{-k}\\
\geq&\mathcal M_\psi^{-1}\mathcal S_\Gamma\prod_{1\leq i\leq d}\max(|(\zeta^i_{g})^m-(\zeta^i_{g})^{m'}|,1)^{-k}\\
\geq&\mathcal M_\psi^{-1}\mathcal S_\Gamma\prod_{1\leq i\leq d}\Big(\max(|(\zeta^i_{g})^{m-m'}-1|,1)^{-k}\cdot\max(|(\zeta^i_{g})|^{m'},1)^{-k}\Big)\\
\geq&\mathcal M_\psi^{-1}\mathcal S_\Gamma Q^{-k},\end{split}\end{equation}
where the last step was deduced as in (\ref{detQ}). This completes the proof in case (i).

(ii) If $x$ is rational with denominator $Q$, then $g^m.\tilde x-g^{m'}.\tilde x-\Xi$ has denominator $Q$ as well, thus $|g^m.\tilde x-g^{m'}.\tilde x-\Xi|\geq Q^{-1}$ as long as it doesn't vanish. Since $\|\psi^{-1}\|\leq\mathcal M_\psi\mathcal S_\Gamma^{-1}$ by definition, $|\psi(g^m.x)-\psi(g^{m'}.x)|=|\psi\big(g^m.\tilde x-g^{m'}.\tilde x-\Xi\big)|\geq M_\psi^{-1}\mathcal S_\Gamma Q^{-1}$. Therefore if suffices to show $g^m.\tilde x-g^{m'}.\tilde x-\Xi\neq 0$.

Suppose by absurd that $g^m.\tilde x-g^{m'}.\tilde x-\Xi=0$, then $\tilde x=(g^m-g^{m'})^{-1}\Xi$. As we already found out in case (i), the vector $(g^m-g^{m'})^{-1}\Xi$ is rational and its denominator is at most $|\det(g^m-g^{m'})|< Q$. But $\tilde x$ has denominator $Q$, contradiction.
\end{proof}

\begin{proof}[\hypertarget{proofdioQ'}{Proof of Proposition \ref{dioQ'}}] Fix $g\in G$ with $0<h^\mathrm{Mah}(\phi(g))\leq\mathcal F_{\phi(G)}$, which is possible by the definition of $\mathcal F_{\phi(G)}$. By lemma, if we denote \begin{equation}\label{startingblock}E:=\{g^m.x|0\leq m<\min(\mathcal F_{\phi(G)}^{-1},1)\log(2^{-d}Q)\},\end{equation} then the $g^m.x$'s are all different for different $m$'s and $\psi(E)$ is $\mathcal M_\psi^{-1}Q^{-k}\mathcal S_\Gamma$-separated.

By assumption (\ref{dioQassum}), when $\ref{dioQ'boundMCONST}$ and $\ref{dioQ'boundFCONST}$ are sufficiently large with respect to $d$, we can safely claim that $\min(\mathcal F_{\phi(G)}^{-1},1)\log(2^{-d}Q)\geq(\log Q)^\frac12$. Hence $$\log\log|E|\gtrsim_{d}\log\log\log Q\geq \big(\ref{dioQ'boundMCONST}\mathcal M_\psi^{30d}, \max(\mathcal F_{\phi(G)},2)^{\ref{dioQ'boundFCONST}\mathcal F_{\phi(G)}^2}\big).$$ By choosing sufficiently large (but still effective) constants $\ref{dioQ'boundMCONST}$ and $\ref{dioQ'boundFCONST}$, this implies assumption (\ref{largeassum}) (using again $\mathcal F_{\phi(G)}\gtrsim_d1$).

Therefore, by Proposition \ref{largetopo'}, the set $B_G^\mathrm{Mah}\big(\log(\mathcal M_\psi Q^k)+\ref{largetopometricballCONST} \log|E|).E$ is $(\log\log |E|)^{-\ref{largetopo'densityCONST} \mathcal F_{\phi(G)}^2}$-dense. 

However, by the definition of $E$, it is contained in $$B_G^\mathrm{Mah}\big(\mathcal F_{\phi(G)}^{-1}\log(2^{-d}Q)\cdot h^\mathrm{Mah}(\phi(g))\big).x.$$ But $h^\mathrm{Mah}(\phi(g))\leq\mathcal F_{\phi(G)}$, so $E\subset B_G^\mathrm{Mah}\big(\log(2^{-d}Q)\big).x$ and \begin{align*}&B_G^\mathrm{Mah}\big(\log(\mathcal M_\psi Q^k)+\ref{largetopometricballCONST} \log|E|\big).E\\\subset& B_G^\mathrm{Mah}\big(\log(\mathcal M_\psi Q^k)+\ref{largetopometricballCONST} \log|E|+\log(2^{-d}Q)\big).x\end{align*}
Notice $|E|\leq \log(2^{-d}Q)+1\leq\log Q$ and by (\ref{dioQassum}), $\log\mathcal M_\psi\lesssim_{d}\log\log\log Q$. So $9\log\mathcal M_\psi+\ref{largetopometricballCONST} \log|E|+\log(2^{-d}Q)\leq2\log Q$ when $Q$ is sufficiently large (which is true if $\ref{dioQ'boundMCONST}$ and $\ref{dioQ'boundFCONST}$ are large.) Thus $$9\log(\mathcal M_\psi Q^k)+\ref{largetopometricballCONST} \log|E|+\log(2^{-d}Q)\leq \log(Q^k)+2\log Q\leq (k+2)\log Q$$ and $B_G^\mathrm{Mah}\big(\log(\mathcal M_\psi Q^k)+\ref{largetopometricballCONST} \log|E|\big).E\subset B_G^\mathrm{Mah}\big((k+2)\log Q\big).x$.

On the other hand, $(\log\log |E|)^{-\ref{largetopo'densityCONST} \mathcal F_{\phi(G)}^2}\leq(\log\log\log Q)^{-\ref{dioQ'densityCONST}\mathcal F_{\phi(G)}^2}$ for $\ref{dioQ'densityCONST}=\frac12\ref{largetopo'densityCONST} $ when $\ref{dioQ'boundMCONST}$, $\ref{dioQ'boundFCONST}$ are large (so that $Q\gg 1$) because $\log\log|E|\gtrsim_{d}\log\log\log Q$.

Hence we finally proved $B_G^\mathrm{Mah}\big((k+2)\log Q\big).x$ is $(\log\log\log Q)^{-\ref{dioQ'densityCONST}\mathcal F_{\phi(G)}^2}$ dense when $\ref{dioQ'boundMCONST}$, $\ref{dioQ'boundFCONST}$ and $\ref{dioQ'densityCONST}$ are properly chosen.\end{proof}

\begin{proof}[Proof of Theorem \ref{dioQ}] By Theorem \ref{condGfield}, Proposition \ref{dioQ'} contains the theorem as $d$, $\mathcal M_\psi$ and $\mathcal F_{\phi(G)}$ are all determined by the group $G$.\end{proof}

\section{Appendix: A number-theoretical application}\label{minmaxapp}

\hskip\parindent Before finishing the paper, we give an example of how our results can be applied to number theory. \\

This appendix follows observations by Cerri \cite{C07} and  Bourgain-Lindenstrauss \cite[\S0]{BC07}.\\

Let $K$ be a number field and $\mathcal O_K$ be its ring of integers.  For an integral ideal $I$ and an invertible residue class $\beta\in(\mathcal O_K/I)^*$, define the minimal norm on $\beta$ by \begin{equation}N_I(\beta):=\min_{y\in\beta}|N_K(y)|\end{equation} where $N_K$ is the norm in the number field $K$. Define \begin{equation}L(K,I):=\max_{\beta\in(\mathcal O_K/I)^*}N_I(\beta).\end{equation}

The norm of an ideal $I$ is $N(I):=|\mathcal O_K/I|$. Recall $\forall x\in I$, $N(I)|N_K(x)$.

In \cite{KS99} it was proved that

\begin{theorem}\label{KS}{\rm (Konyagin-Shparlinski)} $L(K,\mathfrak p)=o(N(\mathfrak p))$ for a sequence of prime ideals of asymptotic density 1 as $N(\mathfrak p)\rightarrow\infty$.\end{theorem}

Konyagin and Shparlinski also asked whether the same claim is true for almost all ideals. Recently this question was affirmatively answered by Bourgain and Chang \cite{BC07}. They also improved the result regarding prime ideals by reducing the aymptotic size of the exceptional set. 

\begin{theorem}\label{BC}{\rm (Bourgain-Chang)} $\forall\epsilon>0$, $\exists \delta=\delta(\epsilon)$ such that $\lim_{\epsilon\rightarrow 0}\delta=0$ and \begin{equation}\label{BCbound}L(K,I)\leq N(I)^{1-\epsilon}\end{equation} holds for all ideals outside an exceptional collection $\Omega$ of ideals of asymptotic density at most $\delta$ as $N(I)\rightarrow\infty$.\end{theorem}

The proofs of both Theorem \ref{KS} and Theorem \ref{BC} are based on exponential sums on finite fields.

Bourgain and Lindenstrauss observed that, assuming $K$ is not a CM-field and $\mathrm{rank}(U_K)\geq 2$, a bound which is weaker than (\ref{BCbound}), but still $o(N(I))$, can be achieved by ergodic-theoretical methods for all ideals without exceptional set. Cerri made observations along the same lines in the study of Euclidean minima of number fields.

Such observations can be implemented using Proposition \ref{largetopo'}.

\declareCONST{appthmdensityCONST} \declareCONST{appthmboundDCONST} \declareCONST{appthmboundFCONST}
\begin{theorem}\label{appthm}Suppose $K$ is a non CM-number field of degree $d$ with $\mathrm{rank}(U_K)\geq 2$. Then $L(K,I)=o(N(I))$ holds for all ideals $I$ as $N(I)\rightarrow\infty$. 

More precisely, there are effective constants $\ref{appthmdensityCONST}$, $\ref{appthmboundDCONST}$, $\ref{appthmboundFCONST}$ depending only on $d$ such that \begin{equation}\label{minresidue}\frac{L(K, I)}{N(I)}\leq \big(\log\log\log N(I)\big)^{-\ref{appthmdensityCONST}\mathcal F_{U_K}^2}\end{equation} for all ideals $I$ with \begin{equation}\label{condI}\log\log\log N(I)\geq\big(\ref{appthmboundDCONST}D(K)^{15(d-1)},\max(\mathcal F_{U_K},2)^{\ref{appthmboundFCONST}\mathcal F_{U_K}^2}\big).\end{equation} Here $D(K)$ is the discriminant of $K$, and $\mathcal F_{U_K}$ is the logarithmic Mahler measure bound on some set that generates $U_K$ up to finite index, as defined in Definition \ref{funduni}.\end{theorem}

The group of units $U_K$ acts on $O_K$ by multiplication and preserves all ideals $I$. We know the embedding $\sigma$ in (\ref{KotimesR}) identifies $K\otimes_{\mathbb Q}\mathbb R$ with $\mathbb R^d$ and embedds $\mathcal O_K$ as a full rank lattice. An ideal $I$ is embedded as a sublattice of $\mathcal O_K\subset K\otimes_{\mathbb Q}\mathbb R$.

Recall the $i$-th coordinate $y_i$ in $\mathbb R^d$ of $y\in K$ is $\sigma_i(y)$ for $1\leq i\leq r_1$ and $y_{r_1+j}+\mathrm{i}y_{r_1+r_2+j}=\sigma_{r_1+j}(y)$ for $1\leq j\leq r_1$. Thus 
\begin{equation}\label{radiusnorm}\begin{split}
|N_K(y)|=&\prod_{i=1}^d|\sigma_i(y)|=\prod_{i=1}^{r_1}|y_i|\prod_{j=1}^{r_2}|y_{r_1+j}+\mathrm iy_{r_1+r_2+j}|^2\\
\leq&(\max_{i=1}^d|y_i|)^{r_1}\big(2(\max_{i=1}^d|y_i|)^2\big)^{r_2}=2^{r_2}(\max_{i=1}^d|y_i|)^d.\end{split}\end{equation}

Identify $I$ with the lattice $\sigma(I)$ in $\mathbb R^d$. We are going to show that $I$ doesn't degenerate.  More precisely, for all $I$, there exists an isomorphism $\psi$ between $\mathbb T^d$ and $(K\otimes_{\mathbb Q}\mathbb R)/I$ whose uniformity $\mathcal M_\psi$ is bounded in terms of $D(K)$.

\begin{lemma}Let $C\subset\mathbb R^d$ be a closed convex set that is symmetric with respect to the origin and $\Lambda<R^d$ be a full rank lattice, $d\geq 2$. Denote by $m_1\leq\cdots\leq m_d$ the successive minima of $C$ with respect to $\Lambda$. Then there is a basis $w^1,\cdots w^n$ for $\Lambda$ such that $w^i\in(\frac 32)^{i-1}m_iC, \forall i$.\end{lemma}
\begin{proof} By definition of successive minima, $\Lambda$ contains $d$ linearly independent vectors $v^1,\cdots, v^d$ such that $v^i\in m_iC$. They generate a finite-index sublattice of $\Lambda$.

For $i=1,\cdots, d$, let $L_i$ be the vector subspace spanned by $v^1,\cdots,v^i$ and $\Lambda_i=\Lambda\cap L_i$. Then $\Lambda_i$ is a lattice of rank $i$ and $\Lambda_d=\Lambda$.

We claim that there exists a set of vectors $w^1,\cdots,w^d$ such that for all $i$, $w^i\in\Lambda_i\cap(\frac32)^{i-1}m_iC$ and $w^1,\cdots,w^i$ form a basis of $\Lambda_i$. The proof is by induction. 

When $i=1$, $L_1=\mathbb Rv^1\cong\mathbb R$ and $\Lambda_1$ is a discrete subgroup in it. Thus $\Lambda_1$ is cyclic. Suppose $w^1$ is a generator, then $v^1=nw^1$ for some $n\in\mathbb N$ up to a change of sign. Hence $w^1=\frac1nv^1\in \frac1n m_dC\subset m_dC$, which proves the $i=1$ case.

Suppose the claim is true for $i-1$. Consider the quotient group $\Lambda_i/\Lambda_{i-1}=\Lambda_i/(\Lambda_i\cap L_{i-1})$, which canonically embedds into $L_i/L_{i-1}\cong\mathbb R$ and hence is torsion-free. Moreover, its torsion free rank is $\mathrm{rank}(\Lambda_i)-\mathrm{rank}(\Lambda_{i-1})=1$. Thus $\Lambda_i/\Lambda_{i-1}\cong\mathbb Z$. Take $w\in\Lambda_i$ such that the coset $w+\Lambda_{i-1}$ is a generator. Notice since $v^i\notin L_{i-1}$, $v^i+\Lambda_{i-1}$ is a non-trivial coset. So $\forall v\in\Lambda_i$, up to a change of sign, $\exists n\in\mathbb N$ such that $n(w+\Lambda_{i-1})=v^i+\Lambda_{i-1}$. In other words, $v=nw+\sum_{j=1}^{i-1}l_jw^j$ for some integers $l_1,\cdots,l_{i-1}$ or equivalently, $w^1,\cdots,w^{i-1}$ and $w$ generate $\Lambda_i$. 

In particular, if we take $v=v^i$, then $v^i=nw+\sum_{j=1}^{i-1}l_jw^j$ and $w=\frac1nv^i+\sum_{j=1}^{i-1}\frac{l_j}nw^j$ for some $n\in\mathbb N$ and $l_1,\cdots, l_{i-1}\in Z$.  Decompose $\frac{l_j}n=a_j+k_j$ with $a_j\in[-\frac12,\frac12]$ and $k_j\in\mathbb Z$, then $w=w^i+\sum_{j=1}^{i-1}k_jw^j$ where $w^i=\frac1nv^i+\sum_{j=1}^{i-1}a_jw^j$. Hence $w^1,\cdots,w^{i-1}$ and $w^i$ generate $\Lambda_i$. It suffices to show $w^i\in (\frac32)^im_dC$.

As $v^i\in m_iC$ and $w_j\in(\frac32)^{j-1}m_jC, \forall j<i$, 
\begin{align*}w^i\in&\frac1nm_iC+\sum_{j=1}^{i-1}a_j(\frac32)^{j-1}m_jC\\
\subset&m_iC+\frac12\sum_{j=1}^{i-1}(\frac32)^{j-1}m_iC=\big(1+\frac12\sum_{j=1}^{i-1}(\frac32)^{j-1}\big)m_iC\\
=&(\frac32)^{i-1}m_iC.\end{align*}

Thus the claims is proved. Let $i=d$, the lemma follows.\end{proof}

\begin{lemma}\label{idealunif} There exists an isomorphim $\psi:\mathbb R^d\mapsto\mathbb R^d$ such that $\psi(\mathbb Z^d)=I$, $\mathcal M_\psi\lesssim_d D(K)^\frac{d-1}{2d}$.\end{lemma}

\begin{proof} Define a closed symmetric convex set $C=\{z\in\mathbb R^d\big|\max_{i=1}^d|z_i|\leq 1\}$ in $\mathbb R^d$, whose volume is $2^d$. Suppose $m_1\leq\cdots\leq m_d$ are the succesive minima of $C$ with respect to the lattice $I$. Then by Minkowski's Second Theorem:
\begin{equation}\label{idealsuc1} m_1^{d-1}m_d\leq m_1m_2\cdots m_d\lesssim_d\mathrm{covol}(I)=D(K)^{\frac12}N(I).\end{equation}
On the other hand, by the definition of successive minima, $\exists y\in I\cap m_1C$, $y\neq 0$. Then $|y_i|\leq m_1$, $\forall i$. By (\ref{radiusnorm})
\begin{equation}\label{idealsuc2}|N_K(y)|\lesssim_dm_1^d.\end{equation}

Notice $N(I)|N_K(y)$ and $N_K(y)\neq 0$, thus $N(I)\leq|N_K(y)|$. Compare (\ref{idealsuc1}) and (\ref{idealsuc2}), we see \begin{equation}\label{maxsucminbound} m_d\lesssim_d{m_1}^{-(d-1)}D(K)^{\frac12}N(I)\lesssim_d |N_K(y)|^{-\frac{d-1}d}D(K)^{\frac12}N(I)\leq D(K)^{\frac12}N(I)^\frac1d.\end{equation}
and for all $k$,
\begin{equation}\label{prosucminbound}\begin{split}\prod_{\substack{i=1,\cdots,d\\i\neq k}}m_i\lesssim_d&\ {m_1}^{-1}D(K)^{\frac12}N(I)\lesssim_d |N_K(y)|^{-\frac1d}D(K)^{\frac12}N(I)\\\leq\ &D(K)^{\frac12}N(I)^\frac{d-1}d.\end{split}\end{equation}

The previous lemma says that there is basis $y^1,\cdots, y^d$ of $I$ such that \begin{equation}\label{basis}\max_{j=1}^d|y^i_j|\lesssim_d m_i,\forall i=1,\cdots,d,\end{equation} where $y^i_j$ is the $j$-th coordinate of $y^i$.

Define a $d\times d$ matrix $\psi:=(y^i_j)$, then $\psi(\mathbb Z^n)=I$. In particular \begin{equation}\label{idealscale}|\det\psi|=\mathcal S_I^d=\mathrm{covol}(I)=D(K)^{\frac12}N(I).\end{equation}

From (\ref{maxsucminbound}) and (\ref{basis}) it is easy to see \begin{equation}\label{psibound}\|\psi\|\lesssim_dm_d\lesssim_dD(K)^{\frac12}N(I)^\frac1d.\end{equation} Denote $\psi^{-1}$ by $(a^j_i)$, then each entry $a^i_j=\frac{A^j_i}{|\det\psi|}$ where $A^j_i$ is the determinant of some $(d-1)\times(d-1)$-submatrix in $\psi$. By (\ref{prosucminbound}) and (\ref{basis}) we obtain $|A^i_j|\lesssim_dD(K)^{\frac12}N(I)^\frac{d-1}d$, so 
\begin{equation}\label{psiinversebound}\|\psi^{-1}\|\lesssim_d\max_{i,j}|a^i_j|\lesssim_d\frac{D(K)^{\frac12}N(I)^\frac{d-1}d}{D(K)^{\frac12}N(I)}=N(I)^{-\frac1d}.\end{equation}

Combining (\ref{idealscale}), (\ref{psibound}) and (\ref{psiinversebound}), we obtain 
\begin{align*}\mathcal M_\psi=\ &\max(\|\psi\|\mathcal S_I^{-1},\|\psi^{-1}\|\mathcal S_I)\\
\lesssim_d&\max\Big(D(K)^{\frac12}N(I)^\frac1d\big(D(K)^{\frac12}N(I)\big)^{-\frac1d},N(I)^{-\frac1d}\big(D(K)^{\frac12}N(I)\big)^{\frac1d}\Big)\\
=\ &D(K)^\frac{d-1}{2d}.\end{align*}\end{proof}

Now we are ready to prove the theorem.

\begin{proof}[Proof of Theorem \ref{appthm}] As $U_K=\mathcal O_K^*$ preserves $I$. $U_K$ acts as an automorphism group of the twisted torus $X=(\mathbb K\otimes_{\mathbb Q}\mathbb R)/I$, moreover $U_K$ is volume preserving as $\forall u\in U_K$, $|\det(\times_u)|=|N(u)|=1$. Thus using the isomorphism $\psi$ constructed in Lemma \ref{idealunif} as a conjugation map, the $U_K$ action is conjugate to the $G$-action on $\mathbb T^d$ where $G\cong U_K$ is an abelian subgroup in $SL_d(\mathbb Z)$. In other words, there is a group isomorphism $\phi:G\overset{\sim}\mapsto U_K$ such that $g.x=\psi^{-1}\big(\times_u.\psi(x)\big), \forall x\in\mathbb T^d$.

So $G$ satisfies Condition \ref{condG'} with $\Gamma=I$ as we assumed $K$ is not CM and $\mathrm{rank}(U_K)\geq 2$.

$\beta\in(\mathcal O_K/I)^*\subset \mathcal O_K/I$ is a rational point in the twisted torus $X$. Let us study its orbit $U_K.\beta\subset X$ under the $U_K$ action. Denote the stabilizer in $U_K$ of $\beta$ by $U_K^\beta$. Suppose $u\in U_K^\beta$, then $u.\beta=\beta$. Because $\beta$ is invertible in the ring $\mathcal O_K/I$, this implies $u.(1+I)=(1+I)$ or equivalently $u\in 1+I$. This helps us to estimate the size of $u$. Notice \begin{align*}2^{h^\mathrm{Mah}(u)}=&\prod_{i=1}^d\max(|\sigma_i(u)|,1)\geq\prod_{i=1}^d\frac{|\sigma_i(u)|+1}2\geq 2^{-d}\prod_{i=1}^d(|\sigma_i(u)-1|)\\
\geq& 2^{-d}\prod_{i=1}^d(|\sigma_i(u-1)|)=2^{-d}|N(u-1)|.\end{align*} However as $u-1\in I$, if $u\neq 1$ then $|N(u-1)|\geq N(I)$ thus its logarithmic Mahler measure is bounded from below: $h^\mathrm{Mah}(u)\geq\log(2^{-d}N(I))$.

Let $\mathcal F_{U_K}$ be defined by Definition \ref{funduni}. Then there are $d$ multiplicatively independent elements $u_1,\cdots,u_r\in U_K$ such that $h^\mathrm{Mah}(u_i)\leq \mathcal F_{U_K}$ where $r=\mathrm{rank}(U_K)$. Consider $Z=\{\prod_{i=1}^ru_i^{n_i}\big|0\leq n_i<\frac1d\mathcal F_{U_K}^{-1}\log(2^{-d}N(I))\}$, then it is not hard to see that for any pair $u\neq u'$  inside $Z$,  $h^\mathrm{Mah}(u^{-1}u')<\log(2^{-d}N(I))$ so $u^{-1}u'\notin U_K^\beta$. Therefore $|U_K.\beta|=|U_K/U_K^\beta|\geq |Z|\gtrsim_d\mathcal F_{U_K}^{-r}\big(\log(2^{-d}N(I))\big)^r,$ where the last inequality relies on assumption (\ref{condI}).

Let $x=\psi^{-1}(\beta)\in\mathbb T^d$, then its orbit is $G.x=\psi^{-1}(U_K.\beta)$, thus \begin{equation}\label{orbitsize}|G.x|\gtrsim_d\mathcal F_{U_K}^{-r}\big(\log(2^{-d}N(I))\big)^r.\end{equation}

At this point we may deduce the follow claim:

\declaresubconst{appthmsubconst1}\declaresubconst{appthmsubconst2}
{\bf Claim}{ \it There are effective constants $\ref{appthmsubconst1}(d)$, $\ref{appthmsubconst2}(d)$ and $\ref{appthmboundFCONST}(d)$ such that if 
\begin{equation}\label{condI'}\log\log\log N(I)\geq\big(\ref{appthmsubconst2}\mathcal M_\psi^{30d},\max(\mathcal F_{U_K},2)^{\ref{appthmboundFCONST}\mathcal F_{U_K}^2}\big)\end{equation}
then the orbit $G.x$ is $\big(\log\log\log N(I)\big)^{-\ref{appthmsubconst1}\mathcal F_{U_K}^2}$-dense.}

The claim is proved by applying Proposition \ref{largetopo'} to the set $E=G.x$. The proof is almost identical to \hyperlink{proofdioQ'}{that} of Proposition \ref{dioQ'}, based on (\ref{orbitsize}) instead of on (\ref{startingblock}) (notice $\mathcal F_{\phi(G)}=\mathcal F_{U_K}$); so we are not going to elaborate here.

By Lemma \ref{idealunif}, when $\ref{appthmboundDCONST}$ is large enough then assumption (\ref{condI'}) follows from the condition (\ref{condI}) in the theorem so the claims applies and $U_K.\beta=\psi(G.x)$ is $\rho_0$-dense in $X$ where $\rho_0=\|\psi\|\big(\log\log\log N(I)\big)^{-\ref{appthmsubconst1}\mathcal F_{U_K}^2}$. This implies that $\exists y\in\beta$ and $u\in U_K$ such that $\sigma(uy)\in\mathbb R^d$ is within distance $\rho_0$ from the origin. By (\ref{radiusnorm}) and (\ref{psibound}) we have \begin{equation}\begin{split}N_I(\beta)\leq&|N_K(y)|=|N_K(uy)|\\
\lesssim_d&\ \rho_0^d=\|\psi\|^d\big(\log\log\log N(I)\big)^{-d\ref{appthmsubconst1}\mathcal F_{U_K}^2}
\\\leq\ &D(K)^{\frac d2}N(I)\big(\log\log\log N(I)\big)^{-d\ref{appthmsubconst1}\mathcal F_{U_K}^2}.\end{split}\end{equation}

As $\log\log\log N(I)\gtrsim_d D(K)^{15(d-1)}$ by assumption and $F_{U_K}\gtrsim_d 1$, whenever $\ref{appthmdensityCONST}$, $\ref{appthmboundDCONST}$, $\ref{appthmboundFCONST}$ in the statement of the theorem are made sufficiently large, in the inequality above both $D(K)^{\frac d2}$ and the implied constant can be absorbed into the last factor, which gives $$N_I(\beta)\leq N(I)\big(\log\log\log N(I)\big)^{-\ref{appthmdensityCONST}\mathcal F_{U_K}^2}.$$ (\ref{minresidue}) follows.
\end{proof}

\end{document}